\documentclass[10pt]{article}
\usepackage{geometry}
\usepackage{amsmath,amsthm,amsfonts,amssymb,dsfont}
\usepackage{mathrsfs}
\usepackage{tikz,pstricks,fp}
\usepackage{multirow}
\usepackage{array}

\newtheorem{them}{Theorem}[section]

\newtheorem{prop}[them]{Proposition}
\newtheorem{defi}{Definition}[section]
\newtheorem{rem}{Remark}[section]
\usepackage[colorlinks=true, pdfstartview=FitV, linkcolor=blue,
      citecolor=red, urlcolor=blue, breaklinks=true]{hyperref}
\usepackage{url}
\usepackage{breakurl}
\newcommand{\mcal}[1]{\mathcal{#1}}
\newcommand{\ds}{\displaystyle}

\definecolor{vertf}{rgb}{0,0.55,0.1}
\definecolor{bclairf}{rgb}{0.40,0.65,0.89}
\definecolor{or}{rgb}{0.98,0.6,.1}

\newcommand{\eps}{\varepsilon}
\def\R{{\mathbb R}}
\def\Z{{\mathbb Z}}

\def\xN{{\mathbb N}}

\def\xL{{\rm L}} 
 
\newcommand{\dv}{\partial}
\newcommand{\dd}{{\rm d}}

\DeclareMathOperator*{\argmin}{arg\,min}

\newcommand{\dx}{\,{\rm d}x}

\newcommand{\dt}{\,{\rm d}t}

\newcommand{\ie}{\emph{i.e.}}
\newcommand{\T}{\mathcal{T}}
\newcommand{\PP}{\mathcal{P}}
\newcommand{\M}{\mathcal{M}}
\newcommand{\E}{\mathcal{E}}

\newcommand{\U}{\mathcal{U}}

\newcommand{\tW}{\widetilde{\mathcal{W}}}

\newcommand{\udd}{u^{\sharp}}
\newcommand{\tdd}{\tau^{\sharp}}

\newcommand{\pidd}{\pi^{\sharp}}
\newcommand{\Mdd}{\mathcal{M}^{\sharp}}
\newcommand{\Pdd}{\mathcal{P}^{\sharp}}

\newcommand{\Udd}{U^{\sharp}}
\newcommand{\Me}{\mathcal{M}^{*}}

\newcommand{\A}{(\textbf{A})}
\newcommand{\B}{(\textbf{A})}
\newcommand{\Bu}{(\textbf{A1})}
\newcommand{\Bd}{(\textbf{A2})}
\newcommand{\Bt}{(\textbf{A3})}
\newcommand{\C}{(\textbf{B})}

\newcommand{\Fcal}{\mathcal{F}}
\newcommand{\Ccal}{\mathcal{C}}
\newcommand{\Fbb}{\mathbb{F}}
\newcommand{\Cbb}{\mathbb{C}}

\def\ie\,{\textit{i.e.}\,}

\newcommand{\vect}[1]{\mathbb{#1}}
\def\unde{\lbrace 1,2 \rbrace}

\def\eos{e.o.s.}
\def\Ind{\mathds{1}}

\date{}

\title{\textbf{A Positive and Entropy-Satisfying Finite Volume Scheme for the Baer-Nunziato Model}}

\author{Fr{\'e}d{\'e}ric Coquel$^{1}$, Jean-Marc H\'erard$^{2}$, Khaled Saleh$^{3}$}
\begin{document}

\maketitle

{\centering
${}^{1}$ CMAP, {\'E}cole Polytechnique CNRS, UMR 7641, Route de Saclay, F-91128 Palaiseau Cedex.\\
${}^{2}$ EDF-R\&D, D{\'e}partement MFEE, 6 Quai Watier, F-78401 Chatou Cedex, France.\\
${}^{3}$Universit\'e de Lyon, CNRS UMR 5208, Universit\'e Lyon 1, Institut Camille Jordan, 43 bd 11 novembre 1918; F-69622 Villeurbanne cedex, France.\\
% ${}^{2}$ UPMC Univ Paris 06, UMR 7598, Laboratoire Jacques-Louis
% Lions, F-75005, Paris, France. \\
% ${}^{3}$ CNRS, UMR 7598, Laboratoire Jacques-Louis Lions, F-75005, Paris,
% France. \\
}

\abstract{
We present a relaxation scheme for approximating the entropy dissipating weak solutions of the Baer-Nunziato two-phase flow model. 
This relaxation scheme is straightforwardly obtained as an extension of the relaxation scheme designed in \cite{CHSS} for the isentropic Baer-Nunziato model and consequently inherits its main properties.
To our knowledge, this is the only existing scheme for which the approximated phase fractions, phase densities and phase internal energies are proven to remain positive without any restrictive condition other than a classical fully computable CFL condition. For ideal gas and stiffened gas equations of state, real values of the phasic speeds of sound are also proven to be maintained by the numerical scheme.
It is also the only scheme for which a discrete entropy inequality is proven, under a CFL condition derived from the natural sub-characteristic condition associated with the relaxation approximation. This last property, which ensures the non-linear stability of the numerical method, is satisfied for any admissible equation of state.
We provide a numerical study for the convergence of the approximate solutions towards some exact Riemann solutions. 
The numerical simulations show that the relaxation scheme compares well with two of the most popular existing schemes available for the Baer-Nunziato model, namely Schwendeman-Wahle-Kapila's Godunov-type scheme \cite{SWK} and Toro-Tokareva's HLLC scheme \cite{TT}. The relaxation scheme also shows a higher precision and a lower computational cost (for comparable accuracy) than a standard numerical scheme used in the nuclear industry, namely Rusanov's scheme. Finally, we assess the good behavior of the scheme when approximating vanishing phase solutions. 
}

\normalsize
\medskip
\noindent \textbf{Key-words\,:} Compressible multi-phase flows, Hyperbolic PDEs, Energy-entropy duality, Entropy-satisfying methods, Relaxation techniques, Riemann problem, Riemann solvers, Finite volumes.\\

\noindent \textbf{AMS subject classifications\,:} 76T10, 65M08, 35L60, 35F55.

%%%%%%%%%%%%%%%%%%%%%%%%%%%%%%%%%%%%%%%%%%%%%%%%%%%%%%%%%%%%%%%%%%%%%%%%%%%%%%%
%%%%%%%%%%%%%%%%%%%%%%%%%%%%%%%%%%%%%%%%%%%%%%%%%%%%%%%%%%%%%%%%%%%%%%%%%%%%%%%
%%%% PRESENTATION DU MODEL   %%%%%%%%%%%%%%%%%%%%%%%%%%%%%%%%%%%%%%%%%%%%%%%%%%
%%%%%%%%%%%%%%%%%%%%%%%%%%%%%%%%%%%%%%%%%%%%%%%%%%%%%%%%%%%%%%%%%%%%%%%%%%%%%%%
%%%%%%%%%%%%%%%%%%%%%%%%%%%%%%%%%%%%%%%%%%%%%%%%%%%%%%%%%%%%%%%%%%%%%%%%%%%%%%%

% {\footnotesize
% \tableofcontents
% }

\section{Introduction}

The modeling and numerical simulation of two-phase flows is a relevant
approach for a detailed investigation of some patterns occurring in
water-vapor flows such as those encountered in nuclear power
plants. The targeted applications are the normal operating mode
of pressurized water reactors as well as incidental configurations
such as the Departure from Nucleate Boiling (DNB) \cite{dnb}, the Loss
of Coolant Accident (LOCA) \cite{loca}, the re-flooding phase
following a LOCA or the Reactivity Initiated Accident (RIA) \cite{ria}. In the normal operating mode, the flow in the primary circuit is quasi monophasic as there is \textit{a priori} no
vapor in the fluid. In the incidental configurations however, the
vapor statistical fraction may take values ranging from zero to nearly
one if some areas of the fluid have reached the boiling point. The
modeling as well as the numerical simulation of such phenomena remains
challenging since both models that can handle phase transitions and
robust numerical schemes are needed. 
While in the normal operating mode of pressurized water reactors, homogeneous models assuming thermodynamical equilibrium between the phases are used (in practice, only the liquid phase is present), the simulation of incidental configurations requires more detailed two-phase flow models accounting for thermodynamical disequilibrium. Naturally, as opposed to the numerical approximation of homogeneous models, explicit 
schemes are needed for the simulation of these potentially highly
unsteady phenomena, and one major challenge therefore is the control of the numerical time
step.
In addition, the derived schemes are expected
to ensure important stability properties such as the positivity of the
densities and internal energies, as well as discrete entropy inequalities. In this context, the aim of this work is to design a
robust positivity-preserving and entropy-satisfying scheme for the numerical approximation of two-phase flows with vapor or liquid fractions arbitrarily close to zero.

\medskip
This paper is concerned with the Baer-Nunziato two-phase flow model introduced in \cite{BN}, and studied in various papers \cite{Embid,Aw2,CGHS,GS,KSBMS}. The model consists of two sets of partial differential equations accounting for the evolution of mass, momentum and total energy for each phase, in addition to a transport equation for the phase fraction. The evolution equations of the two phases are coupled through first order non-conservative terms depending on the phase fraction gradient. A major feature of the Baer-Nunziato model is the assumption of two different velocities and two different pressures for the two phases. This approach is not genuinely usual in the nuclear industry where the commonly implemented methods assume the same pressure for the two phases at every time and everywhere in the flow. This latter assumption is justified by the very short time-scale associated with the relaxation of the phasic pressures towards an equilibrium. In the two-fluid two-pressure models (such as Baer-Nunziato's), zero-th order source terms are added in order to account for this pressure relaxation phenomenon as well as a drag force for the relaxation of the phasic velocities towards an equilibrium. Other source terms can also be included in order to account for the relaxation of phasic temperatures and chemical potentials. However, this work is mainly concerned with the convective effects and these zero-th order relaxation terms are not considered in the present paper. We refer to \cite{CGHS} for some modeling choices of these zero-th order relaxation terms and to \cite{HH,LDGH} for their numerical treatment. We also refer to the Conclusion section \ref{sec_conclusion} for some explanation on how the treatment of these terms will affect the numerical method presented in this paper. Various models exist that are related to the Baer-Nunziato model. One may mention various closure laws for the interfacial velocity and pressure \cite{GHS,SA,Glimm} or extensions to multi-phase flows \cite{Herard,Han-Hantke,Muller-Hantke}.

\medskip
Various approaches were considered to approximate the admissible weak solutions of the first order Baer-Nunziato model. One may mention exact Riemann solvers \cite{SWK} or approximate Riemann solvers  \cite{ACCG,TT,ACR}. Let us mention some other schemes grounded on flux of operator splitting techniques \cite{CCKS,CHS,Daude-Galon,KTN,LiuPhd,mathese,TKC,TT2}. Let us also mention the original work of \cite{Abgrall-Dallet,Dallet} where two staggered grids are used (one for the scalar unknowns and the other for the velocities) and where the internal energies are discretized instead of the total energies.

\medskip
The finite volume scheme we describe in the present paper for the convective part of the Baer-Nunziato model relies on two main building blocks. The first block is a relaxation finite volume scheme previously designed in \cite{CHSS} for the isentropic version of the Baer-Nunziato model (the phasic entropies remain constant in both time and space along the process), a scheme which was proved to ensure positive densities and to satisfy discrete energy dissipation inequalities. The second building block is a duality principle between energy and entropy which, according to the second principle of thermodynamics states that, keeping all the other thermodynamic variables constant, the mathematical entropy is a decreasing function of the total energy. This duality principle was already used in previous works to extend schemes designed for the isentropic Euler equations to the full Euler equations (see \cite{CGPIR} and \cite{Bouchut}), and in this work, we apply these techniques to the Baer-Nunziato two-phase flow model. In \cite{CHSS}, a relaxation Riemann solver was designed for the isentropic Baer-Nunziato model. The main properties of this scheme are  firstly, to compute positive densities thanks to an energy dissipation process, secondly to satisfy discrete energy inequalities for each phase, and finally to compute robust approximations of vanishing phase cases where one (or both) of the phase fractions are arbitrarily close to zero in some areas of the flow. The fact that the phasic entropies are simply advected for smooth solutions of the Baer-Nunziato model, combined with the energy-entropy duality principle, actually allows us to use the very same Riemann solver designed in \cite{CHSS}, provided that one supplements it with a correction step which consists in recovering the energy conservation and entropy dissipation for each phase. Concerning the neglected zero-th order source terms, there exist methods that allow their numerical treatment in accordance with the total entropy (the sum of both phasic entropies) dissipation (see \cite{HH,LDGH} and the Conclusion section \ref{sec_conclusion}).

\medskip

Nevertheless, we draw the reader's attention on the fact that the relaxation scheme for the isentropic model, and its extension to the full model described here, are restricted to the simulations of flows with subsonic relative speeds, \ie\, flows for which the difference between the material velocities of the phases is less than the speed of sound in the dominating phase, which would be the liquid phase in the usual operating of a nuclear power plant. For the simulation of nuclear liquid-vapor flows, this is not a restriction, but it would be interesting though to extend the present scheme to sonic and supersonic flows. An interesting work on this subject is done in \cite{Tassadit}.

\medskip
The resulting scheme is proven to preserve positive phase fractions, densities and internal energies, as well as real values of the phasic speeds of sound for stiffened gas and ideal gas \eos. In addition, it is proven to
satisfy a discrete entropy inequality for each phase, under a sub-characteristic
condition (Whitham's condition). To our knowledge, there exists no
other scheme that is proved to satisfy these properties altogether. The relaxation scheme compares well with two of the most popular existing schemes available for the Baer-Nunziato model, namely Schwendeman-Wahle-Kapila's Godunov-type scheme \cite{SWK} and Toro-Tokareva's HLLC scheme \cite{TT}.
In addition, for the same level of refinement, the scheme is shown to be
much more accurate than Rusanov's scheme, and for a given level of approximation error, the relaxation scheme is shown to perform much better in terms of computational cost than this classical scheme.
This is an important result because the approximate Riemann solver designed in \cite{CHSS} and re-used here relies on a fixed-point research for an increasing scalar function defined on the interval $(0,1)$. Hence, the numerical tests assess that no heavy computational costs are due to this fixed-point research.
Actually, comparing with Rusanov's scheme is quite significant since for such stiff configurations as vanishing phase cases, this scheme is commonly used in the industrial context because of its known robustness and simplicity \cite{HH}. Our relaxation scheme is first-order accurate and an interesting further work is the extension to higher orders (see \cite{DHCPT,Franquet,SWK,TT} for examples of high order schemes).

\medskip

The paper is organized as follows. Section \ref{sec_presbn} is devoted to the presentation of the first order Baer-Nunziato model. In Section \ref{sec_approx_relax}, an auxiliary two-phase flow model is introduced, where the phasic entropies are conserved and the phasic total energies are dissipated. We explain how to extend the relaxation scheme designed in \cite{CHSS} to this auxiliary model. For the sake of completeness, the fully detailed Riemann solution is given in Section \ref{constr_sol} of the appendix. In Section \ref{sec_duality}, we give the correction step which relies on the energy-entropy duality principle, and the resulting finite volume scheme for the Baer-Nunziato model is fully described. Finally, Section \ref{numtest} is devoted to the numerical tests. The relaxation finite volume scheme is compared with Schwendeman-Wahle-Kapila's Godunov-type scheme \cite{SWK}, Toro-Tokareva's HLLC scheme \cite{TT} and Rusanov's scheme. In addition to a convergence and CPU cost study, one test case simulates a near-vacuum configuration, and two test-cases assess that the scheme provides a robust numerical treatment of vanishing phase solutions. For the reader who is eager to rapidly implement the numerical scheme, we refer to Section \ref{choixa1a2} of the Appendix,  where the procedure for computing the finite volume numerical fluxes is fully described.

\section{The first order Baer-Nunziato model}
\label{sec_presbn}

% In this section, we show how to extend the finite volume method devised in the isentropic setting to the framework with phasic energies. The proposed extension relies on two key ingredients. The first one is the extension of the fixed point procedure based on two decoupled Euler like systems, respectively for phase 1 and phase 2, which was at the corner stone of the resolution of the Riemann problem for the Suliciu relaxation system in the isentropic setting. Such a strategy has been actually promoted to permit an easy extension to the full setting. The second ingredient is a duality principle in between entropy and energy that allows a trivial extension of Suliciu like approximations from isentropic pressure laws to the framework with energy. The combination of these two ingredients permits in turn a rather immediate extension, since most of the formulae derived in the isentropic setting are virtually kept unchanged.
%\vskip 0.5cm
The Baer-Nunziato model is a non-viscous two-phase flow model formulated in Eulerian coordinates and describing the evolution of the mass, momentum and total energy of each phase. Each phase is indexed by an integer $k\in\unde$, the density of phase $k$ is denoted $\rho_k$, its velocity $u_k$, and its specific total energy $E_k$. At each point $x$ of the space and at each time $t$, the probability of finding phase $k$ is denoted $\alpha_k(x,t)$. We assume the saturation constraint $\alpha_1+\alpha_2=1$. In one-space dimension, the convective part of the model introduce in \cite{BN} reads:

\begin{equation}
\label{BN_ener}
\dv_t \U + \dv_x {\bf \Fcal}(\U) + {\bf \Ccal}(\U)\dv_x \U = 0, \quad x \in \R, t>0,
\end{equation}
where
\begin{equation}
\label{BN_ener_def}
\U=\left [
\begin{matrix}
 \alpha_1 \\
 \alpha_1 \rho_1 \\
 \alpha_2 \rho_2 \\
 \alpha_1 \rho_1 u_1\\
  \alpha_2 \rho_2 u_2\\
 \alpha_1 \rho_1 E_1\\
   \alpha_2 \rho_2 E_2\\
\end{matrix}   \right ], \qquad
{\bf \Fcal}(\U)=\left [
\begin{matrix}
 0 \\
 \alpha_1 \rho_1 u_1 \\
 \alpha_2 \rho_2 u_2 \\
 \alpha_1 \rho_1 u_1^2 + \alpha_1 p_1\\
 \alpha_2 \rho_2 u_2^2 + \alpha_2 p_2\\
 \alpha_1 \rho_1 E_1 u_1 + \alpha_1 p_1 u_1\\
 \alpha_2 \rho_2 E_2 u_2 + \alpha_2 p_2 u_2\\
 \end{matrix}   \right ], \qquad {\bf \Ccal}(\U)\dv_x\U=
\left [
\begin{matrix}
 u_2  \\
 0 \\
 0\\
 -p_1 \\
 +p_1 \\
 -p_1 u_2 \\
 +p_1 u_2
 \end{matrix}   \right ] \dv_x \alpha_1.
\end{equation}

In the complete model, zero-th order source terms are added in order to account for the pressure relaxation phenomenon as well as a drag force for the relaxation of the phasic velocities towards an equilibrium. Other source terms can also be included in order to account for the relaxation of phasic temperatures and chemical potentials. However, this work is mainly concerned with the convective effects and these zero-th order relaxation terms are not considered in the present paper. We refer to the Conclusion section \ref{sec_conclusion} for some explanation on how to treat these relaxation terms without deteriorating the properties of the numerical method presented in this paper.

\medskip
The state vector $\U$ is expected to belong to the natural physical space:
\begin{equation}
\label{Omega_BN_ener}
 \Omega_\U = \left \lbrace \U \in \R^7,\,\alpha_1\in(0,1),\,\alpha_k\rho_k >0, \text{ and } \alpha_k\rho_k(E_k-u_k^2/2) >0 \text{ for } k\in\unde \right\rbrace.
\end{equation}
For each $k\in\unde$, $p_k$ denotes the pressure of phase $k$. Defining $e_k:=E_k-u_k^2/2$ the specific internal energy of phase $k$, the pressure $p_k=p_k(\rho_k,e_k)$ is given by an equation of state (\eos) as a function defined for all positive $\rho_k$ and all positive $e_k$.  
%assumed to satisfy the natural properties $p_k(\rho_k,e_k)>0$, $\dv_{\rho_k} p_k(\rho_k,e_k)>0$ and $\dv_{e_k} p_k(\rho_k,e_k)>0$ for all $\rho_k$, $e_k>0$.

\medskip
We assume that, taken separately, the two phases follow the second principle of thermodynamics so that for each phase $k\in\unde$, there exists a positive integrating factor $T_k(\rho_k,e_k)$ such that the following differential form
\begin{equation}
 \frac{1}{T_k} \left ( \frac{p_k}{\rho_k^{2}} \dd \rho_k - \dd e_k \right ),
\end{equation}
is the exact differential of some \emph{strictly convex} function $s_k(\rho_k,e_k)$, called the (mathematical) entropy of phase $k$.

%\subsection{Main mathematical properties}
The following proposition characterizes the wave structure of this system:
\begin{prop}
\label{propspectrebn}
For all $\U\in\Omega_\U$, the Jacobian matrix ${\bf \Fcal'}(\U) + {\bf \Ccal}(\U)$  admits the following seven eigenvalues:
\begin{equation}
\label{vp}
\begin{array}{c}
  \sigma_1(\U) = \sigma_2(\U)= u_2,\,\, \sigma_3(\U) =u_1\\
  \sigma_4(\U)=u_1 - c_1(\rho_1,e_1), \,\, \sigma_5(\U)= u_1 + c_1(\rho_1,e_1)\\
  \sigma_6(\U) = u_2 - c_2(\rho_2,e_2), \,\,  \sigma_7(\U)= u_2 + c_2(\rho_2,e_2),
\end{array}
\end{equation}
where $c_k(\rho_k,e_k)^2 = \dv_{\rho_k} p_k(\rho_k,e_k)+ p_k(\rho_k,e_k)/\rho_k^2 \, \dv_{e_k} p_k(\rho_k,e_k)$. If $c_k(\rho_k,e_k)^2>0$, then system (\ref{BN_ener}) is weakly hyperbolic on $\Omega_\U$ in the following sense: all the eigenvalues are real and
%is the speed of sound for phase $k$.
the corresponding right eigenvectors are linearly independent if, and only if,
\begin{equation}
\label{hypfail}
 \alpha_1 \neq 0, \quad \alpha_2 \neq 0, \quad |u_1-u_2| \neq c_1(\rho_1,e_1).
\end{equation}
When (\ref{hypfail}) is not satisfied, the system is said to be \textbf{resonant}. The characteristic fields associated with $\sigma_4$, $\sigma_5$, $\sigma_6$ and  $\sigma_7$ are genuinely non-linear, while the characteristic fields associated with $\sigma_{1,2}$ and $\sigma_{3}$ are linearly degenerate. 
\end{prop}

\begin{rem}
\label{remstiffeos}
 The condition $c_k(\rho_k,e_k)^2>0$ is a classical condition that ensures the hyperbolicity for monophasic flows. In general, assuming $\U\in\Omega_\U$ is not sufficient to guarantee that $c_k(\rho_k,e_k)^2>0$. For the stiffened gas \eos~ for instance, where the pressure is given by 
\begin{equation}
 \label{stiffeos}
 p_k(\rho_k,e_k)=(\gamma_k-1)\rho_k e_k- \gamma_k p_{\infty,k}, 
\end{equation}
where $\gamma_k>1$ and $p_{\infty,k}\geq 0$ are two constants, a classical calculation yields $\rho_kc_k(\rho_k,e_k)^2 = \gamma_k(\gamma_k-1)(\rho_k e_k- p_{\infty,k})$. Hence, the hyperbolicity of the system requires a more restrictive condition than simply the positivity of the internal energy which reads : $\rho_k e_k > p_{\infty,k}$. For the stiffened gas \eos, the relaxation scheme proposed in this article will be proven to preserve this condition at the discrete level.
\end{rem}

\begin{rem}
The system is not hyperbolic in the usual sense because when (\ref{hypfail}) is not satisfied, the right eigenvectors do not span the whole space $\R^7$. Two possible phenomena may cause a loss of the strict hyperbolicity: an interaction between the advective field of velocity $u_2$ with one of the acoustic fields of phase 1, and vanishing values of one of the phase fractions $\alpha_k$. In the physical configurations of interest in the present work (such as two-phase flows in nuclear reactors), the flows have strongly subsonic relative velocities, \ie\, a relative Mach number much smaller than one:
\begin{equation}
M=\frac{|u_1-u_2|}{c_1(\rho_1,e_1)} << 1,
\end{equation}
so that resonant configurations corresponding to wave interaction between acoustic fields and the $u_2$-contact discontinuity are unlikely to occur.
In addition, following the definition of the admissible physical space $\Omega_{\U}$, one never has $\alpha_1=0$ or $\alpha_2=0$. However,  $\alpha_k=0$ is to be understood in the sense $\alpha_k\to 0$ since one aim of this work is to construct a robust enough numerical scheme that could handle all the possible values of $\alpha_k,\,k\in \unde$, especially, arbitrarily small values.
\end{rem}

A simple computation shows that the smooth solutions of (\ref{BN_ener}) also obey the following additional conservation laws on the phasic entropies:
\begin{equation}
\label{BN_entrop_eq}
    \dv_t (\alpha_k \rho_k s_k) + \dv_x (\alpha_k \rho_k s_k u_k) = 0, \quad k\in\lbrace 1,2\rbrace.    
\end{equation}

As regards the non-smooth weak solutions of (\ref{BN_ener}), one has to add a so-called \textit{entropy criterion} in order to select the relevant physical solutions. In view of the convexity of the entropy $s_k(\rho_k,e_k)$, an entropy weak solution is a weak solution of (\ref{BN_ener}) which satisfies the following entropy inequalities in the usual weak sense:
\begin{equation}
\label{BN_entrop_kneq}
\dv_t (\alpha_k \rho_k s_k) + \dv_x (\alpha_k \rho_k s_k u_k)  \leq 0, \quad k\in\lbrace 1,2\rbrace. 
\end{equation}
When the solution contains shock waves, inequalities (\ref{BN_entrop_kneq}) are strict in order to account for the physical loss of entropy due to viscous phenomena that are not modeled in system (\ref{BN_ener}).

\medskip
The existence of the phasic entropy conservation laws (\ref{BN_entrop_eq}) and (\ref{BN_entrop_kneq}) will play a central role in the numerical approximation of the solutions of the Baer-Nunziato model. They permit an energy-entropy duality principle which allows a natural extension to the non-isentropic model (\ref{BN_ener}) of the energy-dissipative relaxation scheme designed for the isentropic model in \cite{CHSS}.

\medskip
For the sake of completeness, let us recall the system of PDEs corresponding to the first order isentropic model: for $x \in \R, t>0$:
\begin{equation}
\label{BN_bar}
          \begin{array}{ll}
\dv_t \alpha_1 + u_2 \dv_x \alpha_1 = 0, \\
\dv_t (\alpha_1 \rho_1) + \dv_x (\alpha_1 \rho_1 u_1) = 0,\\
\dv_t (\alpha_1 \rho_1 u_1) + \dv_x (\alpha_1 \rho_1 u_1^2 + \alpha_1 p_1(\tau_1)) - p_1(\tau_1) \dv_x \alpha_1 =0,\\
\dv_t (\alpha_2 \rho_2) + \dv_x (\alpha_2 \rho_2 u_2) = 0,\\
\dv_t (\alpha_2 \rho_2 u_2) + \dv_x (\alpha_2 \rho_2 u_2^2 + \alpha_2 p_2(\tau_2)) - p_1(\tau_1) \dv_x \alpha_2 =0.
\end{array}
\end{equation}
In this case, the phasic pressures are functions solely of the phasic specific volumes $p_k(\tau_k)$, where $\tau_k=\rho_k^{-1}$, and the admissible weak solutions are seen to dissipate the phasic energies according to: 
\begin{equation}
 \label{BN_bar_ener_inec}
 \dv_t  (\alpha_k \rho_k E_k) + \dv_x ( \alpha_k \rho_k E_k u_k + \alpha_k p_k(\tau_k)u_k) -u_2p_1(\tau_1)\dv_x \alpha_k \le 0,  \quad k\in\lbrace 1,2\rbrace, 
\end{equation}
with $E_k=u_k^2/2+e_k(\tau_k)$ where $e_k$ is an anti-derivative of $-p_k$.

\medskip
In a previous work \cite{CHSS}, a relaxation scheme was designed for this isentropic Baer-Nunziato model. This scheme was proved to satisfy desirable properties such as maintaining positive phase fractions and densities, ensuring discrete counterparts of the energy inequalities \eqref{BN_bar_ener_inec}, and finally computing with robustness solutions where some phase fractions are arbitrarily close to zero.

%%%%%%%%%%%%%%%%%%%%%%%%%%%%%%%%%%%%%%%%%%%%%%%%%%%%%%%%%%%%%%%%%%%%%%%%%%%%%%%
%%%%%%%%%%%%%%%%%%%%%%%%%%%%%%%%%%%%%%%%%%%%%%%%%%%%%%%%%%%%%%%%%%%%%%%%%%%%%%%
%%%% MODEL EN ENTROPIES      %%%%%%%%%%%%%%%%%%%%%%%%%%%%%%%%%%%%%%%%%%%%%%%%%%
%%%%%%%%%%%%%%%%%%%%%%%%%%%%%%%%%%%%%%%%%%%%%%%%%%%%%%%%%%%%%%%%%%%%%%%%%%%%%%%
%%%%%%%%%%%%%%%%%%%%%%%%%%%%%%%%%%%%%%%%%%%%%%%%%%%%%%%%%%%%%%%%%%%%%%%%%%%%%%%
\section{Approximating the weak solutions of an auxiliary model}
\label{sec_approx_relax}

As an intermediate step towards the purpose of approximating the entropy weak solutions of (\ref{BN_ener}), let us introduce the following auxiliary system

\begin{equation}\label{BN_entrop}
\dv_t \vect{U} + \dv_x {\bf \Fbb}(\vect{U}) + {\bf \Cbb}(\vect{U})\dv_x \vect{U} = 0, \quad x \in \R, t>0,
\end{equation}
where
\begin{equation}
\label{BN_entrop_def}
\vect{U}=\left [
\begin{matrix}
 \alpha_1 \\
 \alpha_1 \rho_1 \\
 \alpha_2 \rho_2 \\
 \alpha_1 \rho_1 u_1\\
  \alpha_2 \rho_2 u_2\\
 \alpha_1 \rho_1 s_1\\
   \alpha_2 \rho_2 s_2\\
\end{matrix}   \right ], \qquad
{\bf \Fbb}(\vect{U})=\left [
\begin{matrix}
 0 \\
 \alpha_1 \rho_1 u_1 \\
 \alpha_2 \rho_2 u_2 \\
 \alpha_1 \rho_1 u_1^2 + \alpha_1 \PP_1\\
 \alpha_2 \rho_2 u_2^2 + \alpha_2 \PP_2\\
 \alpha_1 \rho_1 s_1 u_1 \\
 \alpha_2 \rho_2 s_2 u_2 \\
 \end{matrix}   \right ], \qquad {\bf \Cbb}(\vect{U})\dv_x\vect{U}=
\left [
\begin{matrix}
 u_2  \\
 0 \\
 0\\
 -\PP_1 \\
 +\PP_1 \\
 0 \\
 0
 \end{matrix}   \right ] \dv_x \alpha_1.
\end{equation}
Compared to the classical Baer-Nunziato model \eqref{BN_ener}, the phasic energy equations have been replaced by the two conservation laws for the phasic entropies. Hence, $\alpha_k\rho_k s_k$ now play the role of independent conservative variables whose evolution is governed according to their own conservative equations. The phasic pressures $\PP_k$ are now seen as functions of the phasic specific volumes $\tau_k=\rho_k^{-1}$ and the phasic entropies $s_k$ so that $\PP_k=\PP_k(\tau_k,s_k)$. These pressure functions are computed as follows: by the second law of thermodynamics, one has:
\[
\frac{\dv s_k}{\dv e_k}(\rho_k,e_k) = -\frac{1}{T_k(\rho_k,e_k)}, \qquad \text{with $T_k(\rho_k,e_k)>0$}.
\]
Hence, the mapping $e \mapsto s_k(\rho_k,e)$ is monotone and thus invertible for all $\rho_k>0$. We denote by $s \mapsto e_k(\tau_k,s)$ the inverse of this mapping, which is a positive function. The dependency on the density $\rho_k$ has been replaced here by a dependency on the specific volume $\tau_k$. The pressure function $\PP_k(\tau_k,s_k)$ is then defined as follows : $\PP_k(\tau_k,s_k)=p_k(\tau_k^{-1},e_k(\tau_k,s_k))$. The phasic total energy is recovered by computing $E_k(u_k,\tau_k,s_k)=u_k^2/2+e_k(\tau_k,s_k)$.

\medskip
The auxiliary state vector $\vect{U}$ is now expected to belong to the physical space:
\begin{equation}
\label{Omega_BN_ent}
 \Omega_{\vect{U}} = \left \lbrace \vect{U} \in \R^7,\,\alpha_1\in(0,1),\,\alpha_k\rho_k >0, \text{ and } \alpha_k\rho_k e_k(\tau_k,s_k) >0 \text{ for } k\in\unde \right\rbrace.
\end{equation}
% Actually, the entropy is usually defined up to a constant, however, we assume that $s_k>0$ is a sufficient condition to uniquely define the internal energy in the following way: for fixed $\tau_k$ and $s_k$, the phasic internal energy $e_k(\tau_k,s_k)$ is defined as the (assumed unique) positive number $e$ satisfying  $p_k(\tau_k^{-1},e)=\PP_k(\tau_k,s_k)$, and the phasic total energy is recovered by computing $E_k(u_k,\tau_k,s_k)=u_k^2/2+e_k(\tau_k,s_k)$. 

\medskip
We have the following property:

\begin{prop}
\label{prop_ent_convex}
The two following equivalent assertions are satisfied :
\begin{enumerate}
 %\item[(i)]  The mapping $(\tau_k,e_k) \mapsto s_k$ is strictly convex.
 \item[(i)] The mapping 
 $$(\alpha_k \rho_k s_k): \quad  \left \lbrace \begin{array}{lll}
      \Omega_\U & \longrightarrow	&\R \\
      \U 	& \longmapsto 		&(\alpha_k \rho_k s_k)(\U)
   \end{array}
   \right.
 $$
 satisfies $\dv_{\alpha_k \rho_k E_k}(\alpha_k \rho_k s_k)(\U)=-1/T_k$ and is convex.
 \item[(ii)]The mapping 
  $$(\alpha_k \rho_k E_k) :\quad \left \lbrace \begin{array}{lll}
      \Omega_{\vect{U}} & \longrightarrow	&\R^+ \\
      \vect{U} 	& \longmapsto 		&(\alpha_k \rho_k E_k)(\vect{U})
   \end{array}
 \right.
 $$
satisfies $\dv_{\alpha_k \rho_k s_k}(\alpha_k \rho_k E_k)(\vect{U})=-T_k$ and is convex.
\end{enumerate}
\end{prop}

\begin{proof}
In order to compute the partial derivative of $(\alpha_k \rho_k s_k)(\U)$ with respect to $\alpha_k \rho_k E_k$, let us calculate the differential of $\alpha_k \rho_k s_k$. Invoking the second law of thermodynamics $T_k\dd s_k=-\dd e_k + p_k\rho_k^{-2}\dd\rho_k$ and the definition $e_k=E_k-u_k^2/2$ of the internal energy, we obtain:
$$
\begin{aligned}
 T_k \dd (\alpha_k\rho_k s_k) & = (\alpha_k \rho_k) T_k \dd s_k + T_k s_k\dd(\alpha_k\rho_k)\\ 
			  & = - (\alpha_k \rho_k)\dd e_k + p_k \rho_k^{-2}(\alpha_k \rho_k) \dd \rho_k + T_k s_k\dd(\alpha_k\rho_k)\\
			  & = - (\alpha_k \rho_k)\dd e_k -  p_k \dd \alpha_k  + (p_k\rho_k^{-1}+T_k s_k)\dd(\alpha_k\rho_k)\\
			  & = - (\alpha_k \rho_k)\dd E_k + (\alpha_k \rho_k u_k) \dd u_k-  p_k \dd \alpha_k  + (p_k\rho_k^{-1}+T_k s_k)\dd(\alpha_k\rho_k)\\
			  & = - (\alpha_k \rho_k)\dd E_k + u_k\dd(\alpha_k \rho_k u_k) -u_k^2\dd( \alpha_k \rho_k)-  p_k \dd \alpha_k  + (p_k\rho_k^{-1}+T_k s_k)\dd(\alpha_k\rho_k)\\
			  & = - \dd (\alpha_k \rho_k E_k) +u_k\dd(\alpha_k\rho_k u_k) -  p_k \dd \alpha_k +   (E_k-u_k^2+ p_k\rho_k^{-1}+T_k s_k)\dd(\alpha_k\rho_k).
\end{aligned}
$$
For $k\in\unde$,  $\dv_{\alpha_k \rho_k E_k}(\alpha_k \rho_k s_k)(\U)$ is the derivative of ${\alpha_k \rho_k s_k}$ with respect to $\alpha_k \rho_k E_k$ when keeping constant the variables $(\alpha_i,\alpha_i\rho_i,\alpha_i\rho_i u_i)$ for $i\in\unde$, and the variable $\alpha_{3-k} \rho_{3-k} E_{3-k}$. Hence, $\dv_{\alpha_k \rho_k E_k}(\alpha_k \rho_k s_k)(\U)=-1/T_k$ and the same computation proves that $\dv_{\alpha_k \rho_k s_k}(\alpha_k \rho_k E_k)(\vect{U})=-T_k$. The proof of the convexity of these two mappings relies on the convexity of the function $s_k(\rho_k,e_k)$. It follows lengthy calculations (see \cite{GR}). We admit this result.
\end{proof}

\medskip

Of course, smooth solutions of \eqref{BN_entrop} also solve \eqref{BN_ener} in the classical sense, which implies that they share the same hyperbolic structure, but entropy weak solutions of \eqref{BN_ener} and (\ref{BN_entrop}) do differ. Indeed, following Proposition \ref{prop_ent_convex},  since $\vect{U}\mapsto(\alpha_k \rho_k E_k)(\vect{U})$ is convex, while the entropy weak solutions of \eqref{BN_ener} are defined so as to dissipate the phasic entropies, it is natural to select weak solutions of the hyperbolic model (\ref{BN_entrop}) according to the differential inequalities:
\begin{equation}
\label{BN_ener_kneq}
\dv_t  (\alpha_k \rho_k E_k) + \dv_x ( \alpha_k \rho_k E_k u_k + \alpha_k\PP_k(\tau_k,s_k)u_k) - \PP_1(\tau_1,s_1) u_2\dv_x \alpha_k \le 0, \quad k\in\unde.
\end{equation}

Observe that for constant initial entropies $s_k(x,0)=s_k^0$, the auxiliary model \eqref{BN_entrop} reduces to the isentropic model \eqref{BN_bar}, with the pressure laws $\tau_k \mapsto \PP(\tau_k,s_k^0)$. Therefore, in the case of constant entropies, extending the relaxation scheme designed in \cite{CHSS} to the auxiliary model \eqref{BN_entrop} is straightforward. Furthermore, even for non constant initial entropies, the derivation of the self-similar solutions for \eqref{BN_entrop} is very close to the isentropic setting because the specific entropies are now just advected by the corresponding phase velocity:
\begin{equation}
\label{advsbnAE}
\dv_t s_k + u_k\dv_x s_k = 0, \quad k\in\unde.
\end{equation} 
For this reason, the Riemann solutions of the auxiliary model \eqref{BN_entrop} are simpler to approximate than those of (\ref{BN_ener}). But again, if smooth solutions of (\ref{BN_ener}) and (\ref{BN_entrop}) are the same, their shock solutions are distinct. Hence, a numerical scheme for advancing in time discrete solutions of the original PDEs (\ref{BN_ener})--(\ref{BN_entrop_kneq}) based on solving a sequence of Riemann solutions for the auxiliary model (\ref{BN_entrop})--(\ref{BN_ener_kneq}) must be given a correction which enforces an energy discretization which is consistent with the original model \eqref{BN_ener}, while ensuring discrete entropy inequalities consistently with \eqref{BN_entrop_kneq}. 
The required correction step is in fact immediate because of the general thermodynamic assumptions made on the complete equation of state. It relies on a duality principle in between energy and entropy, which, according to Proposition \ref{prop_ent_convex}, states that $\alpha_k\rho_k s_k$ is a decreasing function of $\alpha_k\rho_k E_k$.  

\medskip
In the present section, we provide a relaxation scheme for approximating the energy dissipating weak solutions of the auxiliary system \eqref{BN_entrop}. This relaxation scheme is straightforwardly obtained as an extension of the relaxation scheme designed in \cite{CHSS} for the isentropic Baer-Nunziato model and consequently inherits its main properties (positivity of the phase fractions and densities, numerical energy dissipation, robustness for vanishing phase fractions). Again this extension is made possible thanks to the advective equations \eqref{advsbnAE} on the entropies. In Sections \ref{subsec_relax} and \ref{subsec_Riemsol}, we define the relaxation approximation for system \eqref{BN_entrop} and state the existence theorem for the corresponding Riemann solver. This existence result, as it directly follows from the isentropic case, is not proven here. We refer the reader to \cite{CHSS} for the complete proof. In Section \ref{subsec_scheme_relax}, we derive, thanks to this approximate Riemann solver, the numerical scheme for the auxiliary model \eqref{BN_entrop}. Finally, in Section \ref{sec_duality}, we explain how to obtain a positive and entropy-satisfying scheme for the original model \eqref{BN_ener}, thanks to the above mentioned duality principle between energy and entropy.

\subsection{Relaxation approximation for the auxiliary model \eqref{BN_entrop}}
\label{subsec_relax}
System (\ref{BN_entrop}) shares the same hyperbolic structure as system \eqref{BN_ener}. Therefore, it has four genuinely non-linear fields associated with the phasic acoustic waves, which make the construction of an exact Riemann solver very difficult. In the spirit of \cite{SJ}, the relaxation approximation consists in considering an enlarged system involving two additional unknowns $\T_k$, associated with linearizations $\pi_k$ of the phasic pressure laws. This linearization is designed to get a quasilinear enlarged system, shifting the initial non-linearity from the convective part to a stiff relaxation source term. The relaxation approximation is based on the idea that the solutions of the original system are formally recovered as the limit of the solutions of the proposed enlarged system, in the regime of a vanishing relaxation coefficient $\eps>0$. For a general framework on relaxation schemes we refer to \cite{CGPIR,CGS,Bouchut}.

\medskip
We propose to approximate the Riemann problem for (\ref{BN_entrop}) by the self-similar solution of the following Suliciu relaxation model:
\begin{equation}
\label{BNrelax_entrop}
\dv_t \vect{W}^{\eps} + \dv_x \textbf{g}(\vect{W}^{\eps}) + \textbf{d}(\vect{W}^{\eps})\dv_x \vect{W}^{\eps}= \dfrac{1}{\eps} \mathcal{R}(\vect{W}^{\eps}), \quad x\in\R, \, t>0,
\end{equation}
with state vector $\vect{W} = (\alpha_1,\alpha_1 \rho_1,\alpha_2 \rho_2, \alpha_1\rho_1 u_1, \alpha_2 \rho_2 u_2,\alpha_1 \rho_1 s_1,\alpha_2 \rho_2 s_2, \alpha_1 \rho_1 \T_1, \alpha_2 \rho_2 \T_2)^{T}$ and
\begin{equation}
\textbf{g}(\vect{W})=\left [
\begin{matrix}
 0 \\
 \alpha_1 \rho_1 u_1 \\
 \alpha_2 \rho_2 u_2 \\
 \alpha_1 \rho_1 u_1^2 + \alpha_1 \pi_1\\
  \alpha_2 \rho_2 u_2^2 + \alpha_2 \pi_2\\
 \alpha_1 \rho_1 s_1 u_1 \\
\alpha_2 \rho_2 s_2 u_2 \\
 \alpha_1 \rho_1 \T_1 u_1 \\
\alpha_2 \rho_2 \T_2 u_2 
 \end{matrix}   \right ], \ \ \textbf{d}(\vect{W})\dv_x \vect{W}=
\left [
\begin{matrix}
 u_2  \\
 0 \\
 0 \\
 -\pi_1 \\
 \pi_1 \\
 0 \\
 0 \\
 0 \\
 0 
\end{matrix}   \right ]\dv_x \alpha_1, \ \ \mathcal{R}(\vect{W})=\left [
\begin{matrix}
 0 \\
 0  \\
 0 \\
 0  \\
 0 \\
 0 \\
 0 \\
\alpha_1 \rho_1 (\tau_1 -\T_1) \\
\alpha_2 \rho_2 (\tau_2 -\T_2 ) 
 \end{matrix}   \right ].
\end{equation}
For each phase $k$ in $\lbrace 1,2 \rbrace$ the pressure $\pi_k$ is a (partially) linearized pressure $\pi_k(\tau_k,\T_k,s_k)$, the \eos \, of which is defined by:
\begin{equation}
\label{BNpress_relax}
\pi_k(\tau_k,\T_k,s_k)=\PP_k(\T_k,s_k) + a_k^2(\T_k-\tau_k).
\end{equation}

In the formal limit $\eps \to 0$, the additional variable $\T_k$ tends towards the specific volume $\tau_k$, and the linearized pressure law $\pi_k(\tau_k,\T_k,s_k)$ tends towards the original non-linear pressure law $\PP_k(\tau_k,s_k)$, thus recovering system (\ref{BN_entrop}) in the first seven equations of (\ref{BNrelax_entrop}). The solution of (\ref{BNrelax_entrop}) should be parametrized by $\eps$. However, in order to ease the notation, we omit the superscript $^\eps$ in $\vect{W}^{\eps}$. The constants $a_k$ in (\ref{BNpress_relax}) are two positive parameters that must be taken large enough so as to satisfy the following sub-characteristic condition (also called Whitham's condition):
\begin{equation}
\label{whitham}
a_k^2 > -\dv_{\tau_k}\PP_k(\T_k,s_k), \quad \text{$k$ in $\lbrace 1,2\rbrace$},
\end{equation}
for all $\T_k$ and $s_k$ encountered in the solution of (\ref{BNrelax_entrop}). Performing a Chapman-Enskog expansion, we can see that Whitham's condition expresses that system (\ref{BNrelax_entrop}) is a viscous perturbation of system (\ref{BN_entrop}) in the regime of small $\eps$. In addition, there exists two energy functionals $\mathcal{E}_k(u_k,\tau_k,\T_k,s_k)$, which under Whitham's condition, provide an $H$-theorem like result as stated in
\begin{prop}
\label{prop_gibbs}
The smooth solutions of (\ref{BNrelax_entrop}) satisfy the following energy equations
\begin{equation}
 \dv_t(\alpha_k\rho_k\E_k)+\dv_x(\alpha_k\rho_k\E_k u_k+ \alpha_k\pi_k u_k) -u_2\pi_1\dv_x\alpha_k =  \dfrac{1}{\eps}\alpha_k\rho_k \left (a_k^2+\dv_{\tau_k}\PP_k(\T_k,s_k) \right)(\tau_k-\T_k)^2,
\end{equation}
where
\begin{equation}
\label{ener_relax_def}
 \mathcal{E}_k:= \mathcal{E}_k(u_k,\tau_k,\T_k,s_k) = \dfrac{u_k^2}{2} + e_k(\T_k,s_k) + \dfrac{\pi_k^2(\tau_k,\T_k,s_k) - \PP_k^2(\T_k,s_k)}{2a_k^2}, \qquad k \in \lbrace1,2\rbrace.
\end{equation}
Under Whitham's condition \eqref{whitham}, to be met for all the $(\T_k,s_k)$ under consideration, the following Gibbs principles are satisfied for $k \in \unde$:
\begin{equation}
\label{gibbsAE}
 \tau_k = \argmin_{\T_k} \lbrace \mathcal{E}_k(u_k,\tau_k,{\T_k},s_k)\rbrace, \quad \text{and} \quad \mathcal{E}_k(u_k,\tau_k,{\tau_k},s_k)=E_k(u_k,\tau_k,s_k),
 \end{equation}
where $E_k(u_k,\tau_k,s_k)=u_k^2/2+e_k(\tau_k,s_k)$.
\end{prop}

\bigskip
At the numerical level, a fractional step method is commonly used in the implementation of relaxation methods: the first step is a time-advancing step using the solution of the Riemann problem for the convective part of (\ref{BNrelax_entrop}):
\begin{equation}
\label{BNrelax_entrop_conv}
\dv_t \vect{W} + \dv_x \textbf{g}(\vect{W}) + \textbf{d}(\vect{W})\dv_x \vect{W}=0,
\end{equation}
while the second step consists in an instantaneous relaxation towards
the equilibrium system by imposing $\T_k = \tau_k$ in the solution
obtained by the first step. This second step is equivalent to sending
$\eps$ to $0$ instantaneously. As a consequence, we now focus on constructing an exact Riemann solver for the homogeneous convective system \eqref{BNrelax_entrop_conv}. Let us first state the main mathematical properties of the convective system \eqref{BNrelax_entrop_conv}, the solutions of which are sought in the domain of positive densities $\rho_k$ and positive $\T_k$:
\begin{equation}
\Omega_{\vect{W}}=  \Big \lbrace \vect{W} \in \mathbb{R}^7,
0 < \alpha_1 < 1, \ \alpha_k\rho_k > 0, \  \alpha_k \rho_k \T_k > 0, \text{ for } k \in \lbrace1,2\rbrace \Big \rbrace.
\end{equation}

\begin{prop}
System \eqref{BNrelax_entrop_conv} is weakly hyperbolic on $\Omega_{\vect{W}}$ in the following sense. For all $\vect{W}\in\Omega_{\vect{W}}$, the Jacobian matrix $\textbf{g}'(\vect{W}) + \textbf{d}(\vect{W})$ admits the following real eigenvalues
\begin{equation}
\begin{aligned}
 &\sigma_1(\vect{W}) = \sigma_2(\vect{W}) = \sigma_3(\vect{W})= u_2, \,\sigma_4(\vect{W}) = \sigma_5(\vect{W})= u_1,\\
 &\sigma_6(\vect{W}) = u_1 - a_1 \tau_1, \, \sigma_7(\vect{W}) =  u_1 + a_1 \tau_1,\\
 &\sigma_8(\vect{W})=u_2 - a_2 \tau_2,   \,  \sigma_9(\vect{W})= u_2 + a_2 \tau_2.
\end{aligned}
\end{equation}
All the characteristic fields associated with these eigenvalues are linearly degenerate and the corresponding right eigenvectors are linearly independent if, and only if
\begin{equation}
 \alpha_1 \neq 0, \quad \alpha_2 \neq 0, \quad |u_1-u_2| \neq a_1\tau_1.
\end{equation}
The smooth solutions of system \eqref{BNrelax_entrop_conv} satisfy the following phasic energy equations:
\begin{equation}
\label{local_ener_relax}
 \dv_t(\alpha_k\rho_k\E_k)+\dv_x(\alpha_k\rho_k\E_k u_k+ \alpha_k\pi_k u_k) -u_2\pi_1\dv_x\alpha_k =  0.
\end{equation}
Summing over $k\in\unde$, the smooth solutions are seen to conserve the total mixture energy:
\begin{equation}
 \label{mixt_ener_conv}
 \dv_t\left ( \sum_{k=1}^2\alpha_k\rho_k\E_k \right )+\dv_x \left (\sum_{k=1}^2 \left ( \alpha_k\rho_k\E_k u_k+ \alpha_k\pi_k u_k \right) \right )  =  0.
\end{equation}

\end{prop}

\begin{rem}
 In the definition of $\Omega_{\vect{W}}$, the space of admissible states for the solutions of system \eqref{BNrelax_entrop_conv}, no positivity requirement has been given for the phasic specific entropies $s_k$. However, since all the waves are linearly degenerate, the weak solutions are expected to obey a maximum principle on the specific entropies since these two quantities are simply advected:
 \begin{equation}
 \label{ent_adv}
  \dv_t s_k + u_k \dv_x s_k = 0, \qquad  \text{for $k=\unde$}.
 \end{equation}
\end{rem}

\begin{rem}
We look for subsonic solutions which are solutions that remain in the domain of $\Omega_{\vect{W}}$ where $ |u_1-u_2| < a_1\tau_1$. Here again, one never has $\alpha_1=0$ or $\alpha_2=0$. However,  $\alpha_k=0$ is to be understood in the sense $\alpha_k\to 0$.
\end{rem}

\begin{rem}
Since all the characteristic fields of system \eqref{BNrelax_entrop_conv} are linearly degenerate, the mixture energy equation \eqref{mixt_ener_conv} is expected to be satisfied for not only smooth but also weak solutions. However, in the stiff cases of vanishing phases where one of the left or right phase fractions $\alpha_{k,L}$ or $\alpha_{k,R}$ is close to zero, ensuring positive values of the densities requires an extra dissipation of the mixture energy by the computed solution (see the comments on Definition \ref{def_sol} below).
\end{rem}

%%%%%%%%%%%%%%%%%%%%%%%%%%%%%%%%%%%%%%%%%%%%%%%%%%%%%%%%%%%%%%%%%%%%%%%%%%%%%%%
%%%%%%%%%%%%%%%%%%%%%%%%%%%%%%%%%%%%%%%%%%%%%%%%%%%%%%%%%%%%%%%%%%%%%%%%%%%%%%%
%%%% SOLVEUR DE RIEMANN BAROTROPE    %%%%%%%%%%%%%%%%%%%%%%%%%%%%%%%%%%%%%%%%%%
%%%%%%%%%%%%%%%%%%%%%%%%%%%%%%%%%%%%%%%%%%%%%%%%%%%%%%%%%%%%%%%%%%%%%%%%%%%%%%%
%%%%%%%%%%%%%%%%%%%%%%%%%%%%%%%%%%%%%%%%%%%%%%%%%%%%%%%%%%%%%%%%%%%%%%%%%%%%%%%

\subsection{The relaxation Riemann problem}
\label{subsec_Riemsol}
Let $(\vect{W}_L,\vect{W}_R)$ be two elements of $\Omega_{\vect{W}}$.
We now consider the Cauchy problem for \eqref{BNrelax_entrop_conv} with the following Riemann type initial data:
\begin{equation}
\label{relax_CI}
\vect{W}(x,0) =\left \lbrace
	\begin{array}{ll}
	\vect{W}_{L} \quad \textnormal{if} \quad x<0,\\
	\vect{W}_{R} \quad \textnormal{if} \quad x>0.
	\end{array}
\right.
\end{equation}
%The initial states $(\vect{W}_L,\vect{W}_R)$ are two elements of $\Omega_{\vect{W}}$ which are assumed to be \textit{at equilibrium} which means that the additional unknowns $\T_k$ are initialised with the values of the initial specific volumes : $\T_{k,L}=\tau_{k,L}$ and $\T_{k,R}=\tau_{k,R}$ for $k\in\unde$.

% As required by the numerical method (see section \ref{sec_duality}), the initial states $(\vect{W}_L,\vect{W}_R)$ considered here are assumed to be \textit{at equilibrium} which means that $\T_{k,L}=\tau_{k,L}$ and $\T_{k,R}=\tau_{k,R}$ for $k\in\unde$.

Extending the relaxation Riemann solution computed in \cite[Section 3]{CHSS} for the isentropic setting to the present Riemann problem \eqref{BNrelax_entrop_conv}-\eqref{relax_CI} follows from the crucial observation that both the relaxation specific volume $\T_k$ and the specific entropy $s_k$ are advected in the same way by the phasic flow velocity $u_k$:
\begin{equation*}
\left\lbrace
          \begin{array}{ll}
\dv_t \T_k + u_k \dv_x \T_k = 0, \\
\dv_t s_k + u_k \dv_x s_k = 0. \\
\end{array}
        \right.
\end{equation*}
Therefore, for self-similar initial data, the Riemann solution, as soon as it exists, necessarily obeys
\begin{equation}
\label{chap2bn_valuestE}
\T_k(\xi) =
\left\lbrace
          \begin{array}{ll}
\T_{k,L}, \quad \xi < u^*_k \\
\T_{k,R}, \quad  u^*_k < \xi,
\end{array}
        \right.
\quad
s_k(\xi) =
\left\lbrace
          \begin{array}{ll}
s_{k,L}, \quad \xi < u^*_k \\
s_{k,R}, \quad  u^*_k < \xi,
\end{array}
        \right.
% \quad
% \phi(\T_k,s_k)(\xi) =
% \left\lbrace
%           \begin{array}{ll}
% \phi_L, \quad \xi < u^*_k \\
% \phi_R, \quad  u^*_k < \xi,
% \end{array}
%         \right.
\end{equation} 
where $\xi=x/t$ is the self-similar variable, and $u_k^*$ is the effective propagation speed associated with the eigenvalue $u_k$ in the Riemann solution. Furthermore, any given combination of these variables, say $\phi(\T_k,s_k)$, is also advected by $u_k$. Hence, we obtain from (\ref{chap2bn_valuestE}), that the non-linear laws arising from the equation of state evolve in the Riemann solution, virtually the same way as within the isentropic setting. Indeed, the entropies $s_k$ in the relaxation model \eqref{BNrelax_entrop_conv} and in the associated energies \eqref{ener_relax_def}, are systematically involved in non-linear functions already depending on the variable $\T_k$ : namely $\mathcal{P}_k(\T_k,s_k)$ and  $e_k(\T_k,s_k)$. Such functions are solely evaluated on the left and right states in the self-similar initial data and hence always contribute to any given jump conditions in terms of $\PP_k(\T_{k,L},s_{k,L})$, $e_k(\T_{k,L},s_{k,L})$, $\PP_k(\T_{k,R},s_{k,R})$ or $e_k(\T_{k,R},s_{k,R})$. For instance, computing the value of the linearized pressure $\pi_k(\xi)$ at some point $\xi$ of the Riemann fan goes as follows:
$$
\pi_k(\xi)=\pi_k(\tau_k(\xi),\T_k(\xi),s_k(\xi)) = \PP_k(\T_k(\xi),s_k(\xi))+a_k^2(\T_k(\xi)-\tau_k(\xi)),
$$
where $\PP_k(\T_k(\xi),s_k(\xi))=\PP_k(\T_{k,L},s_{k,L})$ if $\xi<u_k^*$ and  $\PP_k(\T_k(\xi),s_k(\xi))=\PP_k(\T_{k,R},s_{k,R})$ otherwise, whereas in the isentropic setting, one would have $\PP_k(\T_{k,L})$ or $\PP_k(\T_{k,R})$. The same observations can be made for the internal energy $e_k(\T_k,s_k)$ when computing the total energy $\mathcal{E}_k(u_k(\xi),\tau_k(\xi),\T_k(\xi),s_k(\xi))$.
Hence, compared to the isentropic case, it is just as if the relaxation unknown $\T_k$ is replaced by a two-dimensional vector $(\T_k,s_k)$. 
\medskip
We formalize these observations in
\begin{prop}
\label{prop_lien_isent}
Let $(\vect{W}_L,\vect{W}_R)\in \Omega_{\vect{W}} \times \Omega_{\vect{W}}$. The Riemann problem \eqref{BNrelax_entrop_conv}-\eqref{relax_CI} admits a solution if, and only if, the isentropic Riemann problem obtained when taking constant initial entropies $s_{k,L}=s_{k,R},\,k\in\unde$ while keeping the other initial data unchanged, admits a solution. When such a solution exists,
the mathematical formulae for defining the phasic quantities $\tau_k,u_k,\pi_k,e_k$ and the void fraction $\alpha_k$ within the Riemann fan read exactly the same as in the isentropic framework \cite[Section 3]{CHSS}, provided the following replacements:
\begin{equation}
\begin{array}{ll}
 \PP_k(\T_{k,L}) \longrightarrow \PP_k(\T_{k,L},s_{k,L}),  & \PP_k(\T_{k,R}) \longrightarrow \PP_k(\T_{k,R},s_{k,R}),  \\
  e_k(\T_{k,L}) \longrightarrow e_k(\T_{k,L},s_{k,L}),     &  e_k(\T_{k,R}) \longrightarrow e_k(\T_{k,R},s_{k,R}). 
\end{array}
\end{equation}
\end{prop}

\medskip
In the following definition, we recall the main features of a solution to the Riemann problem (\ref{BNrelax_entrop_conv})-(\ref{relax_CI}).

\begin{defi}
 \label{def_sol}
Let $(\vect{W}_{L},\vect{W}_{R})$ be two states in $\Omega_{\vect{W}}$. A solution to the Riemann problem (\ref{BNrelax_entrop_conv})-(\ref{relax_CI}) \textbf{with subsonic wave ordering} is a self-similar mapping $\vect{W}(x,t)=\vect{W}_r(x/t;\vect{W}_{L},\vect{W}_{R})$ where the function $\xi \mapsto \vect{W}_r(\xi;\vect{W}_{L},\vect{W}_{R})$ satisfies the following properties: 
\begin{enumerate}
 \item[(i)] $\vect{W}_r(\xi;\vect{W}_{L},\vect{W}_{R})$ is a piecewise constant function, composed of (at most) seven intermediate states belonging to $\Omega_{\vect{W}}$, separated by (at most) six contact discontinuities associated with the eigenvalues $u_1 \pm a_1 \tau_1$, $u_2 \pm a_2 \tau_2$, $u_1$, $u_2$ and such that
\begin{equation}
\label{limits}
\begin{aligned}
 & \xi < \min\limits_{k\in\lbrace1,2\rbrace} \left \lbrace u_{k,L}-a_{k}\tau_{k,L}\right \rbrace \Longrightarrow 
 \vect{W}_r(\xi;\vect{W}_{L},\vect{W}_{R})=\vect{W}_L, \\
 & \xi > \max\limits_{k\in\lbrace1,2\rbrace} \left \lbrace u_{k,R}+a_{k}\tau_{k,R}\right \rbrace \Longrightarrow \vect{W}_r(\xi;\vect{W}_{L},\vect{W}_{R})=\vect{W}_R.
\end{aligned}
\end{equation}
 \item[(ii)] There exists two real numbers $u_2^*$ and $\pi_1^*$ (depending on $(\vect{W}_L,\vect{W}_R)$) such that the function $\vect{W}(x,t)=\vect{W}_r(x/t;\vect{W}_{L},\vect{W}_{R})$ satisfies the following PDEs in the distributional sense: for $k\in\unde$,
\begin{align}
&\dv_t \alpha_k + u_2^* \dv_x \alpha_k = 0, \label{Walpha}\\
&\dv_t (\alpha_k \rho_k) + \dv_x (\alpha_k \rho_k u_k) = 0, \label{Wmass}\\
&\dv_t (\alpha_k \rho_k u_k) + \dv_x (\alpha_k \rho_k u_k^2 + \alpha_k ) - \pi_1^* \dv_x \alpha_k =0, \label{Wmomentum}\\
&\dv_t (\alpha_k \rho_k s_k) + \dv_x (\alpha_k \rho_k s_k u_k )=0, \label{Wentropy}\\
&\dv_t (\alpha_k \rho_k \T_k) + \dv_x (\alpha_k \rho_k \T_k u_k )=0, \label{WT}
\end{align}
where $\dv_x \alpha_k$ identifies with the Dirac measure $\Delta \alpha_k \delta_{x-u_2^*t}$, with $\Delta \alpha_k=\alpha_{k,R}-\alpha_{k,L}$.
 
\item[(iii)] Furthermore, the function $\vect{W}(x,t)=\vect{W}_r(x/t;\vect{W}_{L},\vect{W}_{R})$ also satisfies the following energy equations in the distributional sense:
\begin{align}
&\dv_t (\alpha_2 \rho_2 \E_2) + \dv_x (\alpha_2 \rho_2 \E_2 u_2 + \alpha_2 \pi_2 u_2) - u_2^*\pi_1^* \dv_x \alpha_2 =0,\\
&\dv_t (\alpha_1 \rho_1 \E_1) + \dv_x (\alpha_1 \rho_1 \E_1 u_1 + \alpha_1 \pi_1 u_1) - u_2^*\pi_1^* \dv_x \alpha_1 = -\mathcal{Q}(u_2^*,\vect{W}_L,\vect{W}_R)\delta_{x-u_2^*t}, \label{relax_phase1_diss}
\end{align}
where $\mathcal{Q}(u_2^*,\vect{W}_L,\vect{W}_R)$ is a non-negative number. 

\item[(iv)] The solution has a \textbf{subsonic wave ordering} in the following sense: 
\begin{equation}
 \label{chap2subsol}
 u_{1,L}-a_1\tau_{1,L} < u_2^* < u_{1,R}+a_1\tau_{1,R}.
\end{equation}
\end{enumerate}
\end{defi}

\medskip
Before stating the existence theorem for subsonic solutions proved in \cite[Section 3]{CHSS}, let us introduce some notations built on the initial states $(\vect{W}_L,\vect{W}_R)$ and on the relaxation parameters $(a_1,a_2)$. For $k$ in $\lbrace 1,2 \rbrace$,
\begin{eqnarray}
\label{chap3nota}
\udd_k &:=& \dfrac{1}{2} \left (u_{k,L}+u_{k,R} \right )-\dfrac{1}{2a_k} \left (\pi_k(\tau_{k,R},\T_{k,R},s_{k,R}) - \pi_k(\tau_{k,L},\T_{k,L},s_{k,L}) \right ),
\label{chap3wdd}\\
\pidd_k &:=& \dfrac{1}{2} \left ( \pi_k(\tau_{k,R},\T_{k,R},s_{k,R}) + \pi_k(\tau_{k,L},\T_{k,L},s_{k,L}) \right )- \dfrac{a_k}{2} \left (u_{k,R}- u_{k,L} \right ),
\label{chap3pdd}\\
\tdd_{k,L} &:=& \tau_{k,L} + \dfrac{1}{a_k}(\udd_k - u_{k,L}), 
% = \tau_{k,L}+\dfrac{1}{2a_k}(u_{k,R} - u_{k,L}) - \dfrac{1}{2a_k^2}(p_k(\tau_{k,R}) - p_k(\tau_{k,L})),
\label{chap3tddl}\\
\tdd_{k,R} &:=& \tau_{k,R} - \dfrac{1}{a_k}(\udd_k - u_{k,R}).
%= \tau_{k,R}+\dfrac{1}{2a_k}(u_{k,R} - u_{k,L}) + \dfrac{1}{2a_k^2}(p_k(\tau_{k,R}) - p_k(\tau_{k,L})). \label{chap3tddr}
\end{eqnarray} 
% These quantities only depend on the parameters $(a_1,a_2)$ and on the velocities and densities of the initial states (in particular, they are independent of the initial phase fractions $\alpha_{1,L}$ and $\alpha_{1,R}$). 
We also introduce the following dimensionless number that only depends on the initial phase fractions:
\begin{equation}
\label{chap3notabis}
\Lambda^\alpha:=\dfrac{\alpha_{2,R}-\alpha_{2,L}}{\alpha_{2,R}+\alpha_{2,L}}. % \in (-1,1).
\end{equation}

\medskip
We may now state the existence result for the Riemann problem \eqref{BNrelax_entrop_conv}-\eqref{relax_CI}, which is directly inferred from the existence Theorem for the solutions to the relaxation Riemann problem for the isentropic case designed in \cite[Section 3]{CHSS}. The construction of the self-similar solution is fully provided in Appendix \ref{constr_sol}.
\begin{them}
\label{relax_thm}
Given a pair of admissible initial states $(\vect{W}_{L},\vect{W}_{R}) \in \Omega_{\vect{W}} \times \Omega_{\vect{W}}$, assume that the parameter $a_k$ is such that $\tdd_{k,L} >0$ and $\tdd_{k,R} > 0$ for $k$ in $\lbrace 1,2 \rbrace$. Then there exists solutions with subsonic wave ordering to the Riemann problem \eqref{BNrelax_entrop_conv}-\eqref{relax_CI}, in the sense of Definition \ref{def_sol}, if the following condition holds:
\begin{equation*}
  \B \qquad -a_1\tdd_{1,R} < \dfrac{\udd_1-\udd_2-\frac{1}{a_2}\Lambda^{\alpha}(\pidd_1-\pidd_2)}{1+\frac{a_1}{a_2}|\Lambda^{\alpha}|}< a_1\tdd_{1,L}.
\end{equation*}
\end{them}

\begin{proof}
Following Proposition \ref{prop_lien_isent}, see \cite[Section 3]{CHSS} for a constructive proof and the remarks below. See Appendix \ref{constr_sol} for the expressions of the intermediate states of the solution.
\end{proof}

\paragraph{Some comments on Definition \ref{def_sol} and Theorem \ref{relax_thm}:}
\begin{enumerate}
\item Assumption $\B$ \textbf{can be explicitly tested} in terms of the initial data and the parameters $a_k,\,k\in\lbrace1,2\rbrace$. The quantities $a_1\tdd_{1,L}$ and $a_1\tdd_{1,R}$ can be seen as two sound propagation speeds, while the quantity $(\udd_1-\udd_2-\frac{1}{a_2}\Lambda^{\alpha}(\pidd_1-\pidd_2))/(1+\frac{a_1}{a_2}|\Lambda^{\alpha}|)$, which has the dimension of a velocity, measures the difference between the pressures and kinematic velocities of the two phases, in the initial data. Observe that if the initial data is close to the pressure and velocity equilibrium between the two phases, this quantity is expected to be small compared to $a_1\tdd_{1,L}$ and $a_1\tdd_{1,R}$. This is actually the case when, in addition to the convective system \eqref{BN_ener}, zero-th order source terms are added to the model in order to account for relaxation phenomena that tend to bring the two phases towards thermodynamical ($T_1=T_2$), mechanical ($u_1=u_2$ and $p_1=p_2$) and chemical equilibria (see \cite{CGHS,FH} for the models and \cite{HH,LDGH} for adapted numerical methods).

\item The quantity $u_2^*$ is the propagation velocity of the phase fraction wave. It is computed as the zero of a monotone real function $z\mapsto \Psi_{(\vect{W}_L,\vect{W}_R)}(z)$ on a bounded interval. Assumption $\B$ is a sufficient and necessary condition for this function $\Psi_{(\vect{W}_L,\vect{W}_R)}$ to have a unique zero (\textit{i.e.} a unique number $u_2^*$ satisfying $\Psi_{(\vect{W}_L,\vect{W}_R)}(u_2^*)=0$). Hence, solving this fixed-point problem enables to locate the phase fraction wave by coupling two monophasic systems. Let us stress again on the fact this fixed-point problem is very easy to solve numerically, since it boils down to searching the zero of a strictly monotone function on a bounded interval. We refer the reader to equation \eqref{chap3mel} in Appendix \ref{constr_sol} and to the paper \cite{CHSS} for more details.

The phase fraction derivative $\dv_x \alpha_1$ identifies with the Dirac measure  $\Delta \alpha_1 \delta_0(x-u_2^*t)$. This means that for all open subset $\omega\subset \R$ and for any self-similar function $g(x,t)=g_r(x/t)$, one has :
 \[
 \int_{(\xi,t)\in\omega\times\R^+}  \dv_x \alpha_1(\xi t,t)g(\xi t,t){\rm d} \xi \dt =
  \left \lbrace
  \begin{array}{ll}
      g_r(u_2^*), & \quad \text{if $u_2^*\in\omega$},\\
      0, & \quad \text{otherwise}.
  \end{array}
 \right.
 \]
Item \textit{(ii)} implies that, away from the $u_2$-wave, the system behaves as two independent relaxation systems, one for each phase.

%\item Summing \eqref{Wmomentum} over $k\in\unde$ yields the conservation of the total momentum.

\item \textbf{Positivity of phase 1 densities.} If the ratio $\frac{\alpha_{1,L}}{\alpha_{1,R}}$ is in a neighborhood of $1$, the solution computed thanks to condition $\B$ has positive densities and satisfies the phasic energy equations \eqref{local_ener_relax} in the weak sense. In this case, the solution is said to be energy-preserving and the total mixture energy is also conserved according to the conservative equation \eqref{mixt_ener_conv}. If  $\frac{\alpha_{1,L}}{\alpha_{1,R}}$ is too large, or too small, depending on the wave ordering between $u_2^*$ and $u_1^*$, the solution computed thanks to condition $\B$ may have non-positive densities in phase 1. In such stiff cases, ensuring positive densities for phase $1$ is recovered by allowing a strict dissipation of the phase 1 energy:
\begin{equation}
  \dv_t \left (\alpha_1 \rho_1 \mathcal{E}_1 \right ) + \dv_x \left (  \alpha_1 \rho_1\mathcal{E}_1 + \alpha_1 \pi_1 \right) u_1 - u_2^* \pi_1^*\dv_x \alpha_1 = -\mathcal{Q}(u_2^*,\vect{W}_L,\vect{W}_R) \delta_{x-u_2^*t},
\end{equation}
where $\mathcal{Q}(u_2^*,\vect{W}_L,\vect{W}_R)<0$. The function $\mathcal{Q}(u_2^*,\vect{W}_L,\vect{W}_R)$ is a \textbf{\textit{kinetic relation}} which is chosen large enough so as to impose the positivity of all the phase 1 densities. The value of $\mathcal{Q}(u_2^*,\vect{W}_L,\vect{W}_R)$ parametrizes the whole solution and the choice of $\mathcal{Q}(u_2^*,\vect{W}_L,\vect{W}_R)$ prescribes a unique solution.

\item \textbf{Positivity of phase 2 densities.} Assumption $\B$ allows to compute the value of the wave propagation velocity $u_2^*$ (see comment 2). With this value, one has to \textit{verify} that the following property, which is equivalent to the positivity of the phase $2$ densities, is satisfied:
\begin{equation}
 \C \qquad \udd_2-a_2\tdd_{2,L} < u_2^* < \udd_2+a_2\tdd_{2,R}. \qquad \qquad \qquad  \
\end{equation}
In the numerical applications using this Riemann solver (see Section \ref{numtest}), it will always be possible to ensure property $\C$ by taking a large enough value of the relaxation parameter $a_2$ (see Appendix \ref{choixa1a2}). Note that this condition is a monophasic condition which is not related to the two-fluid modeling. Indeed, the same condition is required when approximating Euler's equations with a similar relaxation scheme.   

\item \textbf{Maximum principle for the entropies.} The phasic entropies $s_k,\, k \in \unde$ satisfy a maximum principle in the solution since they are simply advected by the phasic velocities according to \eqref{ent_adv}.
%Hence, since at the initial time $s_{k,L}>0$ and $s_{k,R}>0$ for $k\in\unde$, one has $s_k(\xi) > 0$ everywhere in the Riemann fan, for $k \in \unde$.
 
\item For the applications envisioned for this work, such as nuclear flows, we are only interested in solutions which have a \textit{subsonic wave ordering}, \ie\, solutions for which the propagation velocity $u_2^*$ of the phase fraction $\alpha_1$ lies in-between the acoustic waves of phase $1$, which is what is required in item \textit{(iv)}. However, the considered solutions are allowed to have \textbf{phasic supersonic speeds} $|u_k|>a_k \tau_k$. Indeed, the subsonic property considered here is related to the \textbf{relative velocity} $u_1-u_2$ with respect to the phase 1 speed of sound $a_1\tau_1$.

\end{enumerate}

\subsection{The relaxation scheme for the auxiliary model}
\label{subsec_scheme_relax}

In this section, the exact Riemann solver $\vect{W}_r(\xi;\vect{W}_{L},\vect{W}_{R})$ for the relaxation system \eqref{BNrelax_entrop_conv} is used to derive an approximate Riemann solver of Harten, Lax and van Leer \cite{HLL} for the simulation of the auxiliary system (\ref{BN_entrop}). The aim is to approximate the admissible weak solution of a Cauchy problem associated with system (\ref{BN_entrop}):
\begin{equation}
\label{aux_cauchy0}
\left \lbrace
          \begin{array}{ll}
		\dv_t \vect{U} + \dv_x {\bf \Fbb}(\vect{U}) + {\bf \Cbb}(\vect{U})\dv_x \vect{U} = 0, & x \in \R, t>0, \\
		 \vect{U}(x,0) =\vect{U}_0(x), &  x \in \R,
	   \end{array}
\right.
\end{equation}
with a discretization which provides discrete counterparts of the energy inequalities \eqref{BN_ener_kneq} satisfied by the exact solutions of the auxiliary model. As expected, the numerical scheme is identical to the relaxation scheme designed in \cite{CHSS} for the isentropic model.
\medskip

We define a time and space discretization as follows: for simplicity in the notations, we assume constant positive time and space steps $\Delta
t$ and $\Delta x$, and we define $\lambda=\frac{\Delta
t}{\Delta x}$. The space is partitioned into cells $\R=\bigcup_{j
\in \Z}  C_j$ where $C_j= [x_{j-\frac{1}{2}},x_{j+\frac{1}{2}}[$ with
$x_{j+\frac{1}{2}}=(j+\frac 12) \Delta x$ for all $j$ in $\Z$. The centers of the cells are denoted $x_{j}=j \Delta x$ for all $j$ in $\Z$. We also introduce the discrete intermediate times $t^{n}=n\Delta t, \ n
\in \xN$. The approximate solution at time $t^{n}$, $x \in \R \mapsto
\vect{U}_{\lambda}(x,t^{n}) \in \Omega$ is a piecewise constant function whose
value on each cell $C_{j}$ is a 
constant value denoted by $\vect{U}_{j}^{n}$. Since $\vect{U}_{\lambda}(x,t^{n})$ is piecewise constant, the exact solution of the following Cauchy problem at time $t^n$ 
\begin{equation}
\label{aux_cauchy}
\left \lbrace
          \begin{array}{ll}
		\dv_t \vect{U} + \dv_x {\bf \Fbb}(\vect{U}) + {\bf \Cbb}(\vect{U})\dv_x \vect{U} = 0, & x \in \R, t>0, \\
		 \vect{U}(x,0) =\vect{U}_{\lambda}(x,t^{n}), &  x \in \R,
	   \end{array}
\right.
\end{equation}
is obtained by juxtaposing the solutions of the Riemann problems set at each cell interface $x_{j+\frac 12}$, provided that these Riemann problems do not interact. The relaxation approximation is an approximate Riemann solver which consists in defining:
\[\ds \vect{U}_{j}^{n+1}:=\frac{1}{\Delta x}\int_{x_{j-\frac 12}}^{x_{j+\frac 12}}\vect{U}_{app}(x,\Delta t)\dx, \qquad j\in\Z,
\]
where $\vect{U}_{app}(x,t)$ is the following approximate solution of \eqref{aux_cauchy}:
\begin{equation}
\label{app}
 \vect{U}_{app}(x,t):=\sum_{j\in\Z} \mathscr{P} \vect{W}_r\left (\frac{x-x_{j+\frac 12}}{t};\mathscr{M}(\vect{U}_{j}^n),\mathscr{M}(\vect{U}_{j+1}^n) \right ) \Ind_{[x_j,x_{j+1}]}(x),
\end{equation}
where $\Ind_{[x_j,x_{j+1}]}$ is the characteristic function of the interval $[x_j,x_{j+1}]$ and the mappings $\mathscr{P}$ and $\mathscr{M}$ are defined by:
\begin{align}
& \mathscr{M}: \left \lbrace \begin{array}{cccl}
		  &  \R^{7} & \longrightarrow & \R^{9}  \\
		  &  (x_k)_{k=1,..,7} & \longmapsto & (x_1,x_2,x_3,x_4,x_5,x_6,x_7,x_1,1-x_1). 
\end{array} \right. \\
& \mathscr{P}: \left \lbrace \begin{array}{cccl}
		  &  \R^{9} & \longrightarrow & \R^{7}  \\
		  &  (x_k)_{k=1,..,9} & \longmapsto & (x_1,x_2,x_3,x_4,x_5,x_6,x_7). 
\end{array} \right.
\end{align}
For a given vector $\vect{U}$, $\vect{W}=\mathscr{M}(\vect{U})$ is the relaxation vector obtained by keeping $\alpha_k$, $\alpha_k\rho_k$, $\alpha_k\rho_k u_k$ and $\alpha_k\rho_k s_k$ unchanged, while setting  $\T_k$ to be equal to $\tau_k$. One says that $\vect{W}\in \Omega_{\vect{W}}$ is at equilibrium if there exists $\vect{U}\in \Omega_{\vect{U}}$ such that $\vect{W}=\mathscr{M}(\vect{U})$. For a given relaxation vector $\vect{W}$,  $\vect{U}=\mathscr{P} \vect{W}$ is the projection of $\vect{W}$ which consists in dropping the relaxation unknowns $\T_k$.

\medskip
In order for the interface Riemann problems not to interact and thus for $\vect{U}_{app}(x,t)$ to be a correct approximate solution of \eqref{aux_cauchy} at time $\Delta t$, the time step $\Delta t$ is chosen small enough so as to satisfy the CFL condition
\begin{equation}
\label{chap2cfl}
\frac{\Delta t}{\Delta x} \ \underset{k\in\lbrace 1,2\rbrace, j \in \Z} {\max} \max \left \lbrace |
(u_{k}-a_{k}\tau_{k})^n_j|,|
(u_{k}+a_{k}\tau_{k})^n_{j+1}|\right \rbrace < \frac{1}{2}.
\end{equation}
Of course, at each interface $x_{j+\frac 12}$, the relaxation Riemann solver $\vect{W}_r\left (\xi;\mathscr{M}(\vect{U}_{j}^n),\mathscr{M}(\vect{U}_{j+1}^n) \right )$ depends on two parameters $(a_k)_{j+\frac 12}^n,k\in\unde$ which must be chosen so as to ensure the conditions stated in the existence Theorem \ref{relax_thm}, and to satisfy some stability properties. Observe that one might take different relaxation parameters $a_k,\,k\in\unde$ for each interface, which amounts to approximating the equilibrium system \eqref{BN_entrop} by a different relaxation approximation at each interface, which is more or less diffusive depending on how large are the local parameters $(a_k)_{j+\frac 12}^n,\,k\in\unde$. Further discussion on the practical computation of these parameters is postponed to Section \ref{choixa1a2} of the Appendices.  

\medskip
Since $\vect{W}_r(\xi;\vect{W}_{L},\vect{W}_{R})$ is the exact solution of the relaxation Riemann problem \eqref{BNrelax_entrop_conv}-\eqref{relax_CI}, the updated unknown $\vect{U}_{j}^{n+1}$ may be computed by a non-conservative finite volume formula as stated in
\begin{prop}
Provided the CFL condition \eqref{chap2cfl} is satisfied, the updated unknown $\vect{U}_{j}^{n+1}$ is given by:
\begin{equation}
\label{chap2fvschemes}
\vect{U}_{j}^{n+1} = \vect{U}_{j}^{n} - \dfrac{\Delta t}{\Delta x}
\left (\mathbf{F}^{-}(\vect{U}_{j}^{n},\vect{U}_{j+1}^{n}) -
\mathbf{F}^{+}(\vect{U}_{j-1}^{n},\vect{U}_{j}^{n})  \right).
\end{equation}
where the numerical fluxes read
\begin{align}
&\mathbf{F}^{-}(\vect{U}_{L},\vect{U}_{R}) = \mathscr{P} 
\textbf{g} \left (\vect{W}_{r} \left (0^-;\mathscr{M}(\vect{U}_L),\mathscr{M}(\vect{U}_R)\right) \right ) +\mathscr{P} 
\mathbf{D}^* \left(\mathscr{M}(\vect{U}_L),\mathscr{M}(\vect{U}_R)\right) \mathbf{1}_{\left \lbrace u_{2}^* < 0 \right \rbrace}, \label{fluxm} \\
&\mathbf{F}^{+}(\vect{U}_{L},\vect{U}_{R}) = \mathscr{P} 
\textbf{g} \left (\vect{W}_{r} \left (0^+;\mathscr{M}(\vect{U}_L),\mathscr{M}(\vect{U}_R)\right) \right ) -\mathscr{P} 
\mathbf{D}^* \left(\mathscr{M}(\vect{U}_L),\mathscr{M}(\vect{U}_R)\right) \mathbf{1}_{\left \lbrace u_{2}^* > 0 \right \rbrace}, \label{fluxp}
\end{align}
with $\mathbf{D}^*(\vect{W}_L,\vect{W}_R):=(\alpha_{1,R}-\alpha_{1,L})(u_2^*(\vect{W}_L,\vect{W}_R),0,0,-\pi_1^*(\vect{W}_L,\vect{W}_R),\pi_1^*(\vect{W}_L,\vect{W}_R),0,0,0,0)^T$. The quantity $\mathbf{1}_{\left \lbrace u_{2}^* < 0 \right \rbrace}$ (resp. $ \mathbf{1}_{\left \lbrace u_{2}^* > 0 \right \rbrace}$) equals one when $ u_{2}^* < 0$ (resp. $u_{2}^* > 0$) and zero otherwise.
\end{prop}

\begin{proof}
Under the CFL condition \eqref{chap2cfl}, the exact solution of \eqref{BNrelax_entrop_conv} with the piecewise constant initial data $\vect{W}(x,0):=\sum_{j\in\Z}  \mathscr{M}(\vect{U}_{j}^n)\Ind_{[x_j,x_{j+1}]}(x)$ is the function:
$$
 \vect{W}(x,t):=\sum_{j\in\Z}  \vect{W}_r\left (\frac{x-x_{j+\frac 12}}{t};\mathscr{M}(\vect{U}_{j}^n),\mathscr{M}(\vect{U}_{j+1}^n) \right ) \Ind_{[x_j,x_{j+1}]}(x),
$$
since the interface Riemann problems do not interact. In addition, under \eqref{chap2cfl}, \eqref{BNrelax_entrop_conv} may be written:
$$
\dv_t \vect{W} + \dv_x \textbf{g}(\vect{W}) + \sum_{j\in\Z} \mathbf{D}^*(\mathscr{M}(\vect{U}_j^n),\mathscr{M}(\vect{U}_{j+1}^n)) \delta_0 \left (x-x_{j+\frac 12}-(u_2^*)_{j+\frac 12}^n t \right ) =0,
$$
where $\mathbf{D}^*(\vect{W}_L,\vect{W}_R)$ is defined in the proposition. Integrating this PDE over $(x_{j-\frac 12},x_{j+\frac 12})\times [0,\Delta t]$ and dividing by $\Delta x$, one obtains:
\begin{equation*}
\begin{aligned}
 \frac{1}{\Delta x} \int_{x_{j-\frac 12}}^{x_{j+\frac 12}}\vect{W}(x,\Delta t)\dx & = \mathscr{M}(\vect{U}_j^n) \\
 &-  \dfrac{\Delta t}{\Delta x} \Big ( \textbf{g} \left (\vect{W}_{r} \left (0^-;\mathscr{M}(\vect{U}_j^n),\mathscr{M}(\vect{U}_{j+1}^n)\right) \right ) - \textbf{g} \left (\vect{W}_{r} \left (0^+;\mathscr{M}(\vect{U}_{j-1}^n),\mathscr{M}(\vect{U}_{j}^n)\right) \right ) \Big )\\
 &- \dfrac{\Delta t}{\Delta x}  \mathbf{D}^*(\mathscr{M}(\vect{U}_j^n),\mathscr{M}(\vect{U}_{j+1}^n)) \mathbf{1}_{\left \lbrace (u_{2}^*)_{j+\frac 12}^n < 0 \right \rbrace} \\
 &- \dfrac{\Delta t}{\Delta x} \mathbf{D}^*(\mathscr{M}(\vect{U}_{j-1}^n),\mathscr{M}(\vect{U}_{j}^n)) \mathbf{1}_{\left \lbrace (u_{2}^*)_{j-\frac 12}^n > 0 \right \rbrace}.
\end{aligned}
\end{equation*}
Applying operator $\mathscr{P}$ to this equation yields \eqref{chap2fvschemes}.
\end{proof}

\medskip
This approximate Riemann solver is proved to ensure a conservative discretization of the partial masses, partial entropies and total mixture momentum and to satisfy important stability properties such as the preservation of the densities positivity, a maximum principle for the entropies, and hence the positivity of the phasic internal energies, and discrete energy inequalities which are discrete counterparts of the energy inequalities (\ref{local_ener_relax}) satisfied by the exact weak solutions of the model. Indeed, we have the following result:
\begin{prop}
\label{chap2lemmastab1}
The numerical scheme \eqref{chap2fvschemes} for the auxiliary model has the following properties:
 
 \medskip
 \noindent $\bullet$ \textbf{Positivity:}  Under the CFL condition (\ref{chap2cfl}), the scheme preserves positive values of the phase fractions, densities and internal energies: for all $n\in\xN$, if $\vect{U}_j^n\in\Omega_{\vect{U}}$ for all $j\in\Z$, then $0 <(\alpha_{k})_j^{n+1} < 1$, $(\alpha_k \rho_k)_j^{n+1} >0$, and $\left (\alpha_k \rho_k e_k(\tau_k,s_k)\right )_j^{n+1} >0$ for $k=1,2$ and all $j\in\Z$, \ie\, $\vect{U}_j^{n+1}\in \Omega_{\vect{U}}$ for all $j\in\Z$. Moreover, if the thermodynamics of phase $k$ follows an ideal gas or a stiffened gas \eos (see \eqref{stiffeos}), then the finite volume scheme \eqref{chap2fvschemes} preserves positive values of the quantity $\rho_k  e_k(\tau_k,s_k)- p_{\infty,k}$: 
\[
 \left ( \rho_k e_k(\tau_k,s_k) \right )_j^{n}- p_{\infty,k} >0, \ \forall j\in\Z \quad \Longrightarrow \quad   \left (\rho_k  e_k(\tau_k,s_k) \right )_j^{n+1}- p_{\infty,k} >0, \ \forall j\in\Z.
\]

\medskip
 
 \noindent $\bullet$ \textbf{Phasic mass conservation:} Denoting $\mathbf{F}^{\pm}_{i}$ the $i^{\text{th}}$ component of vector $\mathbf{F}^{\pm}$, the fluxes for the phasic partial masses $\alpha_k\rho_k$ are conservative:  $\mathbf{F}^{-}_{i}(\vect{U}_L,\vect{U}_R)=\mathbf{F}_i^{+}(\vect{U}_L,\vect{U}_R)$ for $i$ in $\lbrace 2,3 \rbrace$. Hence, denoting $(\alpha_k\rho_k u_k)_{j+\frac 12}^n= \mathbf{F}^{\pm}_{1+k}(\vect{U}_{j}^n,\vect{U}_{j+1}^n)$ for $k=1,2$, one has:
  \begin{equation}
  \label{mass_prop}
 (\alpha_k \rho_k)_j^{n+1} =   (\alpha_k\rho_k)_j^{n}  - \frac{\Delta t}{\Delta x} \left (  (\alpha_k\rho_k u_k)_{j+\frac 12}^{n}- (\alpha_k\rho_k u_k)_{j-\frac 12}^{n}\right).
  \end{equation}
  
  \medskip

 \noindent $\bullet$ \textbf{Phasic entropy conservation.} The fluxes for the phasic entropies $\alpha_k\rho_k s_k$ are conservative:  $\mathbf{F}^{-}_{i}(\vect{U}_L,\vect{U}_R)=\mathbf{F}_i^{+}(\vect{U}_L,\vect{U}_R)$ for $i$ in $\lbrace 6,7 \rbrace$. Hence, denoting $(\alpha_k\rho_k s_k u_k)_{j+\frac 12}^n= \mathbf{F}^{\pm}_{5+k}(\vect{U}_{j}^n,\vect{U}_{j+1}^n)$ for $k=1,2$, one has:
 \begin{equation}
   \label{ent_prop}
  (\alpha_k \rho_k s_k)_j^{n+1}  =   (\alpha_k\rho_k s_k )_j^{n}  - \frac{\Delta t}{\Delta x} \left (  (\alpha_k\rho_k s_k u_k)_{j+\frac 12}^{n}- (\alpha_k\rho_k s_k u_k)_{j-\frac 12}^{n}\right).
 \end{equation}

 \medskip
 
 \noindent $\bullet$ \textbf{Total momentum conservation.} The fluxes for the mixture momentum $\sum_{k=1,2}\alpha_k\rho_k u_k$ are conservative:  $\sum_{k=1,2}\mathbf{F}^{-}_{3+k}(\vect{U}_L,\vect{U}_R)=\sum_{k=1,2}\mathbf{F}^{+}_{3+k}(\vect{U}_L,\vect{U}_R)$. Hence, denoting $(\sum_{k=1,2} \alpha_k\rho_k u_k^2+\alpha_k \pi_k)_{j+\frac 12}^n= \sum_{k=1,2}\mathbf{F}^{\pm}_{3+k}(\vect{U}_{j}^n,\vect{U}_{j+1}^n)$ for $k=1,2$, one has:
 \begin{equation}
   \label{mom_prop}
 \begin{array}{ll}
   \sum_{k=1}^2 (\alpha_k\rho_k u_k)_j^{n+1}= \sum_{k=1}^2 (\alpha_k\rho_k u_k)_j^{n}  & \displaystyle - \frac{\Delta t}{\Delta x} \Big ( \sum_{k=1,2}  \alpha_k\rho_k u_k^2+ \alpha_k\pi_k \Big)_{j+\frac12}^n \\
    & \displaystyle + \frac{\Delta t}{\Delta x}\Big (\sum_{k=1,2}  \alpha_k\rho_k u_k^2+ \alpha_k\pi_k \Big)_{j-\frac12}^n.
 \end{array}
\end{equation}
 
 \medskip
 
 \noindent $\bullet$ \textbf{Discrete energy inequalities.} Assume that the relaxation parameters $(a_k)_{j+\frac 12}^n,\, k=1,2$ satisfy Whitham's condition at each time step and each interface, \ie\, that for all $n\in\xN$, $j\in\Z$, $(a_k)_{j+\frac 12}^n,\, k=1,2$ are large enough so that
\begin{equation}
\label{whithambis}
((a_k)_{j+\frac 12}^n)^2 > -\dv_{\tau_k}\PP_k(\T_k,s_k),
\end{equation}
for all $\T_k$ and $s_k$ in the solution $\xi\mapsto \vect{W}_r\left (\xi;\mathscr{M}(\vect{U}_{j}^n),\mathscr{M}(\vect{U}_{j+1}^n)\right)$. Then, the values $\vect{U}_{j}^n,\, j\in\Z,\,n\in\xN$, computed by the scheme satisfy the following discrete energy inequalities: 
\begin{equation}
  \label{ener_prop}
\begin{array}{ll}
  (\alpha_k\rho_k E_k)(\vect{U}_j^{n+1}) \leq (\alpha_k\rho_k E_k)(\vect{U}_j^{n}) & \displaystyle - \frac{\Delta t}{\Delta x} \left (  (\alpha_k\rho_k \E_k u_k+\alpha_k \pi_k u_k)_{j+\frac 12}^{n}- (\alpha_k\rho_k \E_k u_k+\alpha_k \pi_k u_k)_{j-\frac 12}^{n}\right) \\[2ex] 
   & \displaystyle +\frac{\Delta t}{\Delta x}  \Ind_{\left \lbrace (u_2^*)_{j-\frac 12}^n \geq 0 \right \rbrace }(u_2^* \, \pi_1^*)_{j- \frac 12}^n\left ( (\alpha_k)_{j}^n- (\alpha_k)_{j-1}^n \right ) \\[2ex]
   & \displaystyle+\frac{\Delta t}{\Delta x}  \Ind_{\left \lbrace (u_2^*)_{j+\frac 12}^n \leq 0 \right \rbrace }(u_2^* \, \pi_1^*)_{j+ \frac 12}^n \left ( (\alpha_k)_{j+1}^n- (\alpha_k)_{j}^n \right ),
   \end{array}
\end{equation}
where  for $j\in\Z$, $(\alpha_k\rho_k \E_k u_k+\alpha_k \pi_k u_k)_{j+\frac 12}^{n}=(\alpha_k\rho_k \E_k u_k+\alpha_k \pi_k u_k) \left (\vect{W}_r\left (0^+;\mathscr{M}(\vect{U}_{j}^n),\mathscr{M}(\vect{U}_{j+1}^n) \right ) \right )$ is the right hand side trace of the phasic energy flux evaluated at $x_{j+\frac 12}$.
%\displaystyle \frac{1}{\Delta x}\int_{x_{j-\frac 12}}^{x_{j+\frac 12}}(\alpha_k \rho_k \E_k)(x,\Delta t)dx
%\end{itemize}
\end{prop}

Note that \eqref{mass_prop} and \eqref{ent_prop} are updating formulae for the next time step unknown $\vect{U}_{j}^{n+1}$ whereas \eqref{mom_prop} and the energy inequalities \eqref{ener_prop} are properties satisfied by the values $\vect{U}_{j}^n,\, j\in\Z,\,n\in\xN$, computed by the numerical scheme.

\begin{proof}[Proof of Prop. \ref{chap2lemmastab1}]
The approximate Riemann solver is a Godunov type scheme where $\vect{U}_j^{n+1}$ is the cell-average over $C_j$ of the function $\vect{U}_{app}(x,t)$. Hence, the positivity property on the phase fractions and phase densities is a direct consequence of Theorem \ref{relax_thm}. For this purpose, energy dissipation \eqref{relax_phase1_diss} across the $u_2$-contact discontinuity may be necessary for enforcing this property when the ratio $\frac{\alpha_{1,j}}{\alpha_{1,j+1}}$ (or its inverse) is large for some $j\in\Z$. 

The positivity of the phasic internal energies is more intricate. Under the CFL condition \eqref{chap2cfl}, $(s_k)_j^{n+1}$ is a convex combination of $(s_k)_{j-1}^{n}$, $(s_k)_{j}^{n}$ and $(s_k)_{j+1}^{n}$ since the phasic entropies are advected in the solutions of the local Riemann problems for \eqref{BNrelax_entrop_conv}. Let us define ${\rm jmax}\in \lbrace j-1,j,j+1\rbrace$ such that $(s_k)_{\rm jmax}^{n}=\max\limits_{i=j-1,j,j+1}(s_k)_i^{n}$.
Since $s\mapsto e_k( (\tau_k)_j^{n+1},s)$ is a positive decreasing function by the second law of thermodynamics, one has:
\[
 e_k \left ((\tau_k)_j^{n+1}, (s_k)_j^{n+1} \right ) \geq e_k \left ((\tau_k)_j^{n+1}, (s_k)_{\rm jmax}^{n} \right ) > 0. 
\]
Note that the quantities $e_k \left ((\tau_k)_j^{n+1}, s \right)$ are well defined (and positive) since $(\tau_k)_j^{n+1}>0$.

In a similar way, we prove that, if the thermodynamics of phase $k$ follows a stiffened gas \eos~ according to \eqref{stiffeos}, then the numerical scheme \eqref{chap2fvschemes} preserves positive values of the quantity $\rho_k  e_k(\tau_k,s_k)- p_{\infty,k}$. It follows from the fact that for a stiffened gas \eos, one has:
\[
 \rho_k e_k(\tau_k,s_k) -  p_{\infty,k}= \rho_k^{\gamma_k}\exp\Big( \frac{s_{k,0}-s_k}{{C_v}_k}\Big ),
\]
where $s_{k,0}$ is a constant reference entropy and ${C_v}_k$ is the (constant) heat capacity at constant volume. Hence, the function $s\mapsto \rho_k e_k(\tau_k,s)-  p_{\infty,k}$ is also a positive and decreasing function whenever $\rho_k >0$. This yields:
\[
   \left ( \rho_k e_k(\tau_k,s_k) \right )_j^{n+1}- p_{\infty,k} \ \geq \ (\rho_k)_{j}^{n+1} e_k \Big ((\tau_k)_j^{n+1},  (s_k)_{\rm jmax}^{n}\Big ) - p_{\infty,k}.
\]
The right hand side of this inequality reads:
\begin{equation*}
\begin{aligned}
 (\rho_k)_{j}^{n+1} e_k \Big ((\tau_k)_j^{n+1},  (s_k)_{\rm jmax}^{n}\Big ) - p_{\infty,k} 
 &= ((\rho_k)_{j}^{n+1})^{\gamma_k} \exp\Big( \frac{s_{k,0}-(s_k)_{\rm jmax}^{n}}{{C_v}_k}\Big ) \\
 &= \left ( \frac{(\rho_k)_{j}^{n+1}}{(\rho_k)_{\rm jmax}^{n}}\right )^{\gamma_k} ((\rho_k)_{\rm jmax}^{n})^{\gamma_k} \exp\Big( \frac{s_{k,0}-(s_k)_{\rm jmax}^{n}}{{C_v}_k}\Big ) \\
 &=  \left ( \frac{(\rho_k)_{j}^{n+1}}{(\rho_k)_{\rm jmax}^{n}}\right )^{\gamma_k} \left ( \left ( \rho_k e_k(\tau_k,s_k) \right )_{\rm jmax}^{n} - p_{\infty,k} \right )
\end{aligned}
 \end{equation*}
and is therefore positive since at time $t^n$, we have $ \left ( \rho_k e_k(\tau_k,s_k) \right )_j^{n}- p_{\infty,k} >0$ for all $j\in\Z$.

The proof of \eqref{mass_prop}, \eqref{ent_prop} and \eqref{mom_prop} involves no particular difficulties. It is a direct consequence of equations \eqref{Wmass}, \eqref{Wmomentum} and \eqref{Wentropy} satisfied by the relaxation Riemann solutions at each interface.

Let us prove the discrete energy inequalities \eqref{ener_prop} satisfied by the scheme under Whitham's condition \eqref{whithambis}. Assuming the CFL condition \eqref{chap2cfl}, the solution of \eqref{BNrelax_entrop_conv} over $[x_{j-\frac 12},x_{j+\frac 12}]\times[t^n,t^{n+1}]$ is the function
\begin{multline}
\vect{W}(x,t):=\vect{W}_r\left (\frac{x-x_{j-\frac 12}}{t-t^n};\mathscr{M}(\vect{U}_{j-1}^n),\mathscr{M}(\vect{U}_{j}^n) \right ) \Ind_{[x_{j-\frac 12},x_{j}]}(x) \\ +\vect{W}_r\left (\frac{x-x_{j+\frac 12}}{t-t^n};\mathscr{M}(\vect{U}_{j}^n),\mathscr{M}(\vect{U}_{j+1}^n) \right ) \Ind_{[x_{j},x_{j+\frac 12}]}(x).
\end{multline}
According to Theorem \ref{relax_thm}, this function satisfies the phase 1 energy equation:
\begin{multline}
\label{local_ener_relax_1}
\dv_t (\alpha_1 \rho_1 \E_1) + \dv_x (\alpha_1 \rho_1 \E_1 u_1 + \alpha_1 \pi_1 u_1) - u_2^*\pi_1^* \dv_x \alpha_1 = \\
- \mathcal{Q}_{j-\frac 12}^n \delta_{0}\left (x-x_{j-\frac12}-(u_2^*)_{j-\frac 12}^n(t-t^n) \right ) 
- \mathcal{Q}_{j+\frac 12}^n\delta_{0} \left (x-x_{j+\frac12}-(u_2^*)_{j+\frac 12}^n(t-t^n) \right ),
 \end{multline}
where for $i\in\Z$, we have denoted $\mathcal{Q}_{i-\frac 12}^n=\mathcal{Q}\left ((u_2^*)_{i-\frac 12}^n,\mathscr{M}(\vect{U}_{i-1}^n),\mathscr{M}(\vect{U}_{i}^n) \right )$.
Integrating this equation over $]x_{j-\frac 12},x_{j+\frac 12}[\times[t^n,t^{n+1}]$ and dividing by $\Delta x$ yields:

\begin{equation}
\label{ener1_proof}
 \begin{array}{ll}
  \displaystyle \frac{1}{\Delta x}\int_{x_{j-\frac 12}}^{x_{j+\frac 12}}(\alpha_1\rho_1 \E_1)(\vect{W}(x,t^{n+1}))\dx 
  &\leq (\alpha_1\rho_1 \E_1)(\mathscr{M}(\vect{U}_j^{n})) \\
  & \displaystyle - \frac{\Delta t}{\Delta x}  (\alpha_1\rho_1 \E_1 u_1+\alpha_1 \pi_1 u_1) \left ( \vect{W}_r\left (0^-;\mathscr{M}(\vect{U}_{j}^n),\mathscr{M}(\vect{U}_{j+1}^n) \right ) \right )\\[2ex]
  &\displaystyle  + \frac{\Delta t}{\Delta x} (\alpha_1\rho_1 \E_1 u_1+\alpha_1 \pi_1 u_1) \left (\vect{W}_r\left (0^+;\mathscr{M}(\vect{U}_{j-1}^n),\mathscr{M}(\vect{U}_{j}^n) \right ) \right ) \\[2ex]
  & \displaystyle +\frac{\Delta t}{\Delta x}  \Ind_{\left \lbrace (u_2^*)_{j-\frac 12}^n \geq 0 \right \rbrace }(u_2^* \, \pi_1^*)_{j- \frac 12}^n\left ( (\alpha_1)_{j}^n- (\alpha_1)_{j-1}^n \right ) \\[2ex]
  & \displaystyle+\frac{\Delta t}{\Delta x}  \Ind_{\left \lbrace (u_2^*)_{j+\frac 12}^n \leq  0 \right \rbrace }(u_2^* \, \pi_1^*)_{j+ \frac 12}^n \left ( (\alpha_1)_{j+1}^n- (\alpha_1)_{j}^n \right ),
  %& \displaystyle -\frac{\Delta t}{\Delta x}  \Ind_{\left \lbrace (u_2^*)_{j-\frac 12}^n > 0 \right \rbrace }\mathcal{Q}_{j- \frac 12}^n \\[2ex]
  %& \displaystyle -\frac{\Delta t}{\Delta x}  \Ind_{\left \lbrace (u_2^*)_{j+\frac 12}^n <  0 \right \rbrace }\mathcal{Q}_{j+ \frac 12}^n .
   \end{array}
\end{equation}
because $\mathcal{Q}_{j- \frac 12}^n\geq0$ and $\mathcal{Q}_{j+ \frac 12}^n \geq 0$. Since the initial data is at equilibrium: $\vect{W}(x,t^{n})=\mathscr{M}(\vect{U}_j^{n})$ for all $x\in C_j$ ( \ie\,  $(\T_1)_j^n$ is set to be equal to $(\tau_1)_j^n$) one has $(\alpha_1\rho_1 \E_1)(\mathscr{M}(\vect{U}_j^{n}))=(\alpha_1\rho_1 E_1)(\vect{U}_j^{n})$ according to Proposition \ref{prop_gibbs}. Applying the Rankine-Hugoniot jump relation to \eqref{local_ener_relax_1} across the line $\lbrace (x,t),x=x_{j+\frac 12}, \, t>0\rbrace$, yields:
\begin{multline*}
(\alpha_1\rho_1 \E_1 u_1+\alpha_1 \pi_1 u_1) \left ( \vect{W}_r\left (0^-;\mathscr{M}(\vect{U}_{j}^n),\mathscr{M}(\vect{U}_{j+1}^n) \right ) \right ) \\
= (\alpha_1\rho_1 \E_1 u_1+\alpha_1 \pi_1 u_1) \left ( \vect{W}_r\left (0^+;\mathscr{M}(\vect{U}_{j}^n),\mathscr{M}(\vect{U}_{j+1}^n) \right ) \right ) + \mathcal{Q}_{j+ \frac 12}^n \Ind_{\left \lbrace (u_2^*)_{j+\frac 12}^n =  0 \right \rbrace }.
 \end{multline*}
Hence, since $\mathcal{Q}_{j+ \frac 12}^n\geq0$, for the interface $x_{j+\frac 12}$, taking the trace of $(\alpha_1\rho_1 \E_1 u_1+\alpha_1 \pi_1 u_1)$ at $0^+$ instead of $0^-$ in \eqref{ener1_proof} only improves the inequality. Furthermore, assuming that the parameter $a_1$ satisfies Whitham's condition \eqref{whithambis}, the Gibbs principle stated in \eqref{gibbsAE} holds true so that:
$$
\frac{1}{\Delta x}\int_{x_{j-\frac 12}}^{x_{j+\frac 12}}(\alpha_1\rho_1 E_1)(\vect{U}_{app}(x,t^{n+1}))\dx \leq \frac{1}{\Delta x}\int_{x_{j-\frac 12}}^{x_{j+\frac 12}}(\alpha_1\rho_1 \E_1)(\vect{W}(x,t^{n+1}))\dx.
$$
Invoking the convexity of the mapping $\vect{U}\mapsto(\alpha_1\rho_1 E_1)(\vect{U})$ (see Prop. \ref{prop_ent_convex}), Jensen's inequality implies that
$$
(\alpha_1\rho_1 E_1)(\vect{U}_j^{n+1}) \leq \frac{1}{\Delta x}\int_{x_{j-\frac 12}}^{x_{j+\frac 12}}(\alpha_1\rho_1 E_1)(\vect{U}_{app}(x,t^{n+1}))\dx,
$$
which yields the desired discrete energy inequality for phase 1. The proof of the discrete energy inequality for phase 2 follows similar steps. 
\end{proof}

\section{A positive and entropy-satisfying scheme for the first order Baer-Nunziato model}
\label{sec_duality}
In the previous section we have designed a numerical scheme for an auxiliary two-phase flow model where the exact solutions conserve the phasic entropies while the phasic energies are dissipated by shock solutions. The scheme has been proven to satisfy discrete counterparts of these features while ensuring the positivity of the relevant quantities. 
\medskip

In the present section, we describe the correction to be given to the auxiliary scheme in order to conserve the phasic energies while dissipating the phasic entropies. We end up with a numerical scheme which is consistent with the entropy weak solutions of any Cauchy problem associated with the original Baer-Nunziato model (\ref{BN_ener}):
\begin{equation}
\label{BN_cauchy0}
\left \lbrace
          \begin{array}{ll}
		\dv_t \U + \dv_x {\bf \Fcal}(\U) + {\bf \Ccal}(\U)\dv_x \U = 0, & x \in \R, t>0, \\
		 \U(x,0) =\U_0(x), &  x \in \R.
	   \end{array}
\right.
\end{equation}

\medskip

We keep the same time and space discretization as described in Section \ref{subsec_scheme_relax}. The approximate solution at time $t^{n}$, $x \in \R \mapsto \U_{\lambda}(x,t^{n}) \in \Omega$ is a piecewise constant function whose value on each cell $C_{j}$ is a  constant value denoted by $\U_{j}^{n}$. The updated value $\U_j^{n+1}$ is computed through a two-step algorithm described hereunder:

\subsection{A fractional step algorithm}

\paragraph{$\bullet$ Step 1: updating the auxiliary unknown.} Given 
$$\U_j^{n}=\left ( (\alpha_1)_j^n, (\alpha_1 \rho_1)_j^n, (\alpha_2 \rho_2)_j^n, (\alpha_1 \rho_1 u_1)_j^n, (\alpha_2 \rho_2 u_2)_j^n, (\alpha_1 \rho_1 E_1)_j^n, (\alpha_2 \rho_2 E_2)_j^n \right )^T,
$$
we begin with setting the auxiliary unknown $\vect{U}_j^n$ as follows:
$$\vect{U}_j^n=\left ( (\alpha_1)_j^n, (\alpha_1 \rho_1)_j^n, (\alpha_2 \rho_2)_j^n, (\alpha_1 \rho_1 u_1)_j^n, (\alpha_2 \rho_2 u_2)_j^n, (\alpha_1 \rho_1 s_1)(\U_j^n), (\alpha_2 \rho_2 s_2)(\U_j^n) \right )^T,
$$
where $(\alpha_k \rho_k s_k)(\U_j^n)$ is the partial entropy of phase $k$, computed from $\U_j^n$, knowing the density $\rho_k$, the total energy $E_k$ and the kinetic energy $u_k^2/2$:
$$
(\alpha_k \rho_k s_k)(\U_j^n) := \alpha_k \rho_k s_k \left ( (\rho_k)_j^n,(E_k-u_k^2/2)_j^n\right ).
$$
Observe that, with this definition of $\vect{U}_j^{n}$, one has $ (\alpha_k\rho_k E_k)(\vect{U}_j^{n})=(\alpha_k \rho_k E_k)_j^n$.

\medskip
We then compute $\vect{U}_j^{n+1,-}$ by applying the relaxation scheme designed for the auxiliary model:
\begin{equation}
\label{fvs}
\vect{U}_{j}^{n+1,-} = \vect{U}_{j}^{n} - \dfrac{\Delta t}{\Delta x}
\left (\mathbf{F}^{-}(\vect{U}_{j}^{n},\vect{U}_{j+1}^{n}) -
\mathbf{F}^{+}(\vect{U}_{j-1}^{n},\vect{U}_{j}^{n})  \right).
\end{equation}
According to Proposition \ref{chap2lemmastab1} the phasic energies are dissipated at the discrete level following:
\begin{equation}
  \label{ener_prop_bis}
\begin{array}{ll}
  (\alpha_k\rho_k E_k)(\vect{U}_j^{n+1,-}) \leq (\alpha_k\rho_k E_k)_j^n & \displaystyle - \frac{\Delta t}{\Delta x} \left (  (\alpha_k\rho_k \E_k u_k+\alpha_k \pi_k u_k)_{j+\frac 12}^{n}- (\alpha_k\rho_k \E_k u_k+\alpha_k \pi_k u_k)_{j-\frac 12}^{n}\right) \\[2ex] 
   & \displaystyle +\frac{\Delta t}{\Delta x}  \Ind_{\left \lbrace (u_2^*)_{j-\frac 12}^n \geq 0 \right \rbrace }(u_2^* \, \pi_1^*)_{j- \frac 12}^n\left ( (\alpha_k)_{j}^n- (\alpha_k)_{j-1}^n \right ) \\[2ex]
   & \displaystyle+\frac{\Delta t}{\Delta x}  \Ind_{\left \lbrace (u_2^*)_{j+\frac 12}^n \leq 0 \right \rbrace }(u_2^* \, \pi_1^*)_{j+ \frac 12}^n \left ( (\alpha_k)_{j+1}^n- (\alpha_k)_{j}^n \right ).
   \end{array}
\end{equation}

\paragraph{$\bullet$ Step 2: Exchanging energy and entropy.} This final step is a correction step which aims at enforcing conservative updates for the energies of the original unknown $\U_j^{n+1}$. It simply consists in keeping unchanged the updates of the phase fractions, partial masses and momentum:
\begin{equation}
\label{autre_update}
 (\alpha_1)_j^{n+1}:= (\alpha_1)_j^{n+1,-}, \quad (\alpha_k \rho_k)_j^{n+1}:= (\alpha_k \rho_k)_j^{n+1,-}, \quad (\alpha_k \rho_k u_k)_j^{n+1}:= (\alpha_k \rho_k u_k)_j^{n+1,-}, \quad k=1,2,
\end{equation}
while enforcing energy conservation by defining the energies updates as:
\begin{equation}
  \label{ener_update}
\begin{array}{ll}
  (\alpha_k\rho_k E_k)_j^{n+1} := (\alpha_k\rho_k E_k)_j^n & \displaystyle - \frac{\Delta t}{\Delta x} \left (  (\alpha_k\rho_k \E_k u_k+\alpha_k \pi_k u_k)_{j+\frac 12}^{n}- (\alpha_k\rho_k \E_k u_k+\alpha_k \pi_k u_k)_{j-\frac 12}^{n}\right) \\[2ex] 
   & \displaystyle +\frac{\Delta t}{\Delta x}  \Ind_{\left \lbrace (u_2^*)_{j-\frac 12}^n \geq 0 \right \rbrace }(u_2^* \, \pi_1^*)_{j- \frac 12}^n\left ( (\alpha_k)_{j}^n- (\alpha_k)_{j-1}^n \right ) \\[2ex]
   & \displaystyle+\frac{\Delta t}{\Delta x}  \Ind_{\left \lbrace (u_2^*)_{j+\frac 12}^n \leq 0 \right \rbrace }(u_2^* \, \pi_1^*)_{j+ \frac 12}^n \left ( (\alpha_k)_{j+1}^n- (\alpha_k)_{j}^n \right ).
   \end{array}
\end{equation}

\subsection{Finite volume formulation of the scheme}
In practice, in the implementation, when performing the first step of the method, \ie\, when applying the relaxation scheme to the auxiliary variable $\vect{U}$, one does not update the last two variables which are the phasic entropies. Indeed, computing the phasic entropies $(\alpha_k \rho_k s_k)_j^{n+1,-}$ is not needed for the update of the phasic energies which is performed in the second step. Therefore, the two step algorithm described in the previous section can be reformulated as a classical non-conservative finite volume scheme. Indeed, we have the following result:
\begin{prop}
The two step algorithm described in equations \eqref{fvs}-\eqref{autre_update}-\eqref{ener_update} is equivalent to the following non-conservative finite volume scheme:
\begin{equation}
\label{fvs_2}
\U_{j}^{n+1} = \U_{j}^{n} - \dfrac{\Delta t}{\Delta x}
\left (\mathcal{F}^{-}(\U_{j}^{n},\U_{j+1}^{n}) -
\mathcal{F}^{+}(\U_{j-1}^{n},\U_{j}^{n})  \right).
\end{equation}
where the first five components of $\mathcal{F}^{\pm}(\U_{L},\U_{R})$ coincide with the first five components of $\mathbf{F}^{\pm}(\vect{U}_{L},\vect{U}_{R})$ with $\vect{U}$ computed from $\U$ by imposing $\alpha_k\rho_k s_k = (\alpha_k\rho_k s_k)(\U)$. The last two components of $\mathcal{F}^{\pm}(\U_{L},\U_{R})$ are given for $k\in\unde$ by:
\begin{equation*}
\begin{aligned}
&\mathcal{F}_{5+k}^{-}(\U_{L},\U_{R}) = (\alpha_k\rho_k \E_k u_k+\alpha_k \pi_k u_k) \left ( \vect{W}_r\left (0^+;\mathscr{M}(\vect{U}_{L}),\mathscr{M}(\vect{U}_{R}) \right ) \right ) + \mathbf{1}_{\left \lbrace u_2^* \leq 0 \right \rbrace }(u_2^* \, \pi_1^*) \left ( (\alpha_k)_{R}- (\alpha_k)_{L} \right ), \\
&\mathcal{F}_{5+k}^{+}(\U_{L},\U_{R}) = (\alpha_k\rho_k \E_k u_k+\alpha_k \pi_k u_k) \left ( \vect{W}_r\left (0^+;\mathscr{M}(\vect{U}_{L}),\mathscr{M}(\vect{U}_{R}) \right ) \right ) - \mathbf{1}_{\left \lbrace u_2^* \geq 0 \right \rbrace }(u_2^* \, \pi_1^*) \left ( (\alpha_k)_{R}- (\alpha_k)_{L} \right ). 
\end{aligned}
\end{equation*}
\end{prop}

\begin{proof}
 The proposition follows from elementary verifications using equations  \eqref{fvs}-\eqref{autre_update}-\eqref{ener_update} and the expressions of the energy numerical fluxes $(\alpha_k\rho_k \E_k u_k+\alpha_k \pi_k u_k)_{j+\frac 12}$ given in Proposition \ref{chap2lemmastab1}.
\end{proof}

For the reader who is eager to rapidly implement the numerical scheme, we refer to Appendix \ref{choixa1a2} where the expressions of the numerical fluxes $\mathcal{F}^{\pm}(\U_{L},\U_{R})$ are given in detail.
\medskip

Recasting the scheme in a finite volume formulation with two interface fluxes $\mathcal{F}^{\pm}(\U_{L},\U_{R})$ is very interesting since it allows a nearly straightforward extension of the scheme to the 2D and 3D versions of the Baer-Nunziato model on unstructured meshes. Indeed, the multi-dimensional Baer-Nunziato model is invariant by Galilean transformations. Therefore, by assuming that, in the neighborhood of a multi-D cell interface, one has a local 1D Riemann problem in the orthogonal direction to the interface, it is possible to use the very same fluxes   $\mathcal{F}^{\pm}(\U_{L},\U_{R})$.

\subsection{Main properties of the scheme}
\medskip
We may now state the following theorem, which gathers the main properties of this scheme,  and which constitutes the main result of the paper.

\begin{them}
The finite volume scheme \eqref{fvs_2} for the Baer-Nunziato model has the following properties:
%\begin{itemize}
 %\item \textbf{Consistency.}  The numerical fluxes satisfy $\mathbf{F}^{-}(\vect{U},\vect{U})=\mathbf{F}^{+}(\vect{U},\vect{U})=\textnormal{\textbf f}(\vect{U})$.
 
 \begin{itemize}
  \item \textbf{Positivity:} Under the CFL condition (\ref{chap2cfl}), the scheme preserves positive values of the phase fractions, densities and internal energies: for all $n\in\xN$, if ($\U_j^n\in\Omega_{\U}$ for all $j\in\Z$), then $0 <(\alpha_{k})_j^{n+1} < 1$, $(\alpha_k \rho_k)_j^{n+1} >0$, and  $(E_k-u_k^2/2)_j^{n+1} >0$ for $k=1,2$ and all $j\in\Z$, \ie\, ($\U_j^{n+1}\in \Omega_{\U}$ for all $j\in\Z$). Moreover, if the thermodynamics of phase $k$ follows an ideal gas or a stiffened gas \eos, then the finite volume scheme \eqref{fvs_2} preserves real values for the speed of sound of phase $k$: for all $n\in\xN$, if $c_k\left ((\rho_k)_j^n,(e_k)_j^n \right )^2>0$, for all $j\in\Z$, then  $c_k\left ((\rho_k)_j^{n+1},(e_k)_j^{n+1} \right )^2>0$, for all $j\in\Z$ (see Proposition \ref{propspectrebn} and Remark \ref{remstiffeos} for the definition of $c_k(\rho_k,e_k)^2$).

  \item \textbf{Conservativity:} The discretizations of the partial masses $\alpha_k\rho_k,\,k\in\unde$, the total mixture momentum $\alpha_1\rho_1 u_1+\alpha_2\rho_2 u_2$ and the total mixture energy $\alpha_1\rho_1 E_1+\alpha_2\rho_2 E_2$, are conservative.
  
  \item \textbf{Discrete entropy inequalities.} Assume that the relaxation parameters $(a_k)_{j+\frac 12}^n,\, k=1,2$ satisfy Whitham's condition at each time step and each interface, \ie\, that for all $n\in\xN$, $j\in\Z$, $(a_k)_{j+\frac 12}^n,\, k=1,2$ are large enough so that
  \begin{equation}
  \label{whithamter}
  ((a_k)_{j+\frac 12}^n)^2 > -\dv_{\tau_k}\PP_k(\T_k,s_k),
  \end{equation}
  for all $\T_k$ and $s_k$ in the solution $\xi\mapsto \vect{W}_r\left (\xi;\mathscr{M}(\vect{U}_{j}^n),\mathscr{M}(\vect{U}_{j+1}^n)\right)$. Then, the values $\U_{j}^n$ computed by the scheme satisfy the following discrete entropy inequalities: for $k=1,2$:
  \begin{equation}
  \label{ent_prop_final}
  (\alpha_k \rho_k s_k)(\U_j^{n+1}) \leq  (\alpha_k\rho_k s_k )(\U_j^{n})  - \frac{\Delta t}{\Delta x} \left (  (\alpha_k\rho_k s_k u_k)_{j+\frac 12}^{n}- (\alpha_k\rho_k s_k u_k)_{j-\frac 12}^{n}\right),
  \end{equation}
where the entropy fluxes $(\alpha_k\rho_k s_k u_k)_{j+\frac 12}^{n}$ are defined in Proposition \ref{chap2lemmastab1}. These inequalities are discrete counterparts of the entropy inequalities \eqref{BN_entrop_kneq} satisfied by the admissible weak solutions of the Baer-Nunziato model \eqref{BN_ener}.

 \end{itemize}

\end{them}

\medskip
Note that preserving positive values of the phase fractions, the phasic densities and also the phasic internal energies altogether is an unprecedented result. Furthermore, to our knowledge, this scheme is the first scheme approximating the solutions of the Baer-Nunziato model, for which discrete entropy inequalities \eqref{BN_entrop_kneq} are proven.

\begin{proof}
The positivity of the phase fractions $(\alpha_{k})_j^{n+1}$ and partial masses $(\alpha_k \rho_k)_j^{n+1}$ follows directly from Proposition \ref{chap2lemmastab1}. To check that the proposed algorithm preserves the positivity of the internal energies, namely $(e_k)^{n+1}_j=(E_k-u_k^2/2)_j^{n+1} > 0$, it is sufficient to notice that before the update of the energy in the second step, the internal energy $(e_k)^{n+1,-}_j := e_k((\tau_k)^{n+1,-}_j,(s_k)^{n+1,-}_j)$ is positive according to Proposition \ref{chap2lemmastab1}. The second step results in increasing the total energies and one has $(\alpha_k\rho_k E_k)_j^{n+1} \geq (\alpha_k\rho_k E_k)(\vect{U}_j^{n+1,-})$ by \eqref{ener_prop_bis}-\eqref{ener_update}, while $(\alpha_k \rho_k)_j^{n+1}= (\alpha_k \rho_k)_j^{n+1,-}$, which yields $(E_k)^{n+1}_j \ge E_k(\vect{U}^{n+1,-}_j)$. Now, as the kinetic energy $((u_k)^{n+1,-}_j)^2/2$ is unchanged by the second step and $E_k(\vect{U}^{n+1,-}_j)= (e_k)^{n+1,-}_j-((u_k)^{n+1,-}_j)^2/2$, we infer that $(e_k)^{n+1}_j  \ge (e_k)^{n+1,-}_j >0$ and hence the required positivity property for the internal energies.

\medskip
In the same way, for an ideal gas or a stiffened gas \eos, we prove that the scheme preserves real values of the speed of sound. One has $\rho_k c_k(\rho_k,e_k)^2 = \gamma_k(\gamma_k-1)(\rho_k e_k- p_{\infty,k})$ (with $ p_{\infty,k}=0$ for an ideal gas). If $c_k\left ((\rho_k)_j^n,(e_k)_j^n \right )^2>0$, then, before the update of the energy in the second step, the quantity $(\rho_k e_k- p_{\infty,k})^{n+1,-}_j$ is positive by Proposition \ref{chap2lemmastab1}. The update of the energy in the second step amounts to increasing the internal energy $(e_k)^{n+1}_j>(e_k)^{n+1,-}_j$ while keeping the density unchanged $ (\rho_k)_j^{n+1}= (\rho_k)_j^{n+1,-}$, which yields $(\rho_k e_k- p_{\infty,k})_j^{n+1}>0$, hence the positivity of $c_k(\rho_k,e_k)^2$.
 
\medskip

We now prove the discrete entropy inequalities \eqref{ent_prop_final}. The first step provides a conservative update of the phasic entropy equations. Indeed, the last two components of the vector equation \eqref{fvs} yield
\begin{equation}
 \label{ent_prop_bis}
  (\alpha_k \rho_k s_k)_j^{n+1,-}  =   (\alpha_k\rho_k s_k )(\U_j^{n})  - \frac{\Delta t}{\Delta x} \left (  (\alpha_k\rho_k s_k u_k)_{j+\frac 12}^{n}- (\alpha_k\rho_k s_k u_k)_{j-\frac 12}^{n}\right).
 \end{equation}
Then, thanks to the thermodynamics assumptions, we know from Proposition \ref{prop_ent_convex} that 
$$
\dv_{\alpha_k \rho_k E_k}(\alpha_k \rho_k s_k)(\U)=-1/T_k.
$$
Consequently, we infer from $(\alpha_k\rho_k E_k)_j^{n+1} \geq (\alpha_k\rho_k E_k)(\vect{U}_j^{n+1,-}) $ that $(\alpha_k \rho_k s_k)(\U_j^{n+1})\leq (\alpha_k \rho_k s_k)_j^{n+1,-}$. Injecting in \eqref{ent_prop_bis} yields the discrete entropy inequalities \eqref{ent_prop_final}.

\end{proof}

\medskip
%\begin{rem}
For most of equations of state, that are given as a function $p_k(\rho_k,e_k)$, the quantity $\alpha_k \rho_k s_k$ cannot be expressed as an explicit function of $\U$, which makes it even impossible to compute the time step initial values of the entropies $(\alpha_k\rho_k s_k )(\U_j^{n})$. Still, this does not prevent the discrete inequalities \eqref{ent_prop_final} from holding true.
%\end{rem}

\medskip
%\begin{rem}
The impossibility, for many equations of state, to express the entropies $s_k$ explicitly does not prevent the computation of the numerical fluxes $\mathcal{F}^{\pm}_{i}(\U_j^{n},\U_{j+1}^n),\,i=1,..,7$, for the updates of $(\alpha_1)_j^{n+1}$, $(\alpha_k \rho_k)_j^{n+1}$, $(\alpha_k \rho_k u_k)_j^{n+1}$ in the first step, and the updates of $(\alpha_k\rho_k E_k)_j^{n+1}$ in the second step. Indeed, even though these numerical fluxes involve terms of the form $e_k(\T_k,s_k)$ and $\pi_k(\tau_k,\T_k,s_k)$ to be evaluated on the relaxation Riemann solution of the first step, the discussion of Section \ref{subsec_Riemsol} shows that these functions are solely evaluated on the piecewise constant initial data in terms of $e_k((\T_{k})_j^n,(s_{k})_j^n)$ and $\PP_k((\T_{k})_j^n,(s_{k})_j^n)$. But observe that by the thermodynamics, $e_k((\T_{k})_j^n,(s_{k})_j^n)$ is nothing else but $(e_k)_j^n=(E_k-u_k^2/2)_j^n$ and $\PP_k((\T_{k})_j^n,(s_{k})_j^n)$ is equal to $p_k((\rho_k)_j^n,(e_k)_j^n)$. Hence, even though the entropy is used for the analysis of the numerical method, one may implement this scheme even for general and possibly incomplete equations of state that are given as a function $p_k(\rho_k,e_k)$, since the numerical fluxes may still be computed at each time step in terms of the initial unknowns $\U_j^n,j\in\Z$. In particular, this allows the use of tabulated equations of state. 
%\end{rem}

\section{Numerical tests}
\label{numtest}
In this section, we present Riemann-type test-cases on which the performance of the relaxation scheme is tested and compared with that of three other schemes: Schwendeman-Wahle-Kapila's first order Godunov-type scheme \cite{SWK}, Toro-Tokareva's finite volume HLLC scheme \cite{TT} and Rusanov's scheme (a Lax-Friedrichs type scheme, see \cite{GHS}). The thermodynamics follow either an ideal gas or a stiffened gas law:
$$
p_k(\rho_k,e_k)=(\gamma_k-1)\rho_k e_k- \gamma_k p_{\infty,k},
$$
where $\gamma_k>1$ and $p_{\infty,k}\geq 0$ are two constants.
The \eos\, parameters of each test-case are given in Table \ref{Table1} as well as the initial discontinuity position, the final time of the simulation and the CFL number. The initial and intermediate states of each solution are given in Tables \ref{Table_TC1} to \ref{Table_TC5}. The $u_2$-contact discontinuity separates two regions denoted $-$ and $+$ respectively on the left and right sides of the discontinuity. If the $u_1$-contact discontinuity has non-zero strength, an additional region $L*$ or $R*$ also exists according to the sign of $u_2-u_1$ as described in Figure \ref{Fig_struct}.

\medskip
We recall that the scheme relies on a relaxation Riemann solver which requires solving a fixed point in order to compute, for every cell interface $x_{j+\frac 12}$, the zero of a scalar function (see eq. \eqref{chap3mel} in Appendix \ref{choixa1a2}). A dichotomy (bisection) method is used in order to compute this solution. The iterative procedure is stopped when the error is less than $ 10^{-12}$.

\newcolumntype{g}{>{\centering\arraybackslash}p{26mm}<{}}
\newcolumntype{h}{>{\centering\arraybackslash}p{21mm}<{}}

\begin{table}[ht!]
\centering
\begin{tabular}{|hhhhhh|}
\hline
			& Test 1 	& Test 2		& Test 3		& Test 4		& Test 5\\
\hline
$\gamma_1$ 		&$1.4$		&$1.4$			&$1.4$			&$3$			&$3$	   \\
$p_{\infty,1}$		&$0$		&$0$			&$0$			&$0$			&$0$  \\
$\gamma_2$		&$1.4$		&$3$			&$1.4$			&$1.4$			&$1.4$   \\
$p_{\infty,2}$		&$0$		&$100$			&$0$			&$0$			&$0$	   \\
$x_0$			&$0$		&$0.8$			&$0.5$			&$0$			&$0$	   \\
$T_{\rm max}$		&$0.15$		&$0.007$		&$0.15$			&$0.15$			&$0.05$	   \\
${\rm CFL}$		&$0.45$		&$0.45$			&$0.45$			&$0.45$			&$0.45$	   \\
\hline
\end{tabular}
\protect \parbox[t]{13cm}{\caption{E.O.S. parameters, initial discontinuity position, final time, Courant-Friedrichs-Lewy number.\label{Table1}}}

\end{table}

\begin{figure}[ht!]
\centering
\begin{tabular}{cc}
\includegraphics[width=6cm,height=4cm]{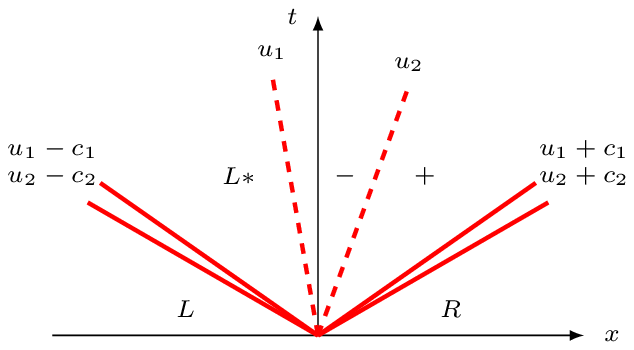}&
\includegraphics[width=6cm,height=4cm]{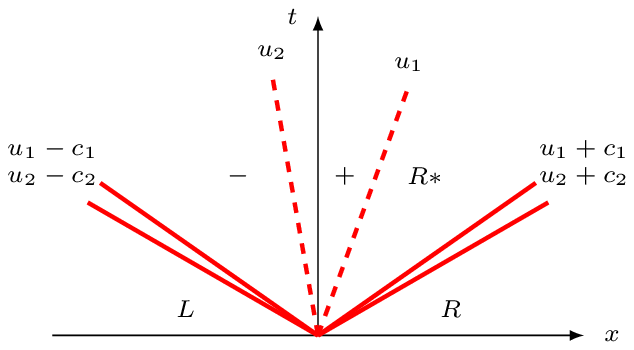}
\end{tabular}
\protect \parbox[t]{13cm}{\caption{Structure of the Riemann solutions, notations for the intermediate states.\label{Fig_struct}}}
\end{figure}

%\clearpage

%%%%%%%%%%%%%%%%%%%%%%%%%% TC1  %%%%%%%%%%%%%%%%%%%%%%%%%%%%%%%%%%%%%%%%%%%%%%%%%%%%%%%%%%
\subsection{Results for Test-case 1}

\begin{table}[ht!]
\centering
\begin{tabular}{|hhhhhh|}
\hline
		& Region $L$ 	& Region $L*$		& Region $-$		& Region $+$		& Region $R$ \\
\hline
$\alpha_1$ 	&$0.2$		&$0.2$			&$0.2$			&$0.7$			& $0.7$	   \\
$\rho_1$	&$0.21430$	&$0.35$			&$0.698$		&$0.90583$		&$0.96964$  \\
$u_1$		&$-0.02609$	&$-0.7683$		&$-0.7683$		&$-0.11581$		&$-0.03629$   \\
$p_1$		&$0.3$		&$0.6045$		&$0.6045$		&$0.87069$		&$0.95776$	   \\
$\rho_2$	&$1.00003$	&$1.00003$		&$0.9436$		&$1.0591$		&$0.99993$	   \\
$u_2$		&$0.00007$	&$0.00007$		&$0.0684$		&$0.0684$		&$-0.00004$	   \\
$p_2$		&$1.0$		&$0.9219$		&$0.9219$		&$1.08383$		&$1.0$	   \\
\hline
\end{tabular}
\protect \parbox[t]{13cm}{\caption{Test-case 1: Left, right and intermediate states of the exact solution.\label{Table_TC1}}}
\end{table}

In this first test-case, both phases follow an ideal gas \eos\, (see Table \ref{Table1}). The wave pattern for phase 1 consists of a left-traveling shock, a material contact discontinuity $u_1$, a phase fraction discontinuity of velocity $u_2$ and a right-traveling rarefaction wave. For phase 2 the wave pattern is composed of a left-traveling rarefaction wave, the phase fraction discontinuity, and a right-traveling shock.

\medskip
In Figure \ref{Figcase1}, the approximate solution computed with the relaxation scheme is compared with the exact solution, and with the approximate solutions obtained with the Godunov-type scheme, the HLLC scheme and Rusanov's scheme. The results show that unlike Rusanov's scheme, the three other methods, which give very similar results, correctly capture the intermediate states even for this rather coarse mesh of $100$ cells. This coarse mesh is a typical example of an industrial mesh, reduced to one direction, since $100$ cells in 1D correspond to a $10^6$-cell mesh in 3D. It appears that the contact discontinuity is captured more sharply by the relaxation scheme, the Godunov-type scheme and the HLLC scheme than by Rusanov's scheme for which the numerical diffusion is larger. However, we observe that the Godunov-type and HLLC schemes seem to be slightly more accurate that the relaxation scheme when capturing the $u_2$-contact discontinuity for the phase 1 variables. Indeed, for the relaxation scheme, there is one more point within the contact discontinuity for these variables. We can also see that for the phase 2 variables, there are no oscillations as one can see for Rusanov's scheme: the curves are monotone between the intermediate states. For phase 1, the intermediate states are captured by the relaxation, the Godunov-type and the HLLC methods, while with Rusanov's scheme, this weak level of refinement is clearly not enough to capture any intermediate state. These observations assess that, for the same level of refinement, the relaxation method (as well as the Godunov-type scheme and the HLLC scheme) is much more accurate than Rusanov's scheme.

\medskip
A mesh refinement process has also been implemented in order to check numerically the convergence of the method, as well as it's performances in terms of CPU-time cost. For this purpose, we compute the discrete $L^1$-error between the approximate solution and the exact one at the final time $T_{\rm max}=N\Delta t$, normalized by the discrete $L^1$-norm of the exact solution:
\begin{equation}
\label{chap3error}
 E(\Delta x) = \dfrac{\ds \sum_{j}|\phi_j^N-\phi_{ex}(x_j,T_{\rm max})| \Delta x}{\ds \sum_{j}|\phi_{ex}(x_j,T)| \Delta x},
\end{equation}
where $\phi$ is any of the non conservative variables $(\alpha_1, \rho_1, u_1,p_1, \rho_2, u_2,p_2)$. The calculations have been implemented on several meshes composed of $100 \times 2^n$ cells with $n=0,1,..,10$ (knowing that the domain size is $L=1$). In Figure \ref{Figcase1bis}, the error $E(\Delta x)$ at the final time $T_{\rm max}=0.15$, is plotted against $\Delta x$ in a $log-log$ scale. We can see that all the errors converge towards zero with the expected order of $\Delta x^{1/2}$, except the error for $u_2$ which seems to converge with a higher rate. However, $\Delta x^{1/2}$ is only an asymptotic order of convergence, and in this particular case, one would have to implement the calculation on more refined meshes in order to reach the theoretically expected order of $\Delta x^{1/2}$.

\medskip
Figure \ref{Figcase1bis} also shows the error on the non conservative variables with respect to the CPU-time of the calculation expressed in seconds for both the relaxation scheme and Rusanov's scheme. Each point of the plot corresponds to one single calculation for a given mesh size. One can see that, if one prescribes a given level of the error, the computational cost of Rusanov's scheme is significantly higher than that of the relaxation method for all the variables. For instance, for the same error on the phase 1 density $\rho_1$, the gain in computational cost is more than a hundred times when using the relaxation method rather than Rusanov's scheme which is a quite striking result. Indeed, even if Rusanov's scheme is known for its poor perfrmances in terms of accuracy, it is also an attractive scheme for its reduced complexity. This means that the better accuracy of the relaxation scheme (for a fixed mesh) widely compensates for its (relative) complexity.

\begin{figure}[ht!]
\scriptsize
\begin{center}
\begin{tabular}{cc}
Wave structure & $\alpha_1$ \\[1ex]
\includegraphics[width=7cm,height=4.5cm]{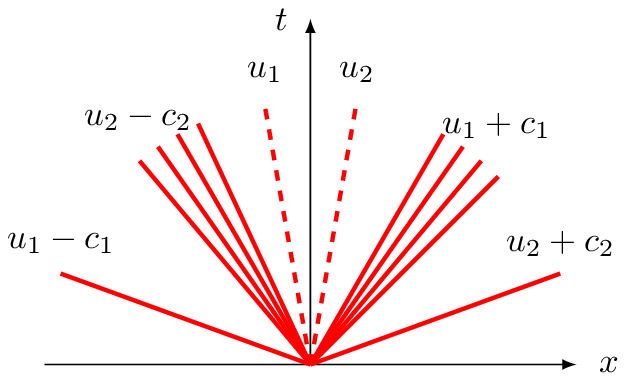}&
\includegraphics[width=7cm,height=4.5cm]{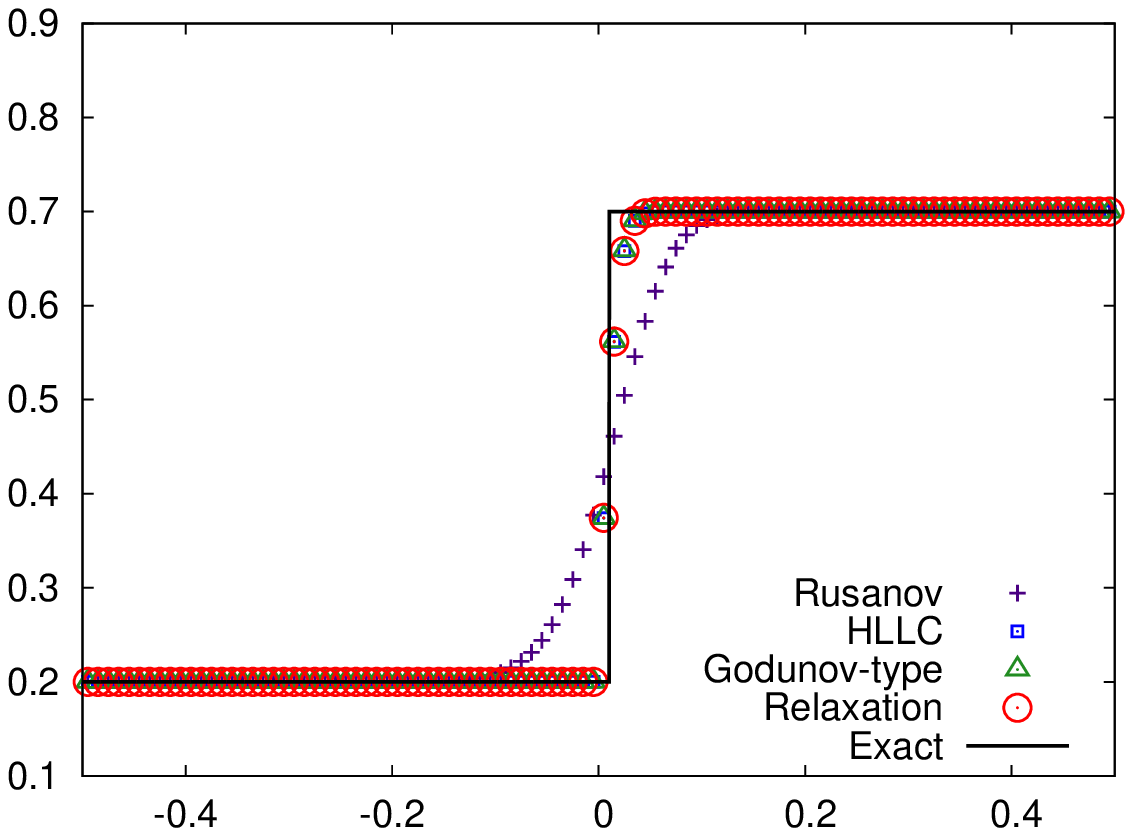} \\[3ex]
$u_1$ & $u_2$\\[1ex]
\includegraphics[width=7cm,height=4.5cm]{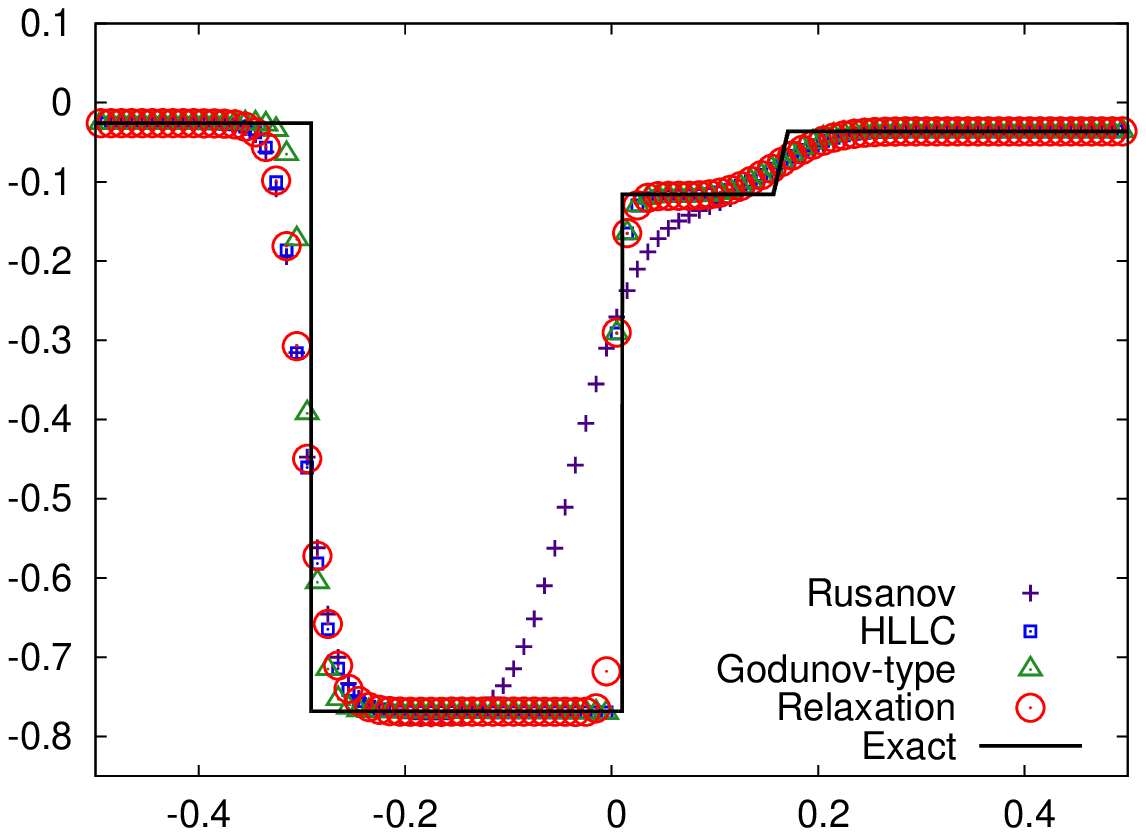} &
\includegraphics[width=7cm,height=4.5cm]{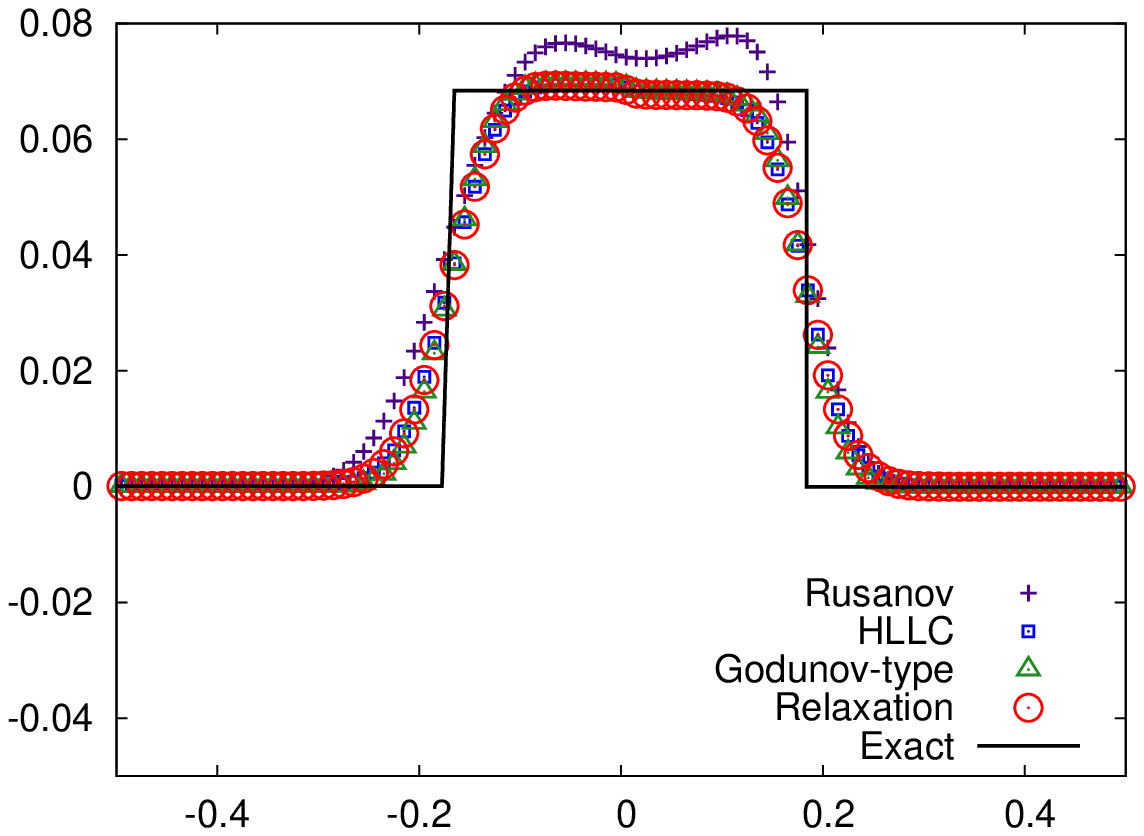} \\[3ex] 
$\rho_1$ & $\rho_2$ \\[1ex]
\includegraphics[width=7cm,height=4.5cm]{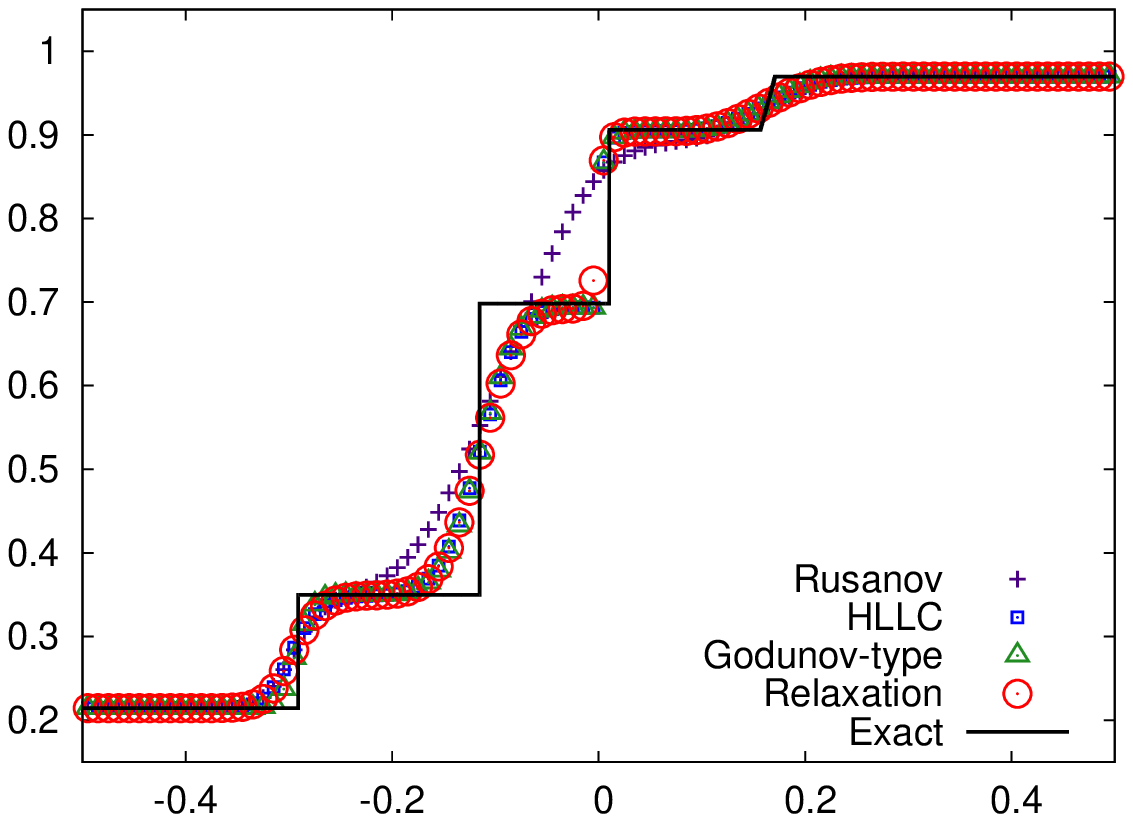} &
\includegraphics[width=7cm,height=4.5cm]{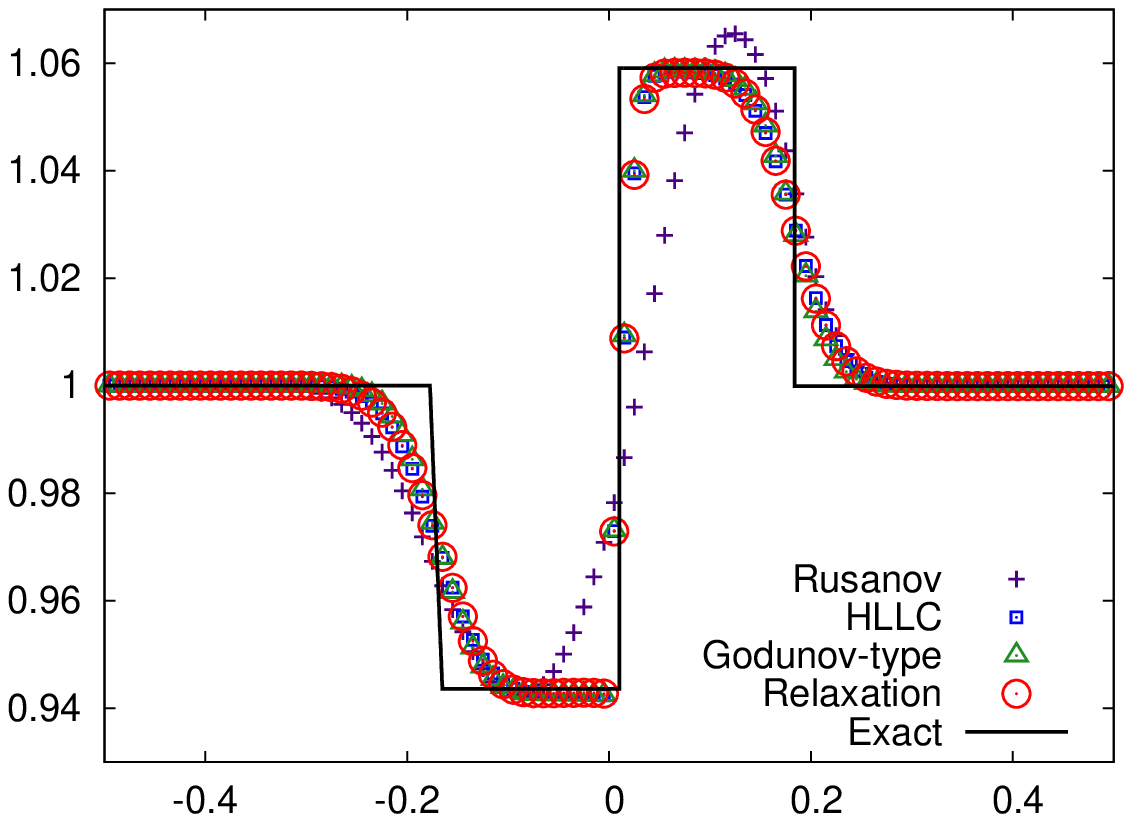}\\[3ex]
$p_1$  & $p_2$\\[1ex]
\includegraphics[width=7cm,height=4.5cm]{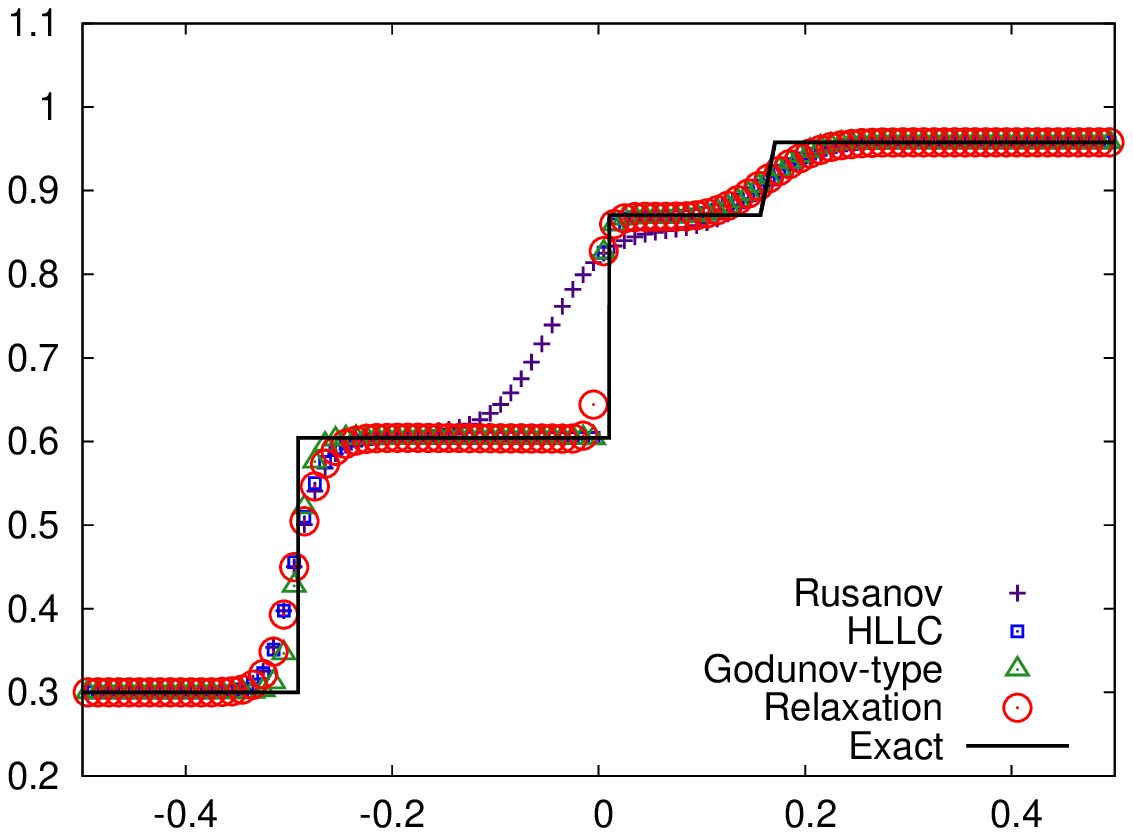} &
\includegraphics[width=7cm,height=4.5cm]{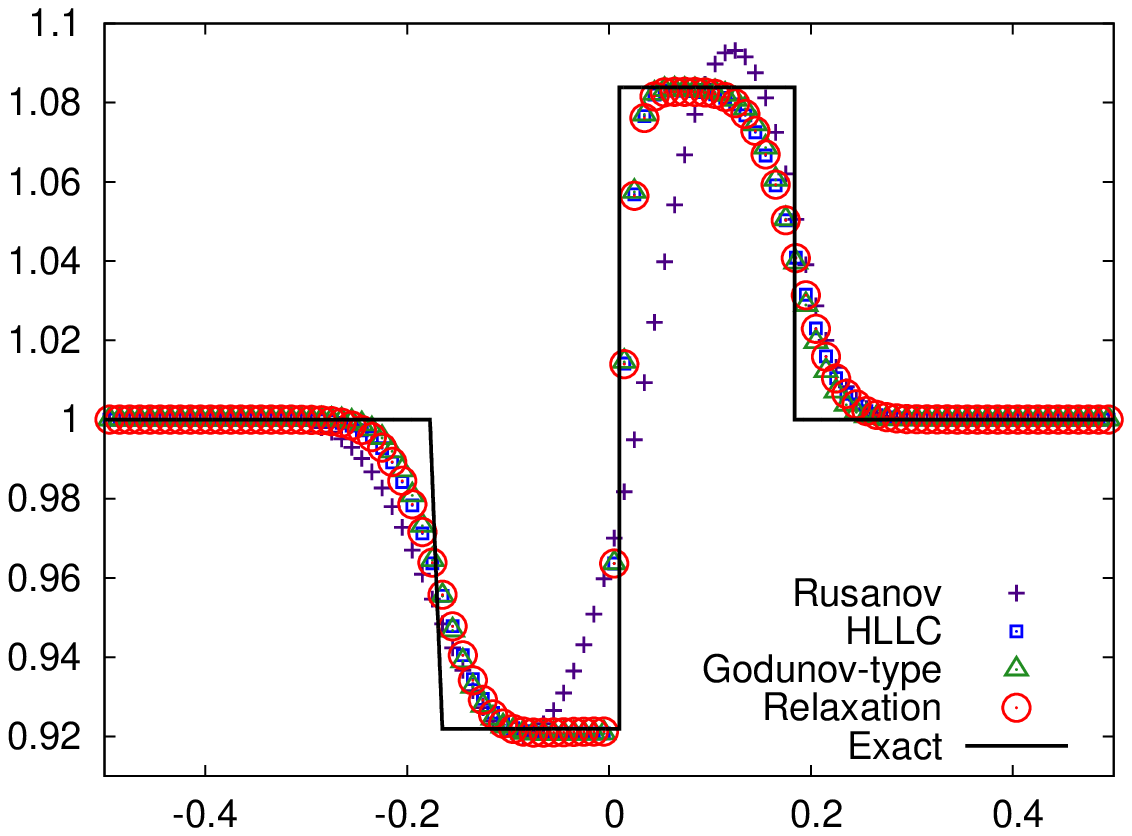}
\end{tabular}
\protect \parbox[t]{13cm}{\caption{Test-case 1: Structure of the solution and space variations of the physical variables at the final time $T_{\rm max}=0.15$. Mesh size: $100$ cells.\label{Figcase1}}}
\end{center}
\normalsize
\end{figure}

\begin{figure}[ht!]
\begin{center}
 \begin{tabular}{cc}
 & $\alpha_1$\\[1ex]
\includegraphics[width=7cm,height=4cm]{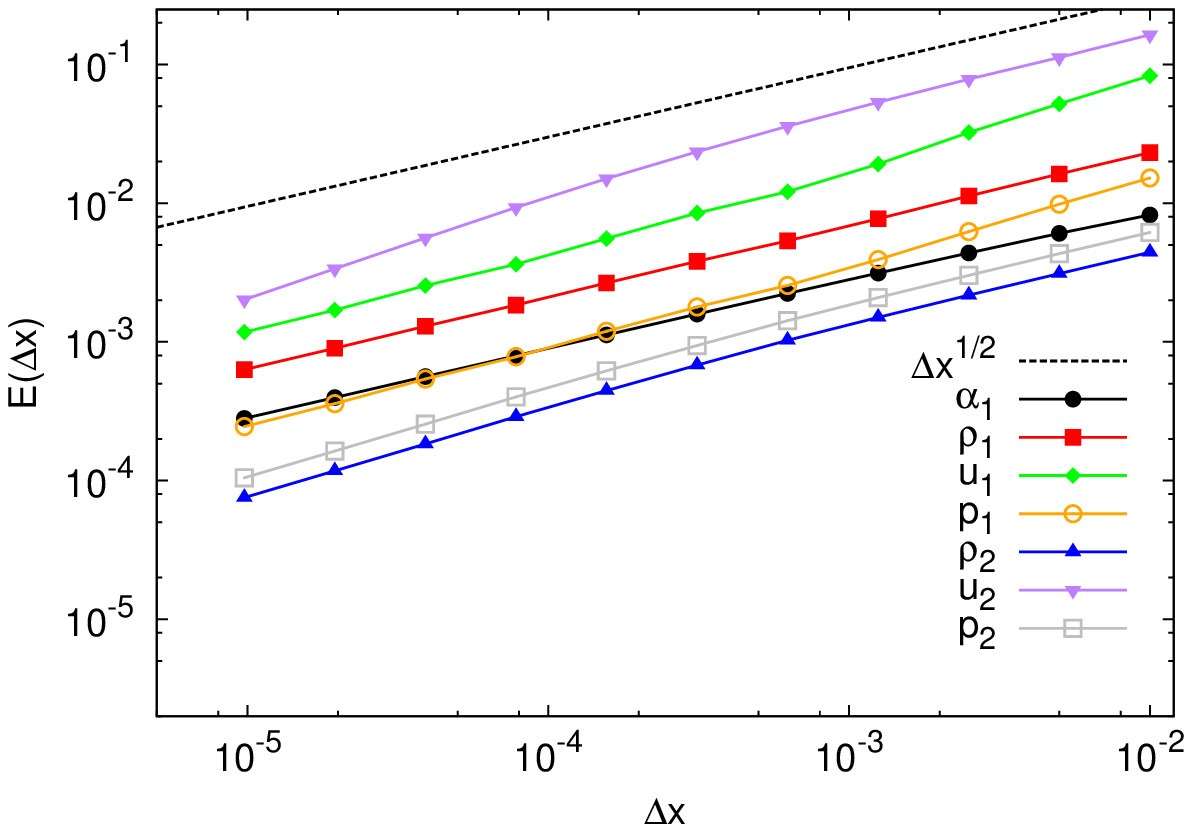}&
\includegraphics[width=7cm,height=4cm]{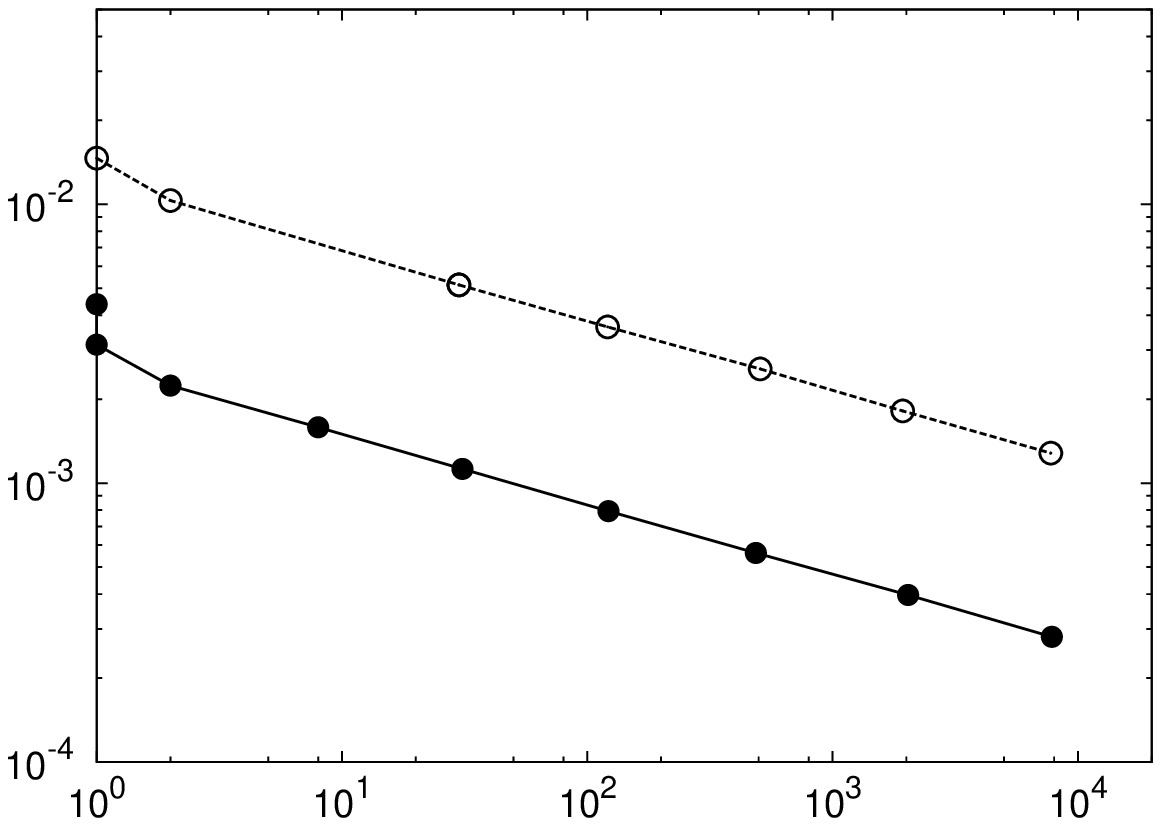} \\[3ex]
$u_1$ & $u_2$\\[1ex]
\includegraphics[width=7cm,height=4cm]{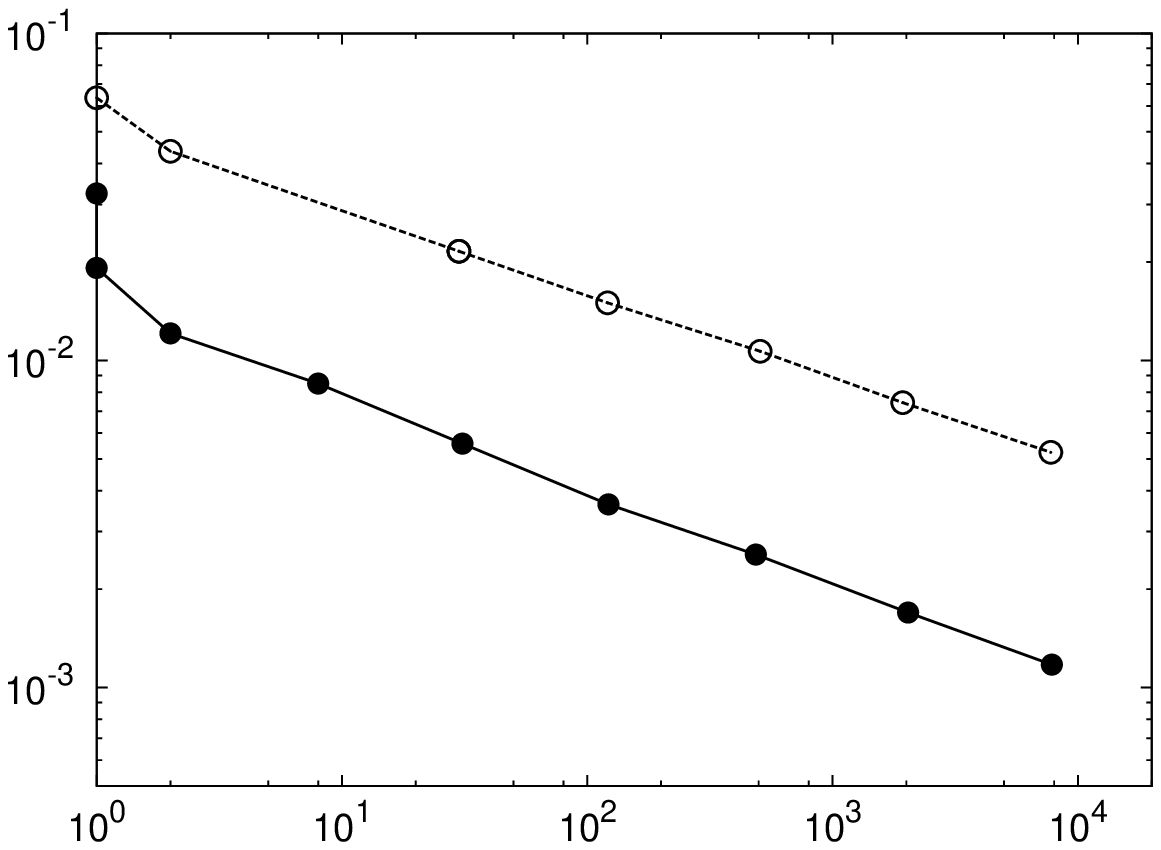} &
\includegraphics[width=7cm,height=4cm]{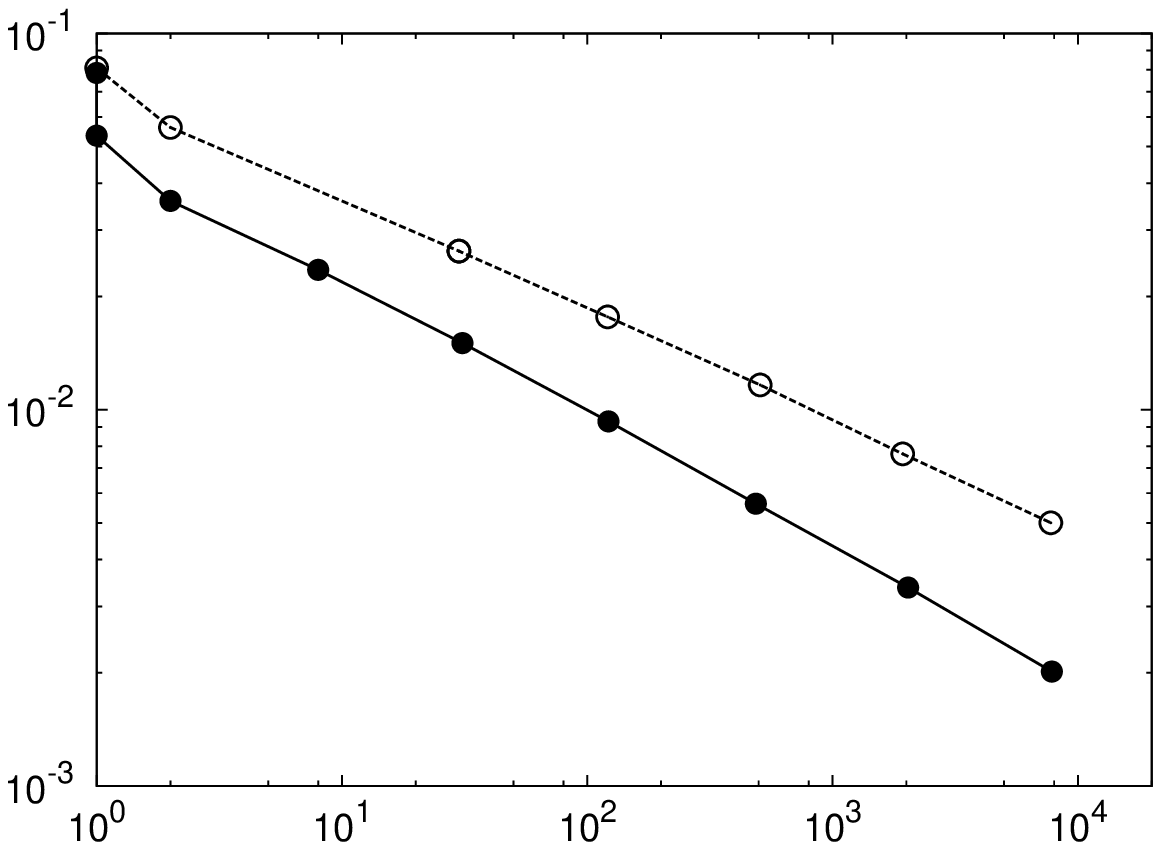} \\[3ex]
$\rho_1$ & $\rho_2$ \\[1ex]
\includegraphics[width=7cm,height=4cm]{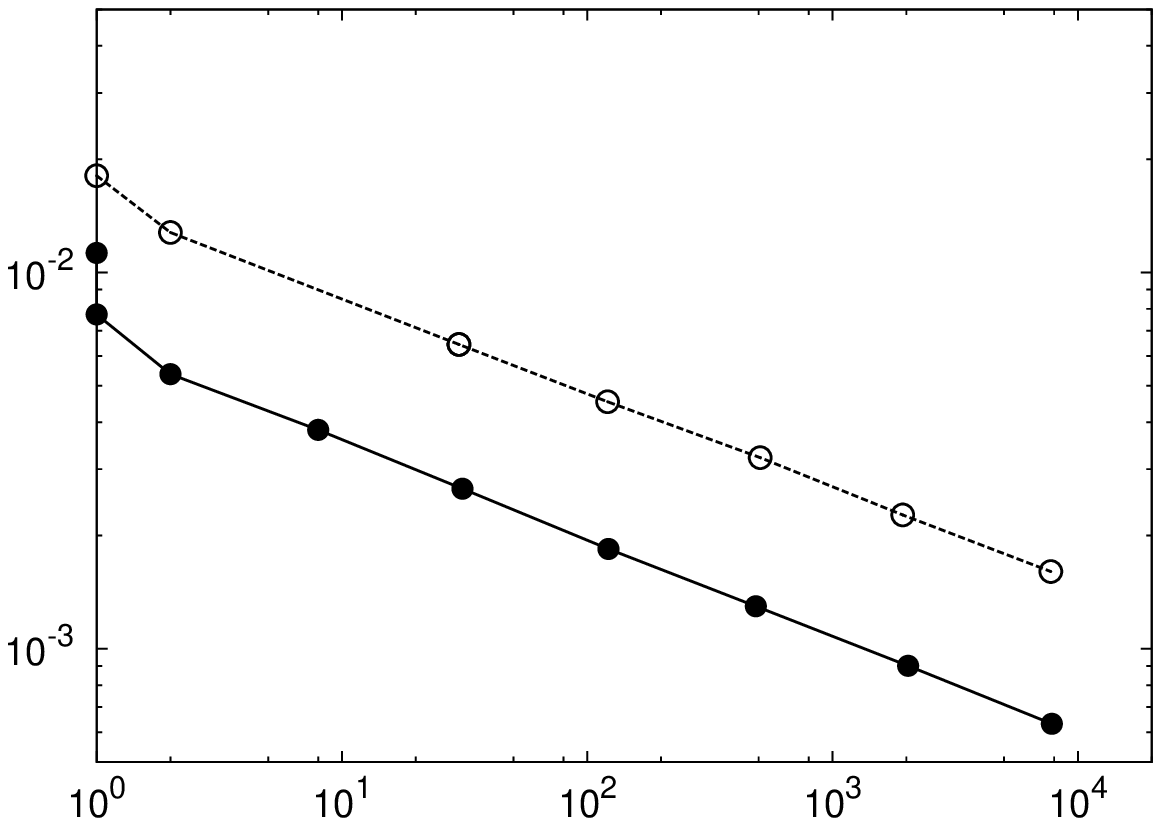} &
\includegraphics[width=7cm,height=4cm]{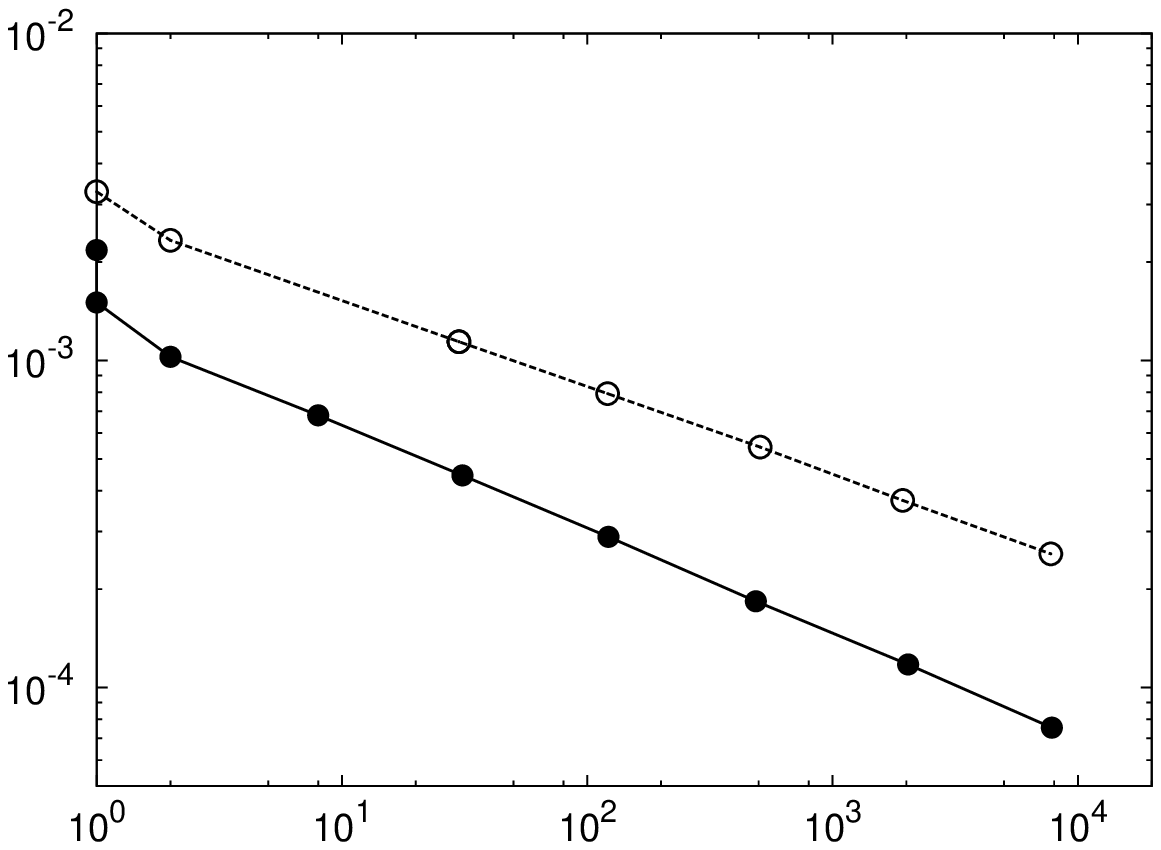}\\[3ex]
$p_1$  & $p_2$\\[1ex]
\includegraphics[width=7cm,height=4cm]{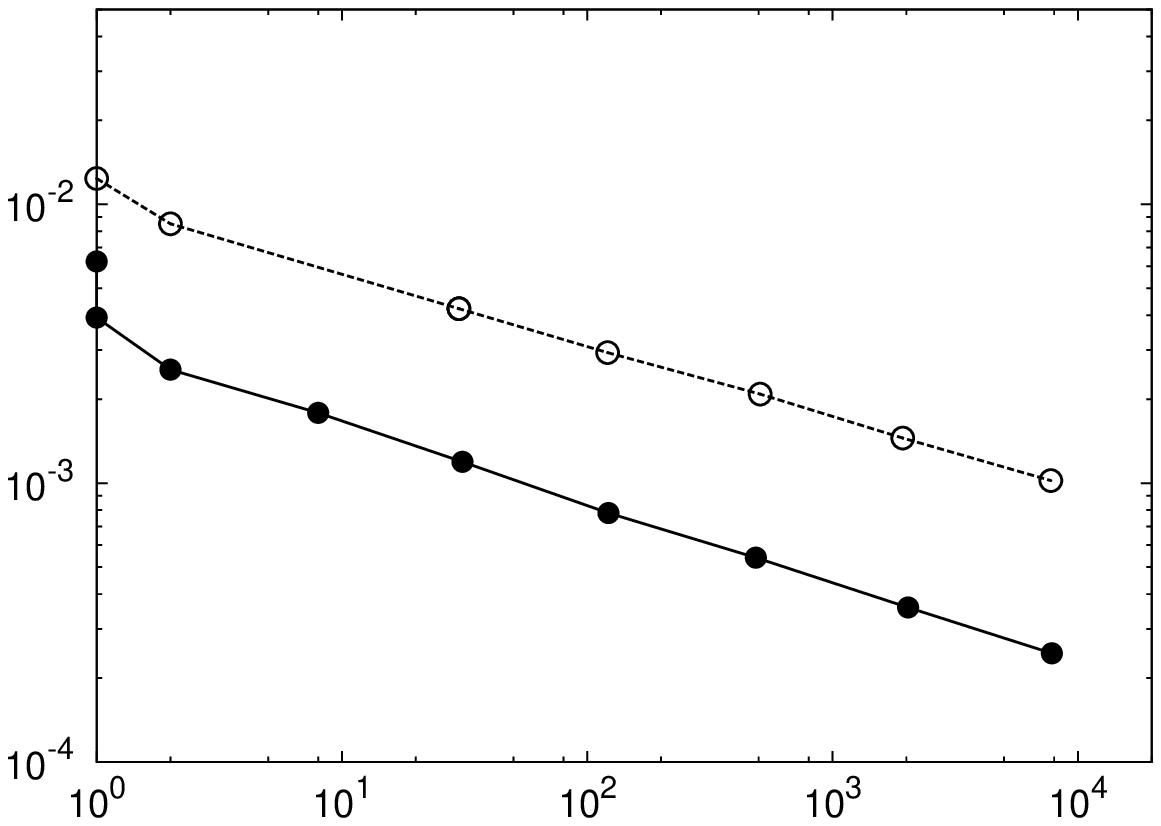} &
\includegraphics[width=7cm,height=4cm]{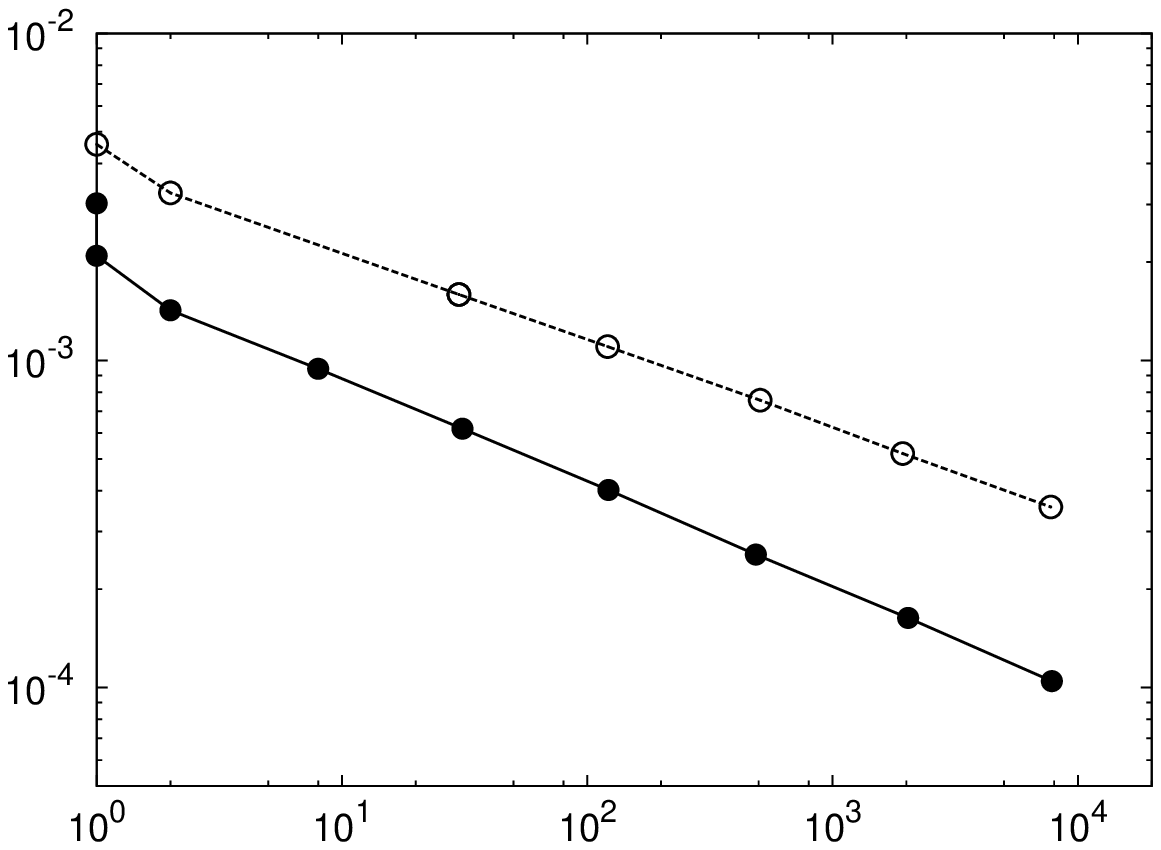}
\end{tabular}
\protect \parbox[t]{13cm}{\caption{Test-case 1: $L^1$-Error with respect to $\Delta x$ for the relaxation scheme and $L^1$-Error with respect to computational cost (in seconds) for the relaxation scheme (straight line) and Rusanov's scheme (dashed line).\label{Figcase1bis}}}
\end{center}
\end{figure}

%%%%%%%%%%%%%%%%%%%%%%%%%% TC2 %%%%%%%%%%%%%%%%%%%%%%%%%%%%%%%%%%

\subsection{Results for Test-case 2}

\begin{table}[ht!]
\centering
\begin{tabular}{|hhhhhh|}
\hline
		& Region $L$ 	& Region $-$		& Region $+$		& Region $R*$		& Region $R$ \\
\hline
$\alpha_1$ 	&$0.3$		&$0.3$			&$0.8$			&$0.8$			&$0.8$	\\
$\rho_1$	&$1.0$		&$0.4684$		&$0.50297$		&$5.9991$		&$1.0$	  \\
$u_1$		&$-19.59741$	&$6.7332$		&$-1.75405$		&$-1.75405$		&$-19.59741$	\\
$p_1$		&$1000.0$	&$345.8279$		&$382.08567$		&$382.08567$		&$0.01$		\\
$\rho_2$	&$1.0$		&$0.7687$		&$1.6087$		&$1.6087$		&$1.0$	\\
$u_2$		&$-19.59716$	&$-6.3085$		&$-6.3085$		&$-6.3085$		&$-19.59741$	\\
$p_2$		&$1000.0$	&$399.5878$		&$466.72591$		&$466.72591$		&$0.01$	\\
\hline
\end{tabular}
\protect \parbox[t]{13cm}{\caption{Test-case 2: Left, right and intermediate states of the exact solution.\label{Table_TC2}}}
\end{table}

The second test-case was taken from \cite{TT}. Phase 1 follows an ideal gas \eos\, while phase 2 follows a stiffened gas \eos\, (see Table \ref{Table1}). From left to right, the solution for phase 1 consists of a left-traveling rarefaction wave, the phase fraction discontinuity, a material contact discontinuity $u_1$, and  a right-traveling shock. For phase 2 the wave pattern is composed of a left-traveling rarefaction wave, the phase fraction discontinuity, and a right-traveling shock.

\medskip
As the jump of initial pressures is very large, \textbf{strong shocks} are generated in each phase. The distance between the right shock and contact waves is small in phase 1, which makes it difficult for all the schemes to capture the intermediate states at this weak level of refinement ($100$ cells). We observe however that the Godunov-type scheme, the HLLC scheme and the relaxation scheme remain more accurate than Rusanov's scheme. We also observe that the narrow intermediate state for $\rho_1$ between the $u_1$-contact discontinuity and the $\lbrace u_1+c_1 \rbrace$-shock is better captured with the Godunov-type and the HLLC schemes than by the relaxation scheme. For phase 2, Rusanov's scheme fails to correctly capture the speed of the right-going shock due to the large difference between the pressures before and after the shock. On the contrary, the other schemes capture the shock with the correct speed.

\medskip 
A convergence study has also been performed for this test-case. The observed convergence rate is slightly larger than $\Delta x^{1/2}$, and the error \textit{v.s.} CPU plots show a smaller computational cost for the relaxation scheme than for Rusanov's scheme. However, on this test-case, the observed gain in the computational time is less than in the first test-case and is not the same for all the variables.

\begin{figure}[ht!]
\scriptsize
\begin{center}
\begin{tabular}{cc}
Wave structure & $\alpha_1$ \\[1ex]
\includegraphics[width=7cm,height=4cm]{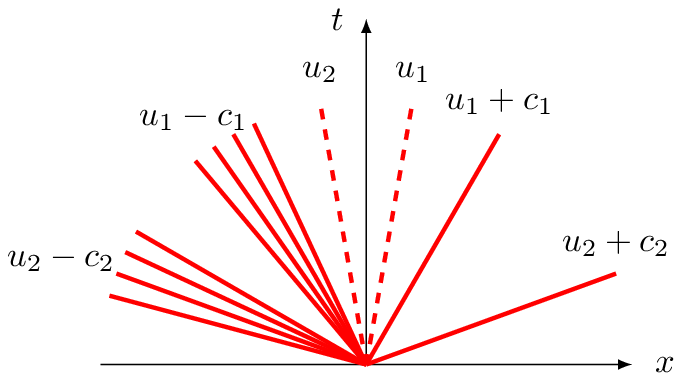}&
\includegraphics[width=7cm,height=4cm]{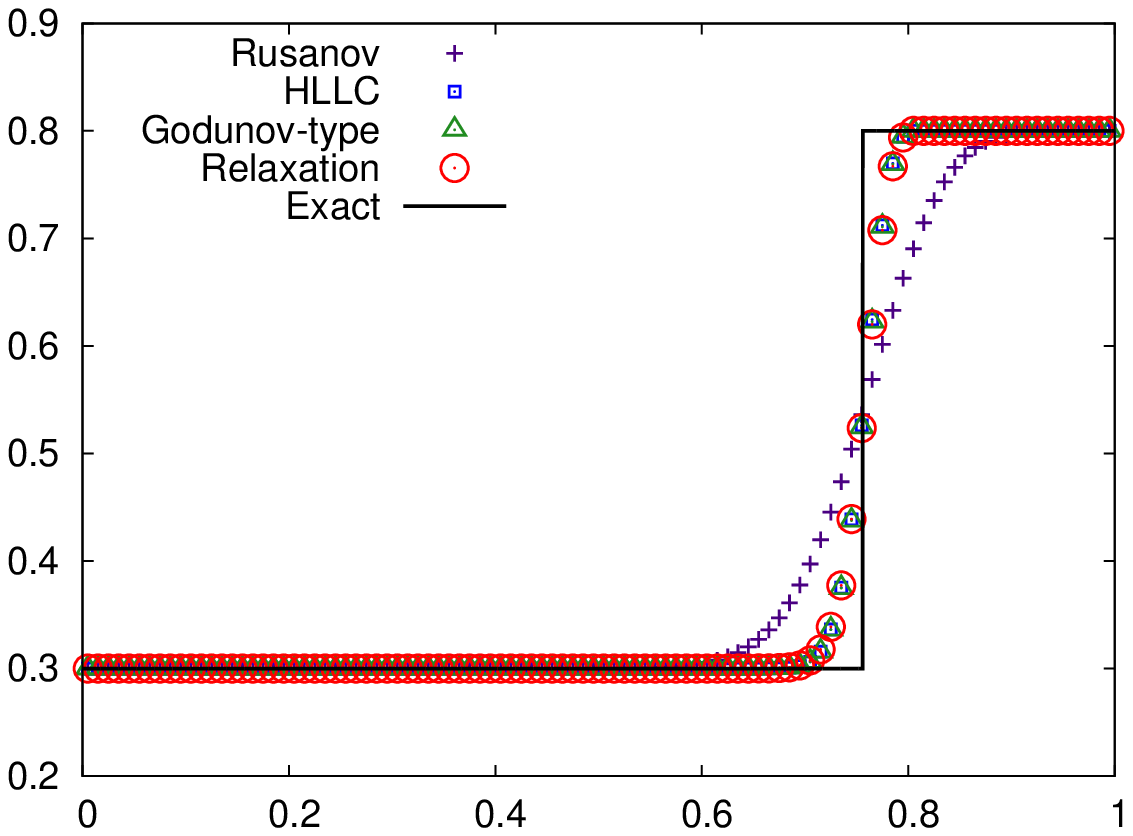} \\[3ex]
$u_1$ & $u_2$\\[1ex]
\includegraphics[width=7cm,height=4cm]{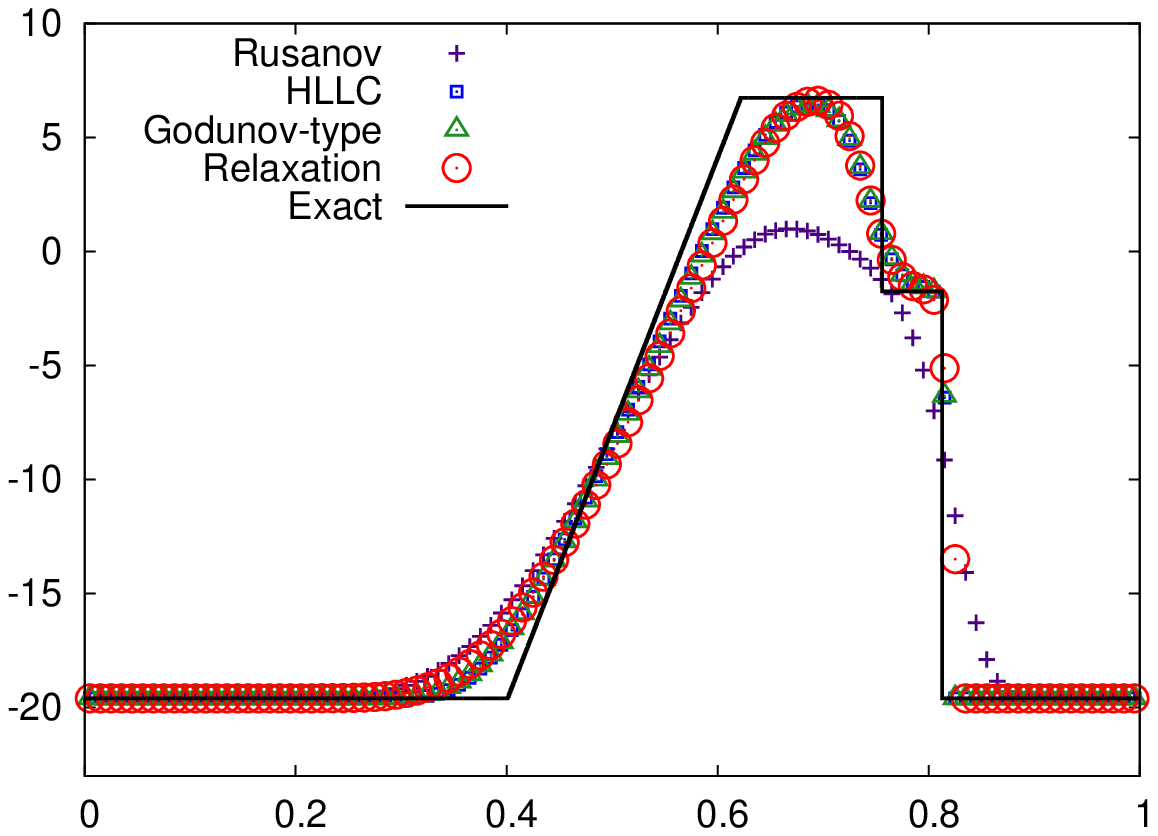} &
\includegraphics[width=7cm,height=4cm]{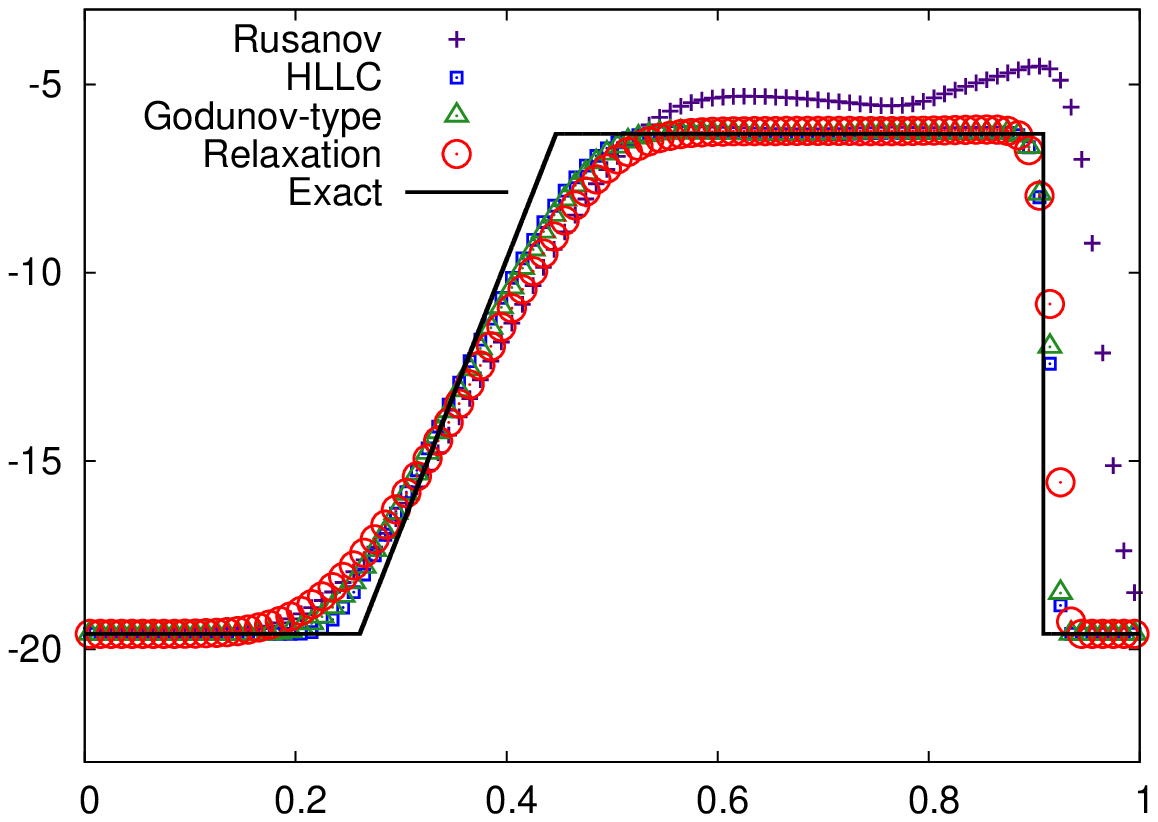} \\[3ex]
$\rho_1$ & $\rho_2$ \\[1ex]
\includegraphics[width=7cm,height=4cm]{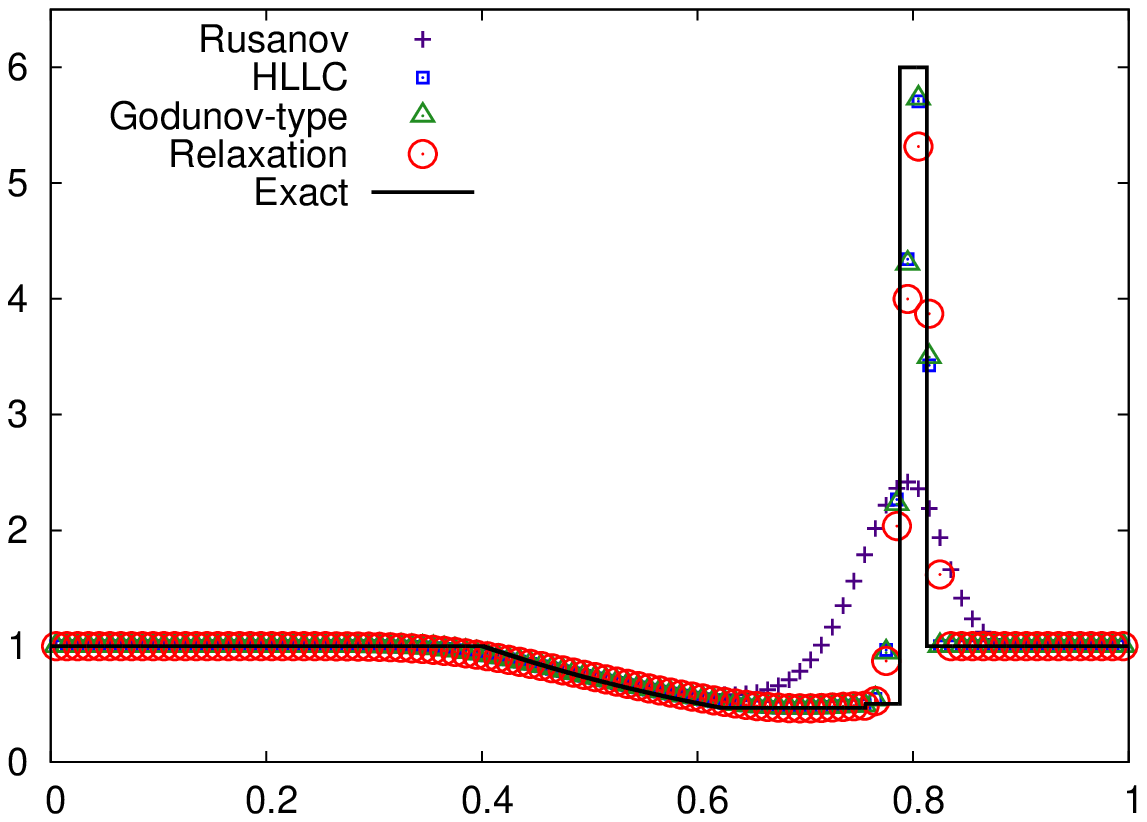} &
\includegraphics[width=7cm,height=4cm]{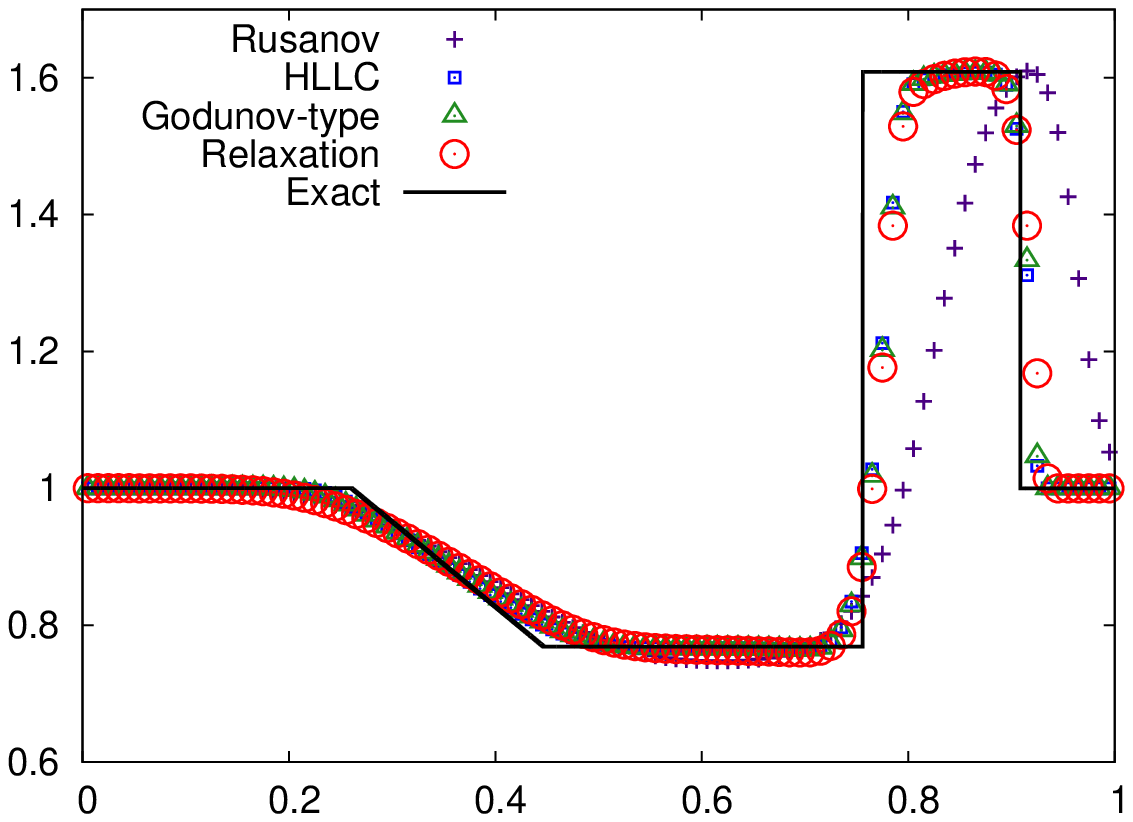}\\[3ex]
$p_1$  & $p_2$\\[1ex]
\includegraphics[width=7cm,height=4cm]{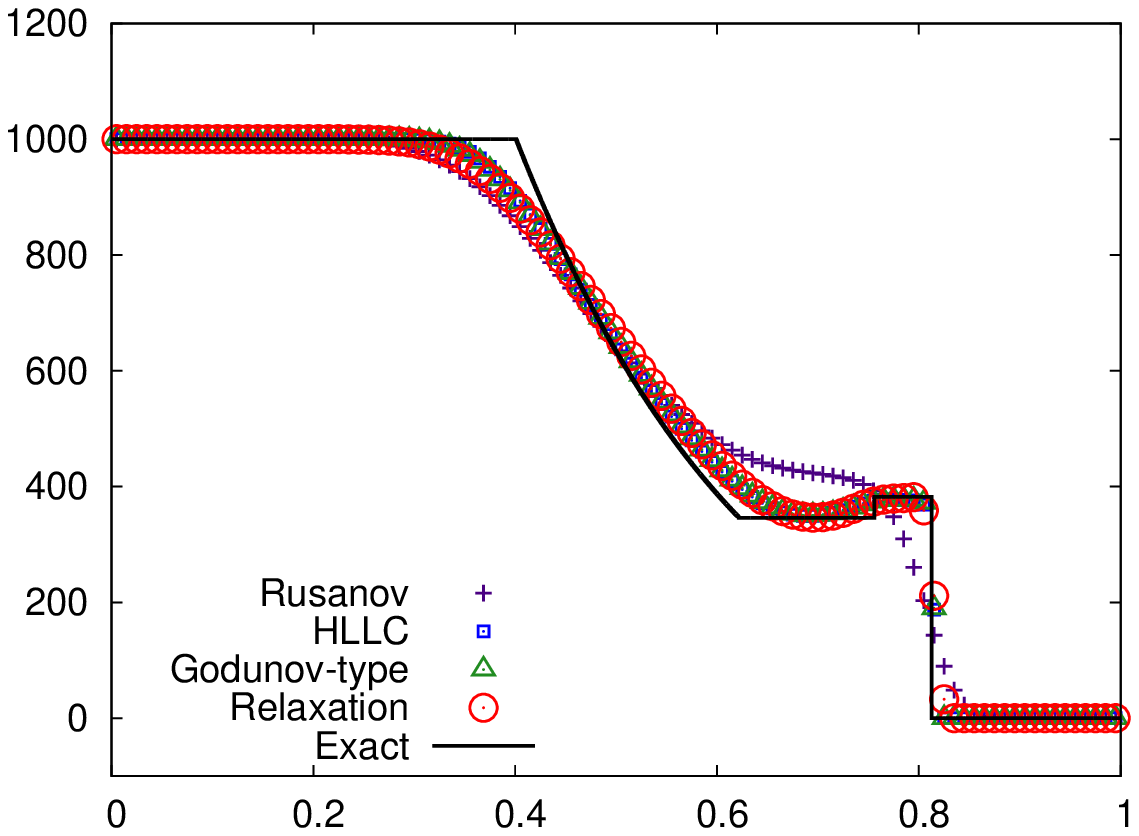} &
\includegraphics[width=7cm,height=4cm]{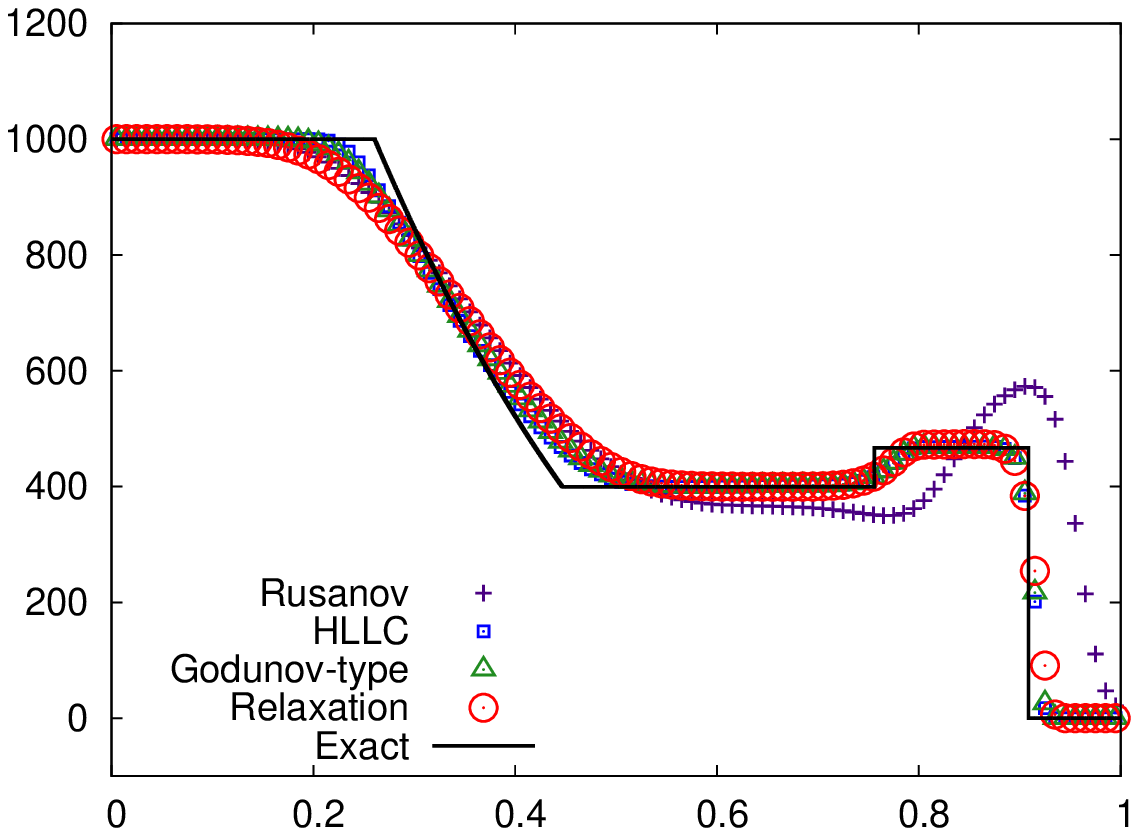}
\end{tabular}
\protect \parbox[t]{13cm}{\caption{Test-case 2: Structure of the solution and space variations of the physical variables at the final time $T_{\rm max}=0.007$. Mesh size: $100$ cells.\label{Figcase2}}}
\end{center}
\normalsize
\end{figure}

\begin{figure}[ht!]
\begin{center}
 \begin{tabular}{cc}
 & $\alpha_1$\\[1ex]
\includegraphics[width=7cm,height=4cm]{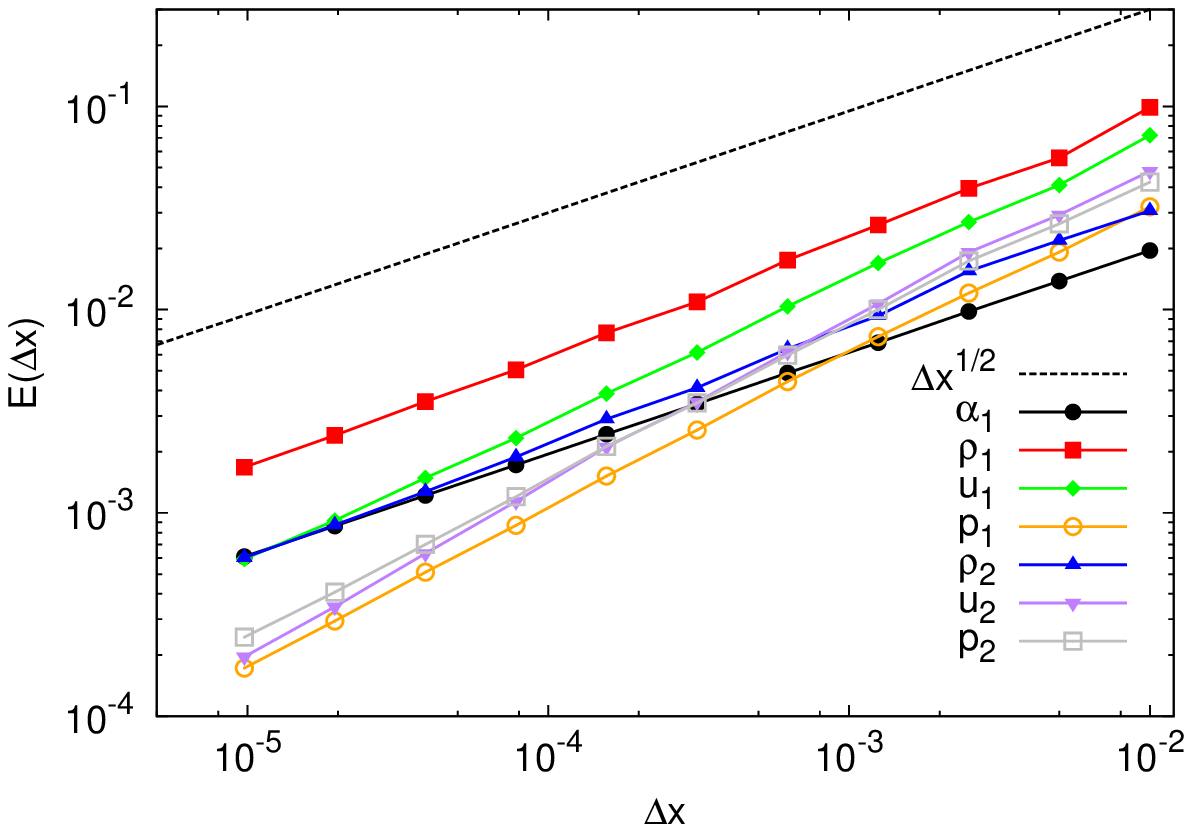}&
\includegraphics[width=7cm,height=4cm]{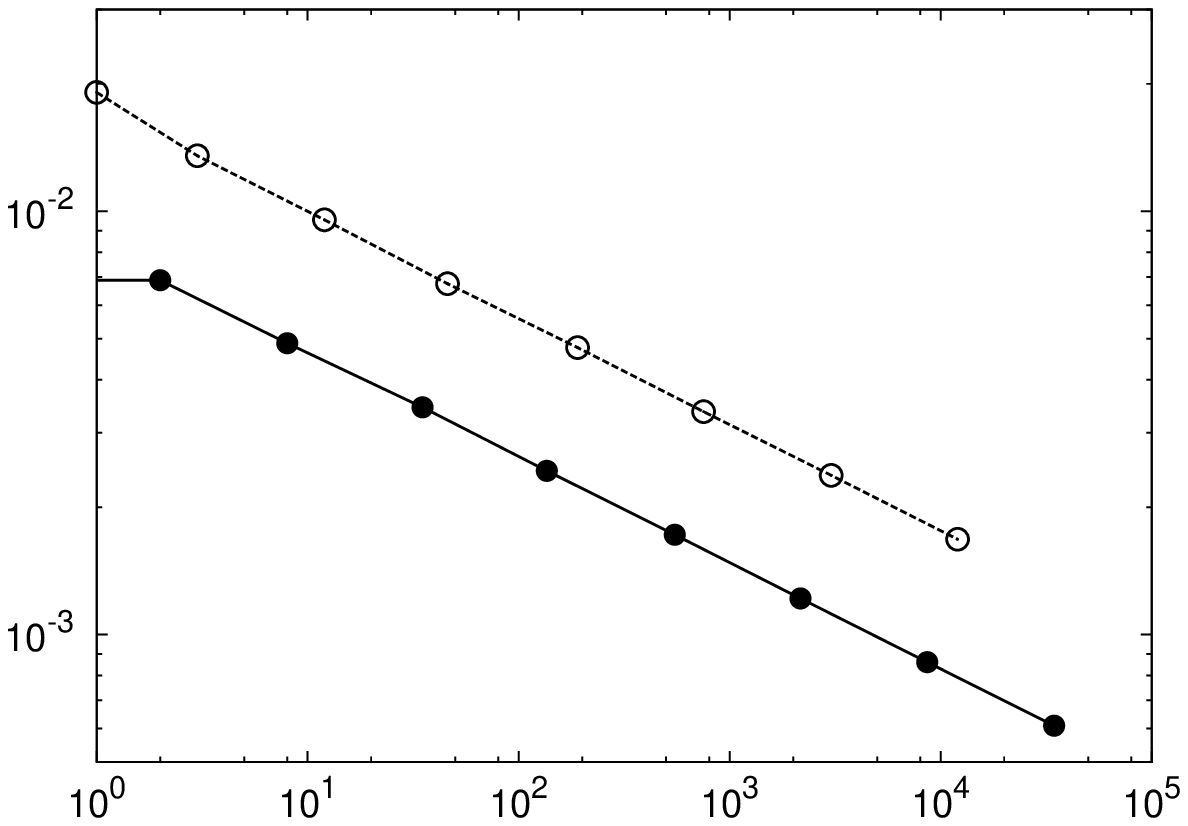} \\[3ex]
$u_1$ & $u_2$\\[1ex]
\includegraphics[width=7cm,height=4cm]{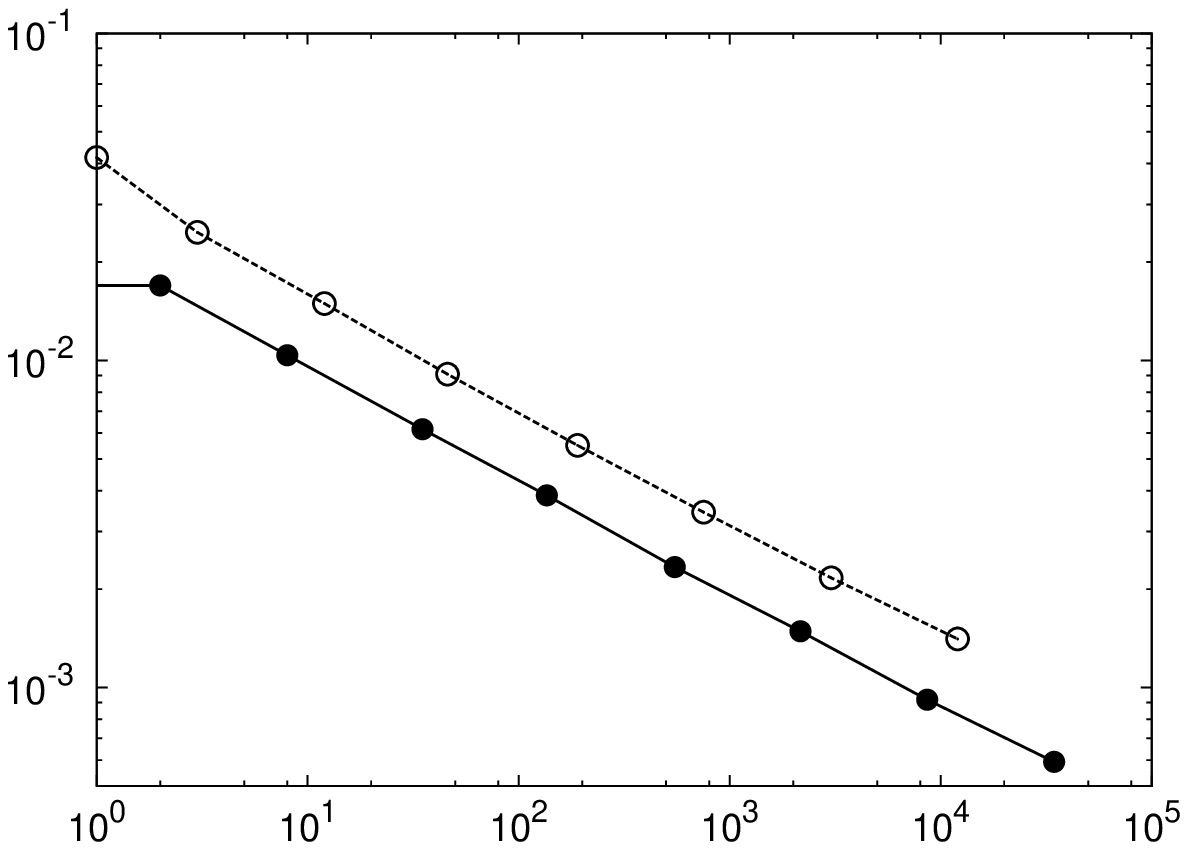} &
\includegraphics[width=7cm,height=4cm]{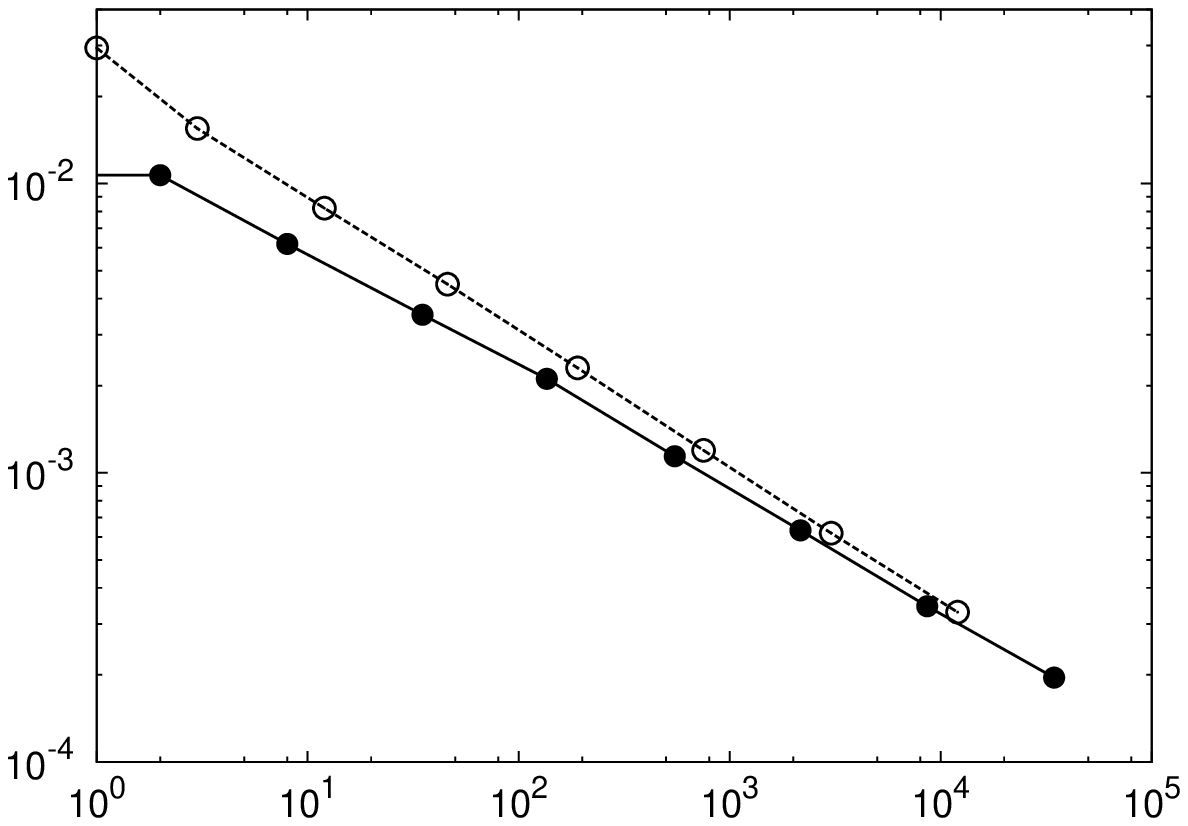} \\[3ex]
$\rho_1$ & $\rho_2$ \\[1ex]
\includegraphics[width=7cm,height=4cm]{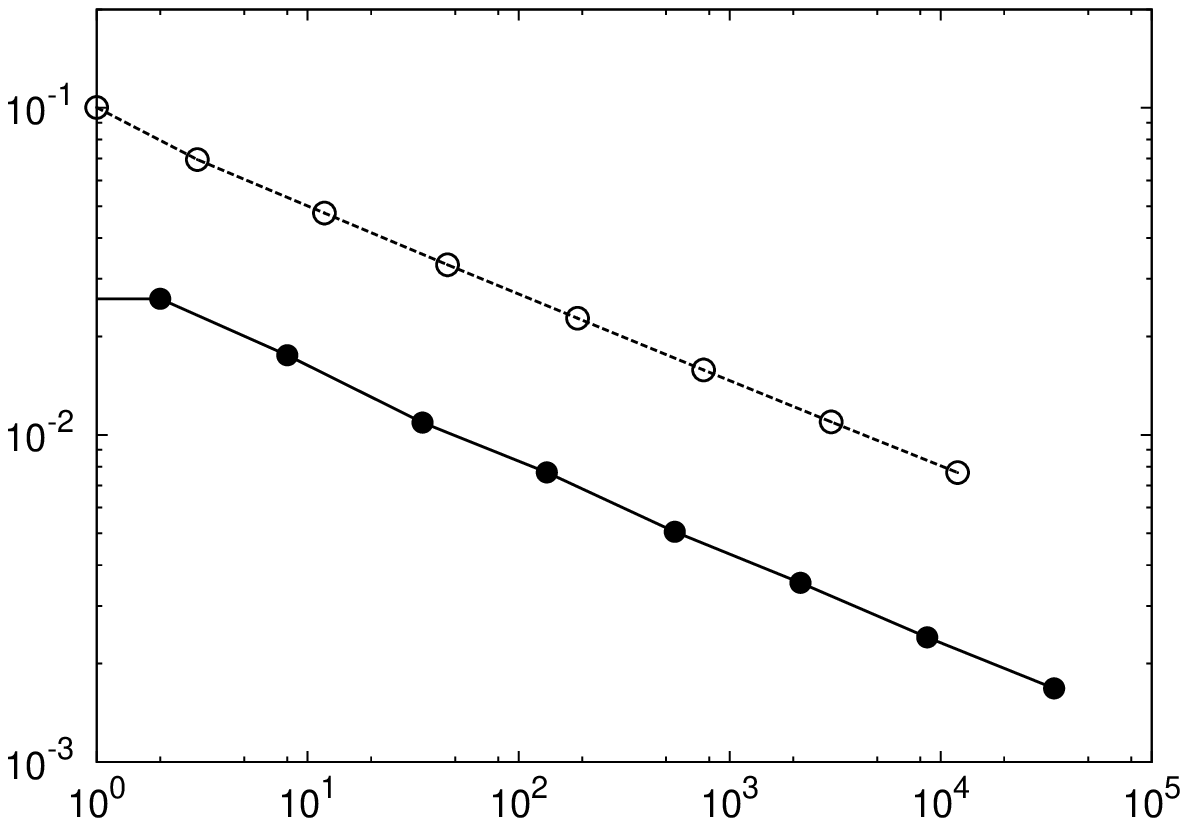} &
\includegraphics[width=7cm,height=4cm]{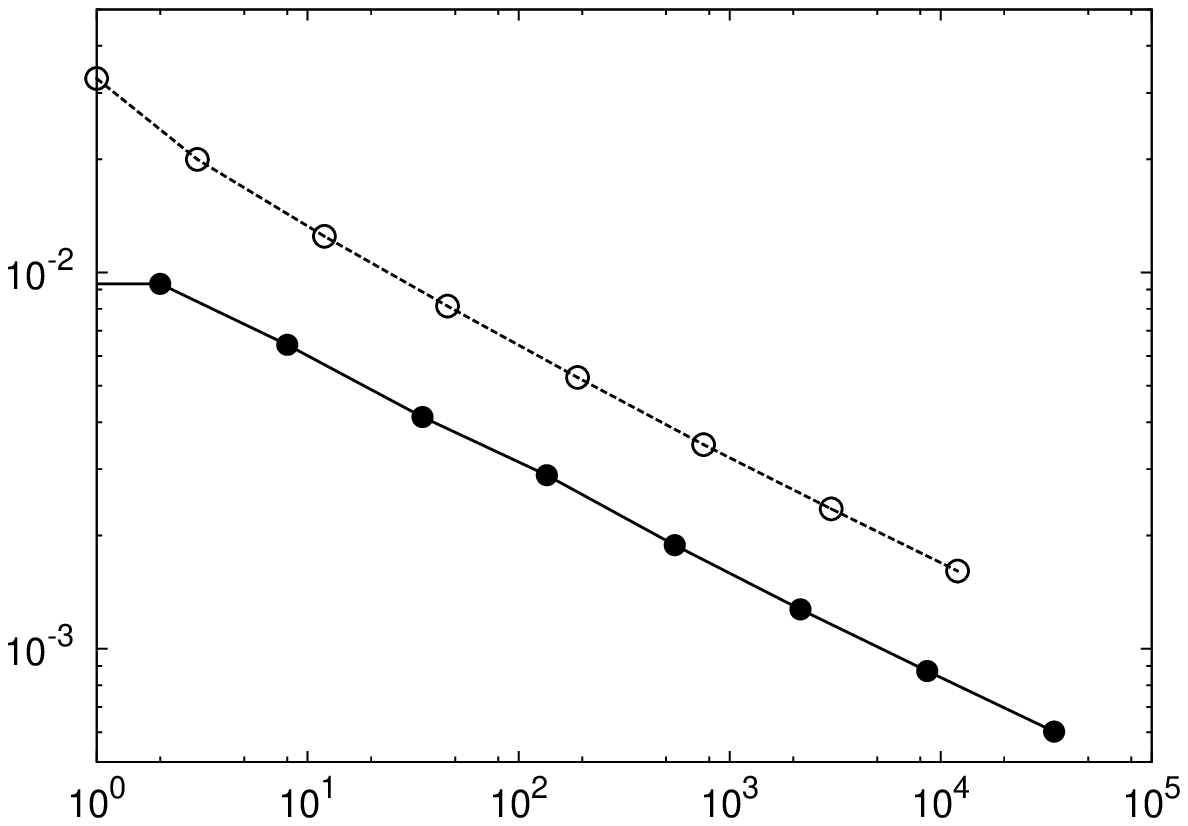}\\[3ex]
$p_1$  & $p_2$\\[1ex]
\includegraphics[width=7cm,height=4cm]{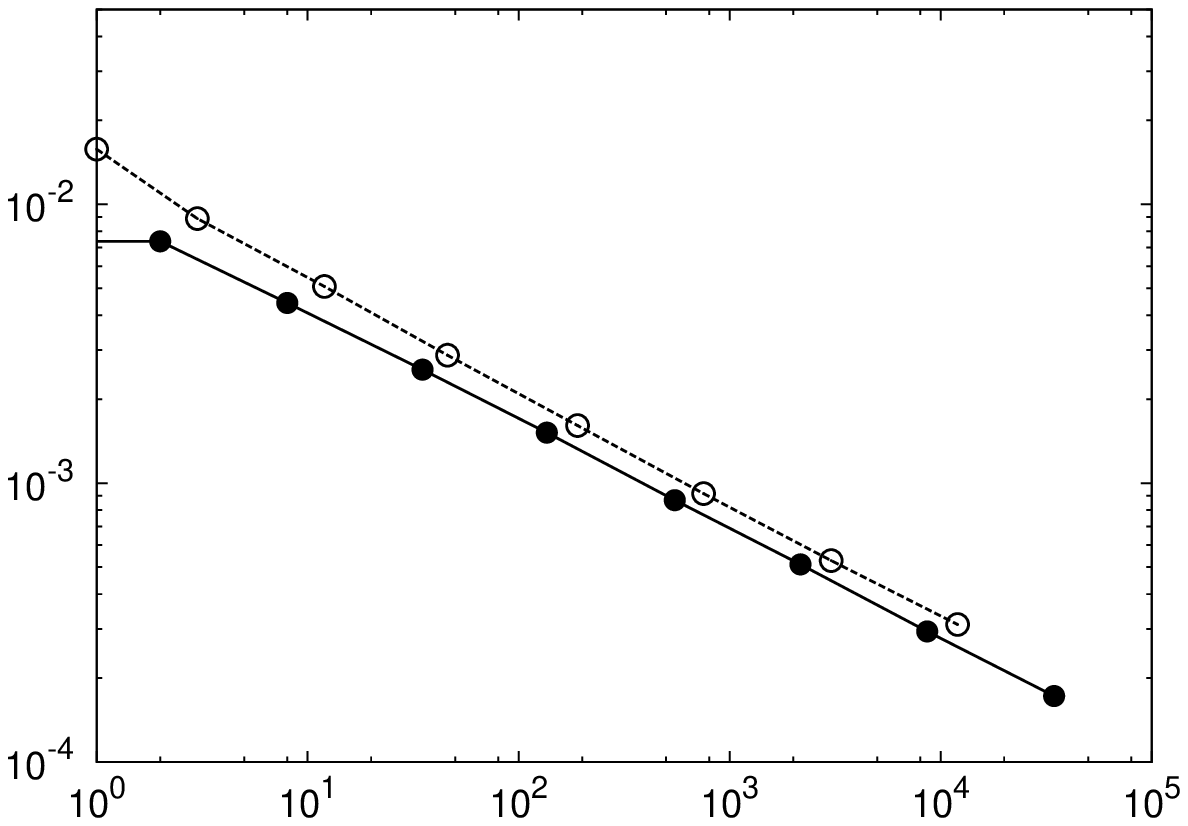} &
\includegraphics[width=7cm,height=4cm]{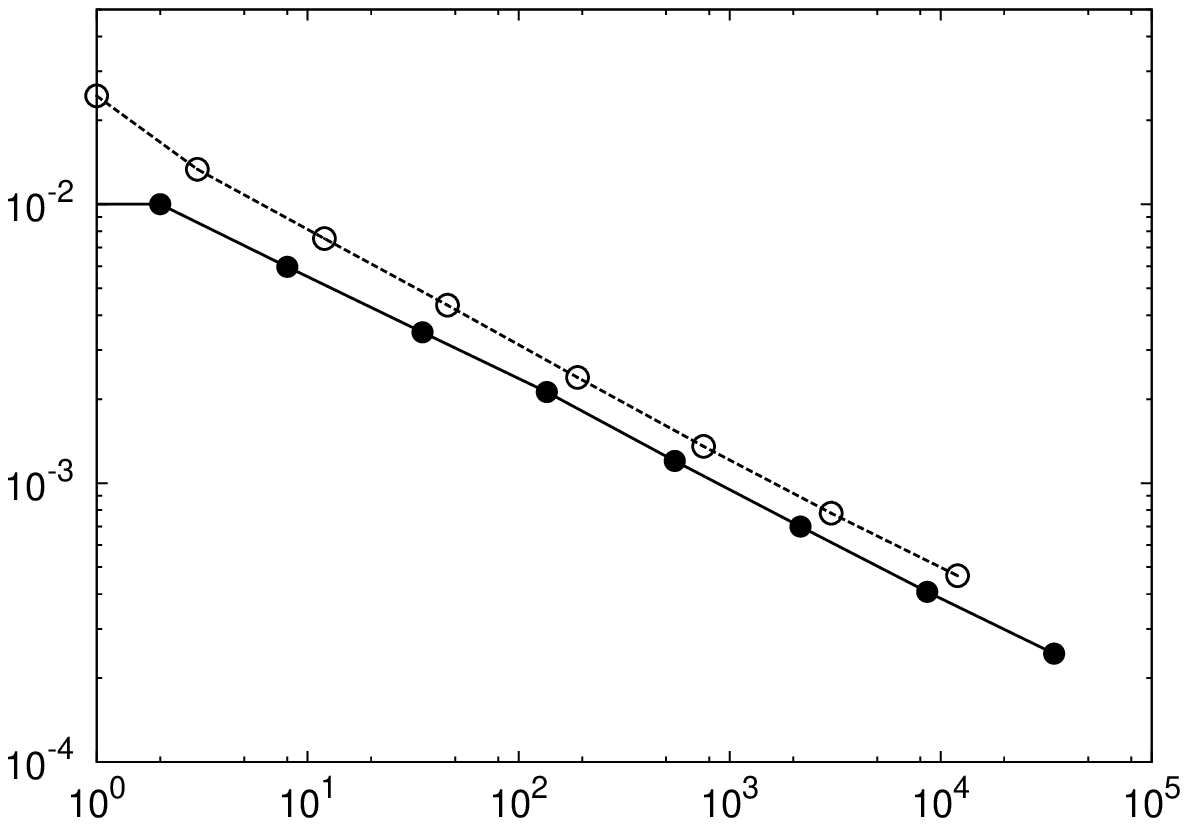}
\end{tabular}
\protect \parbox[t]{13cm}{\caption{Test-case 2: $L^1$-Error with respect to $\Delta x$ for the relaxation scheme and $L^1$-Error with respect to computational cost (in seconds) for the relaxation scheme (straight line) and Rusanov's scheme (dashed line).\label{Figcase2bis}}}

\end{center}
\end{figure}

%%%%%%%%%%%%%%%%%%%%%%%%%% TC3 %%%%%%%%%%%%%%%%%%%%%%%%%%%%%%%%%%

\subsection{Results for Test-case 3}

\begin{table}[ht!]
\centering
\begin{tabular}{|ggggg|}
\hline
		& Region $L$ 	& Region $-$	& Region $+$	& Region $R$		 \\
\hline
$\alpha_1$ 	&$0.2$		&$0.2$		&$0.5$		&$0.5$			\\
$\rho_1$	&$0.99988$	&$0.0219$	&$0.0219$	&$0.99988$		  \\
$u_1$		&$-1.99931$	&$0.0$		&$0.0$		&$1.99931$			\\
$p_1$		&$0.4$		&$0.0019$	&$0.0019$	&$0.4$			\\
$\rho_2$	&$0.99988$	&$0.0219$	&$0.0219$	&$0.99988$				\\
$u_2$		&$-1.99931$	&$0.0$		&$0.0$		&$1.99931$			\\
$p_2$		&$0.4$		&$0.0019$	&$0.0019$	&$0.4$			\\
\hline
\end{tabular}
 \protect \parbox[t]{13cm}{\caption{Test-case 3: Left, right and intermediate states of the exact solution.\label{Table_TC3}}}
\end{table}

This test was also taken from \cite{TT}. Both phases consist of two symmetric rarefaction waves and a stationary $u_2$-contact discontinuity. As the region between the rarefaction waves is close to vacuum, this test-case is useful to assess the pressure positivity property. Note that the positivity of the pressures is expected here since both phases follow an ideal gas \eos. The results are given in Figure \ref{Figcase3}. We can see that the computed pressures are positive for all the schemes. In addition, all the schemes have similar results, except for the resolution of the phase fraction discontinuity which appears to be very diffused by Rusanov's scheme while it is \textbf{exactly captured} by the other three schemes. For the relaxation scheme, this is a property satisfied by the relaxation Riemann solver for this type of discontinuities. It is also naturally satisfied by the Godunov-type scheme and by the HLLC scheme (see \cite{TT}).

\begin{figure}[ht!]
\begin{center}
\begin{tabular}{cc}
Wave structure & $\alpha_1$ \\[1ex]
\includegraphics[width=7cm,height=4cm]{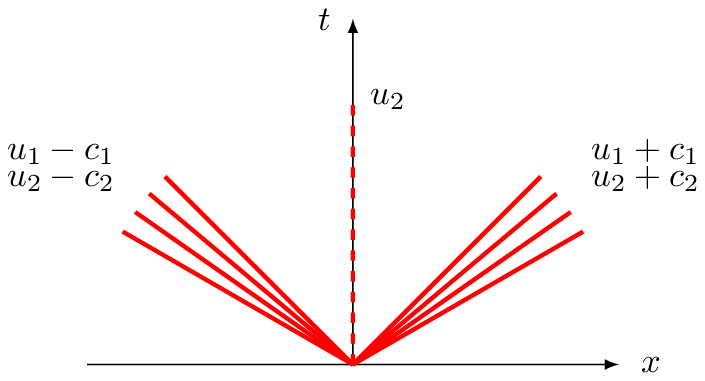}&
\includegraphics[width=7cm,height=4cm]{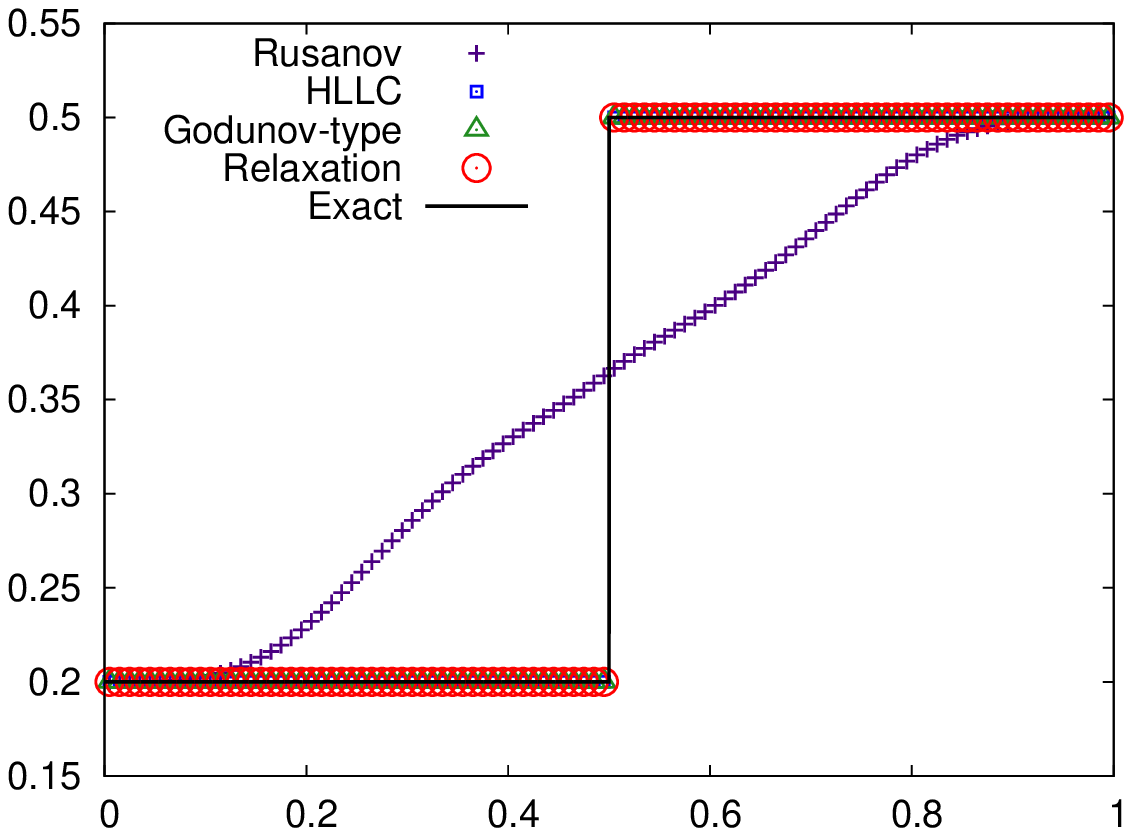} \\[3ex]
$u_1$ & $u_2$\\[1ex]
\includegraphics[width=7cm,height=4cm]{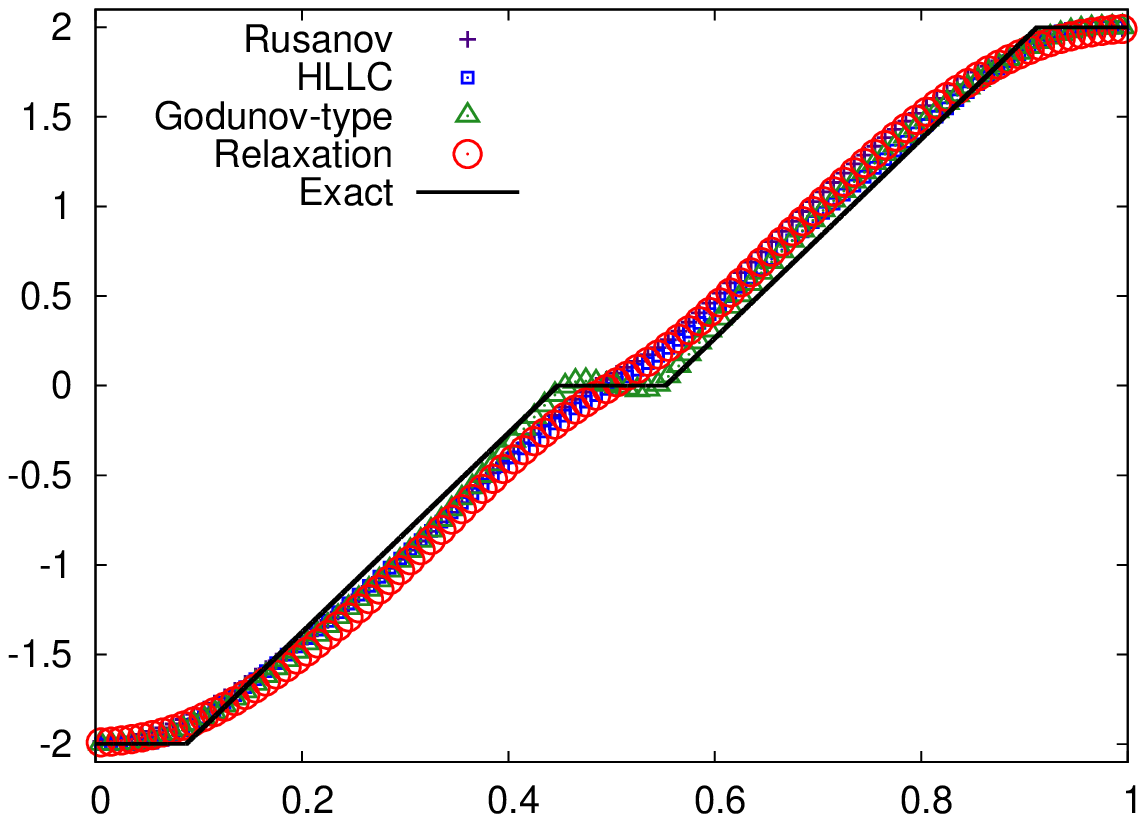} &
\includegraphics[width=7cm,height=4cm]{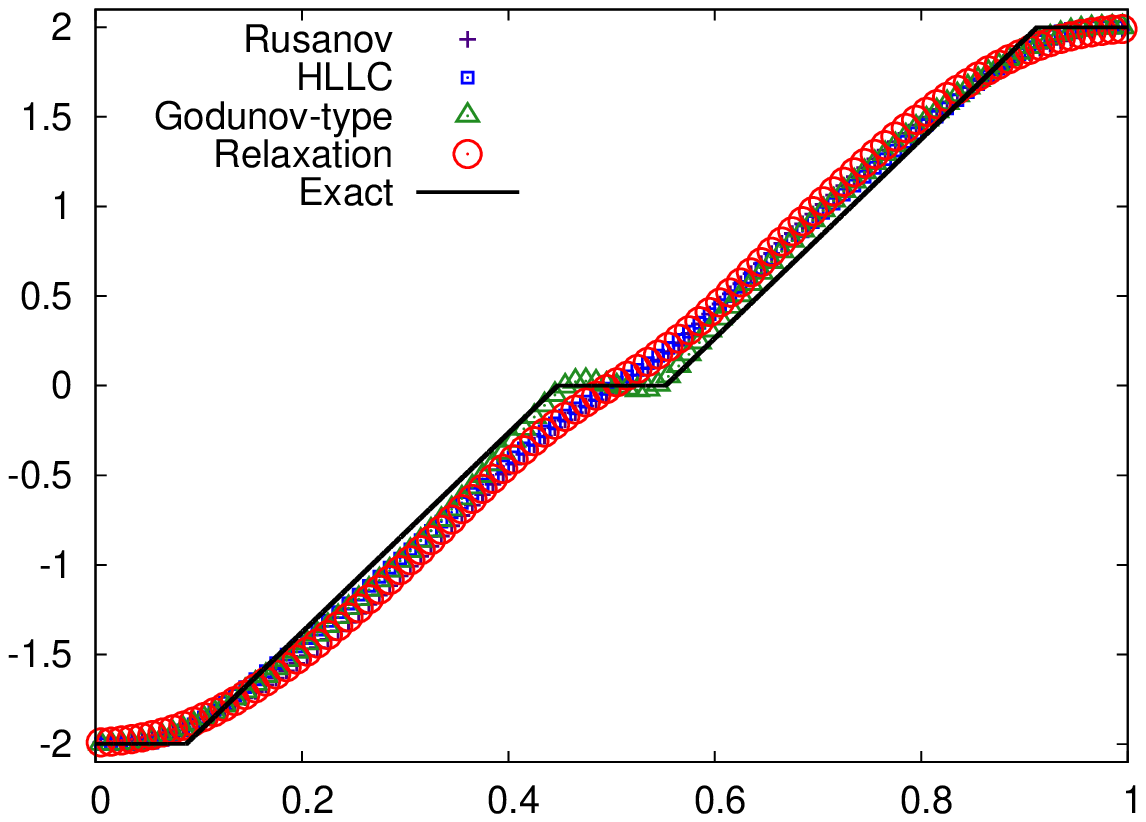} \\[3ex]
$\rho_1$ & $\rho_2$ \\[1ex]
\includegraphics[width=7cm,height=4cm]{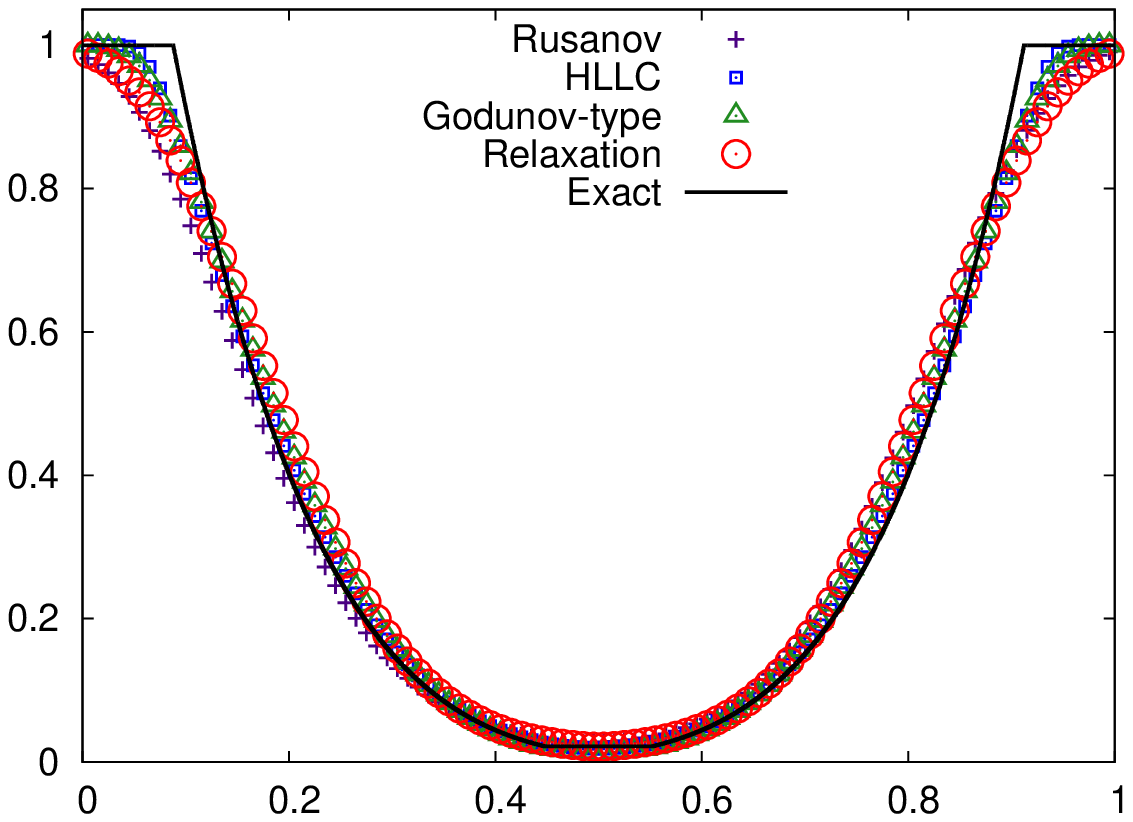} &
\includegraphics[width=7cm,height=4cm]{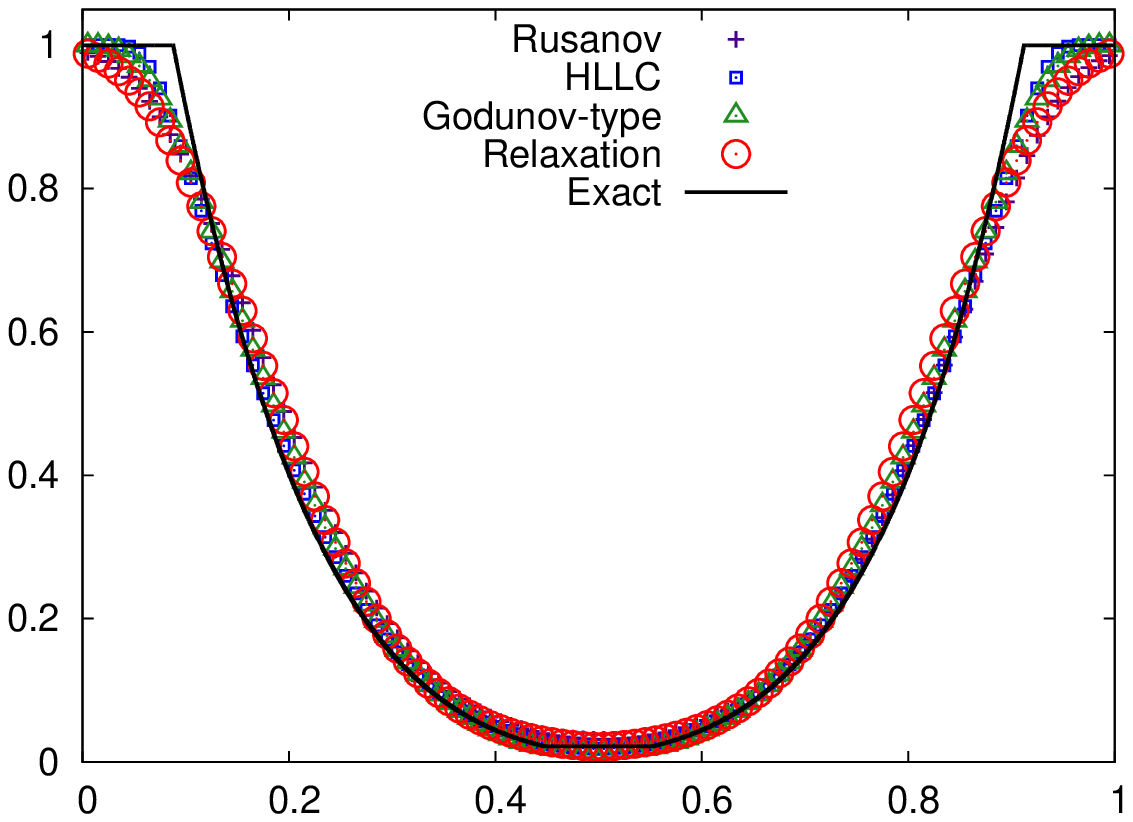}\\[3ex]
$p_1$  & $p_2$\\[1ex]
\includegraphics[width=7cm,height=4cm]{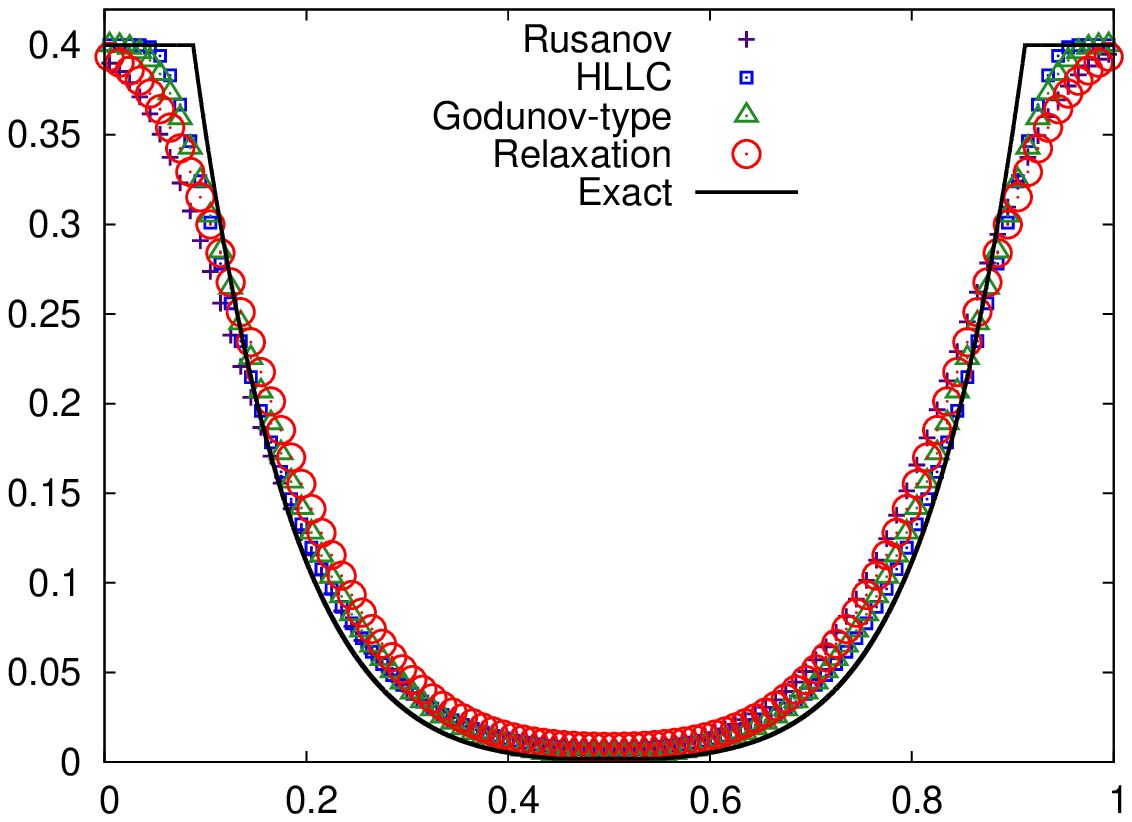} &
\includegraphics[width=7cm,height=4cm]{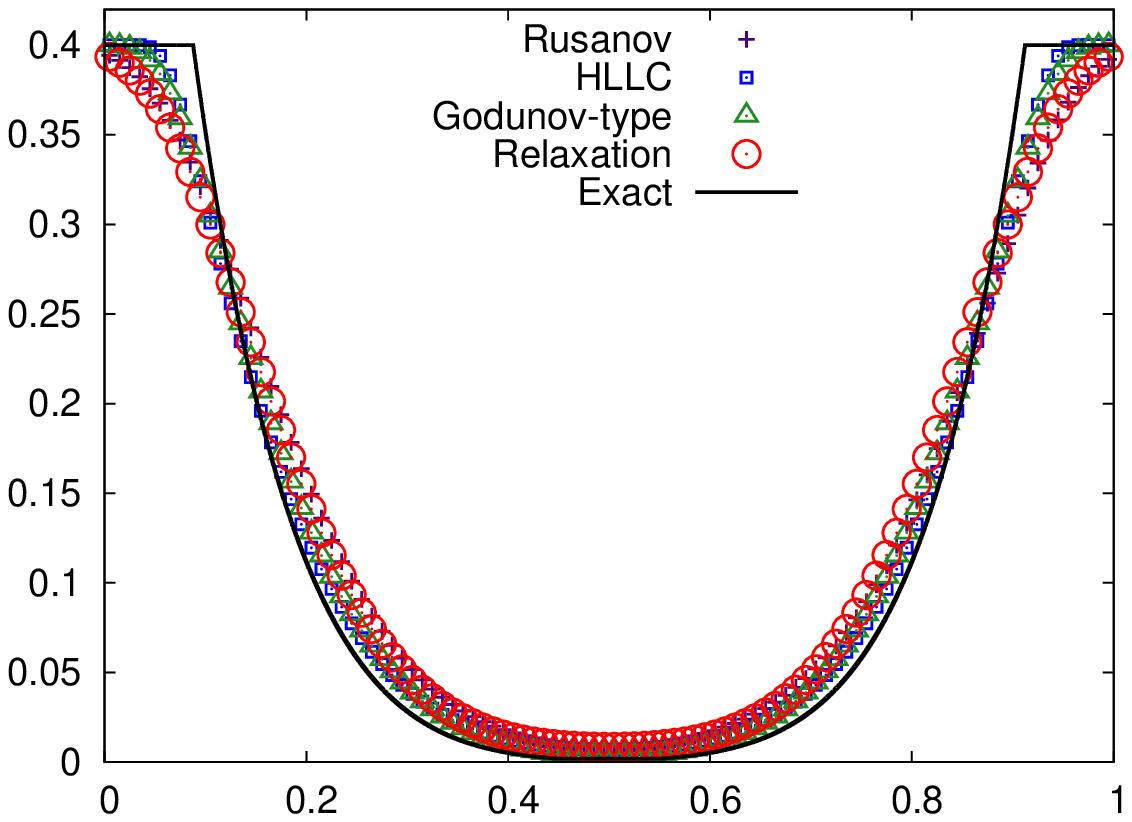}
\end{tabular}
\protect \parbox[t]{13cm}{\caption{Test-case 3: Structure of the solution and space variations of the physical variables at the final time $T_{\rm max}=0.15$. Mesh size: $100$ cells.\label{Figcase3}}}
\end{center}
\end{figure}

%%%%%%%%%%%%%%%%%%%%%%%%%% TC4 %%%%%%%%%%%%%%%%%%%%%%%%%%%%%%%%%%
\subsection{Results for Test-case 4}

\begin{table}[ht!]
\centering
\begin{tabular}{|hhhhhh|}
\hline
		& Region $L$ 	& Region $-$	& Region $+$	& Region $R*$		& Region $R$ \\
\hline
$\alpha_1$ 	&$1.0$		&$1.0$		&$0.4$		&$0.4$		&$0.4$	\\
$\rho_1$	&$1.6$		&$2.0$		&$1.84850$	&$2.03335$	&$1.62668$	  \\
$u_1$		&$0.80311$	&$0.4$		&$0.91147$	&$0.91147$	&$0.55623$	\\
$p_1$		&$1.3$		&$2.6$		&$2.05277$	&$2.05277$	&$1.02638$		\\
$\rho_2$	&$-$		&$-$		&$4.0$		&$4.0$		&$7.69667$	\\
$u_2$		&$-$		&$-$		&$0.1$		&$0.1$		&$0.74797$	\\
$p_2$		&$-$		&$-$		&$2.45335$	&$2.45335$	&$6.13338$	\\
\hline
\end{tabular} 
 \protect \parbox[t]{13cm}{\caption{Test-case 4: Left, right and intermediate states of the exact solution.\label{Table_TC4}}}
\end{table}

We now consider a Riemann problem in which one of the two phases vanishes in one of the initial states, which means that the corresponding phase fraction $\alpha_1$ or $\alpha_2$ is equal to zero. For this kind of Riemann problem, the $u_2$-contact separates a mixture region where the two phases coexist from a single phase region with the remaining phase. Various examples of such problems were introduced in \cite{SWK}, \cite{SA,ASvipi} or \cite{TT}.

\medskip
The solution is composed of a $\{u_1-c_1\}$-shock wave in the left-hand side (LHS) region where only phase 1 is present. This region is separated by a $u_2$-contact discontinuity from the right-hand side (RHS) region where the two phases are mixed. In this RHS region, the solution is composed of a $u_1$-contact discontinuity, followed by a $\{u_2+c_2\}$-rarefaction wave and a $\{u_1+c_1\}$-shock (see Figure \ref{Figcase4}).

\medskip
In practice, the numerical method requires values of $\alpha_{1,L}$ and $\alpha_{1,R}$ that lie strictly in the interval $(0,1)$. Therefore, in the numerical implementation, we take $\alpha_{1,L}= 1- 10^{-4}$. The aim here is to give a qualitative comparison between the numerical approximation and the exact solution. Moreover, there is theoretically no need to specify left initial values for the phase 2 quantities since this phase is not present in the LHS region. For the sake of the numerical simulations however, one must provide such values. We choose to set $\rho_{2,L}$, $u_{2,L}$ and $p_{2,L}$ to the values on the right of the $u_2$-contact discontinuity, which is coherent with the preservation of the Riemann invariants of this wave, and avoids the formation of fictitious acoustic waves for phase 2 in the LHS region.  For the relaxation scheme, this choice enables to avoid oscillations of phase 2 quantities in the region where phase 2 is not present. However, some tests have been conducted that assess that taking other values of $(\rho_{2,L},u_{2,L},p_{2,L})$ has little impact on the phase 1 quantities as well as on the phase 2 quantities where this phase is present.

\medskip
We can see that for the same level of refinement, the relaxation method, the Godunov-type method and the HLLC method are more accurate than Rusanov's scheme, which can be seen especially for phase 1. As regards the region where phase 2 does not exist, we can see that the three other methods are much more stable than Rusanov's scheme. Indeed, theses schemes behave better than Rusanov's scheme when it comes to divisions by small values of $\alpha_2$, since the solution approximated by Rusanov's scheme develops quite large values.

\begin{figure}[ht!]
\begin{center}
\begin{tabular}{cc}
Wave structure & $\alpha_1$ \\[1ex]
\includegraphics[width=7cm,height=4cm]{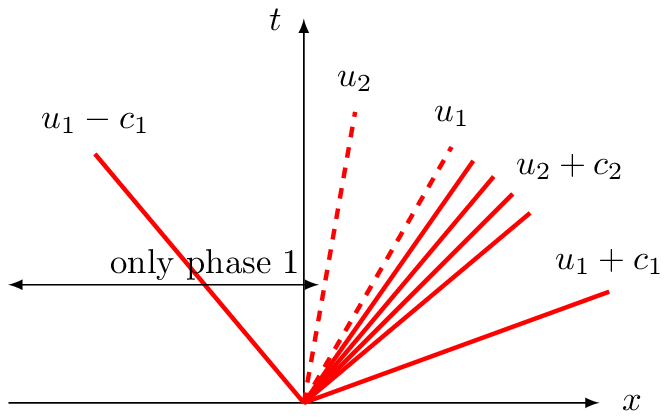}&
\includegraphics[width=7cm,height=4cm]{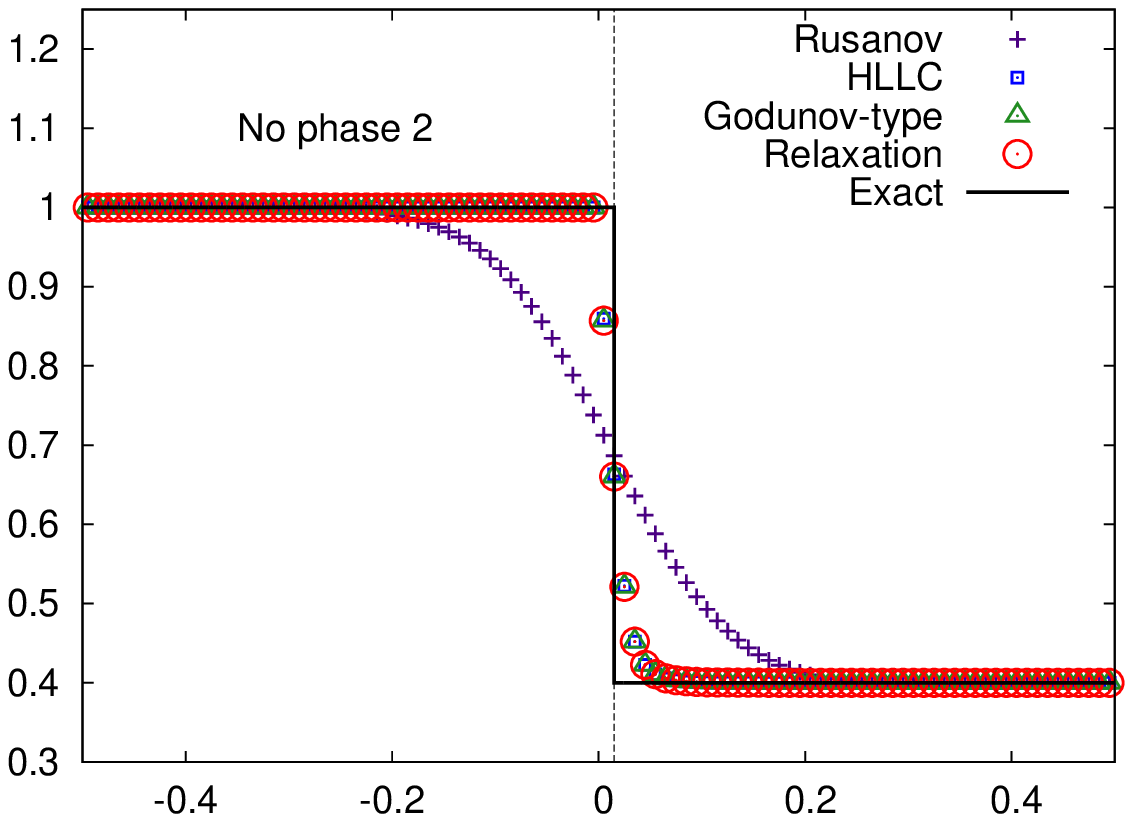} \\[3ex]
$u_1$ & $u_2$\\[1ex]
\includegraphics[width=7cm,height=4cm]{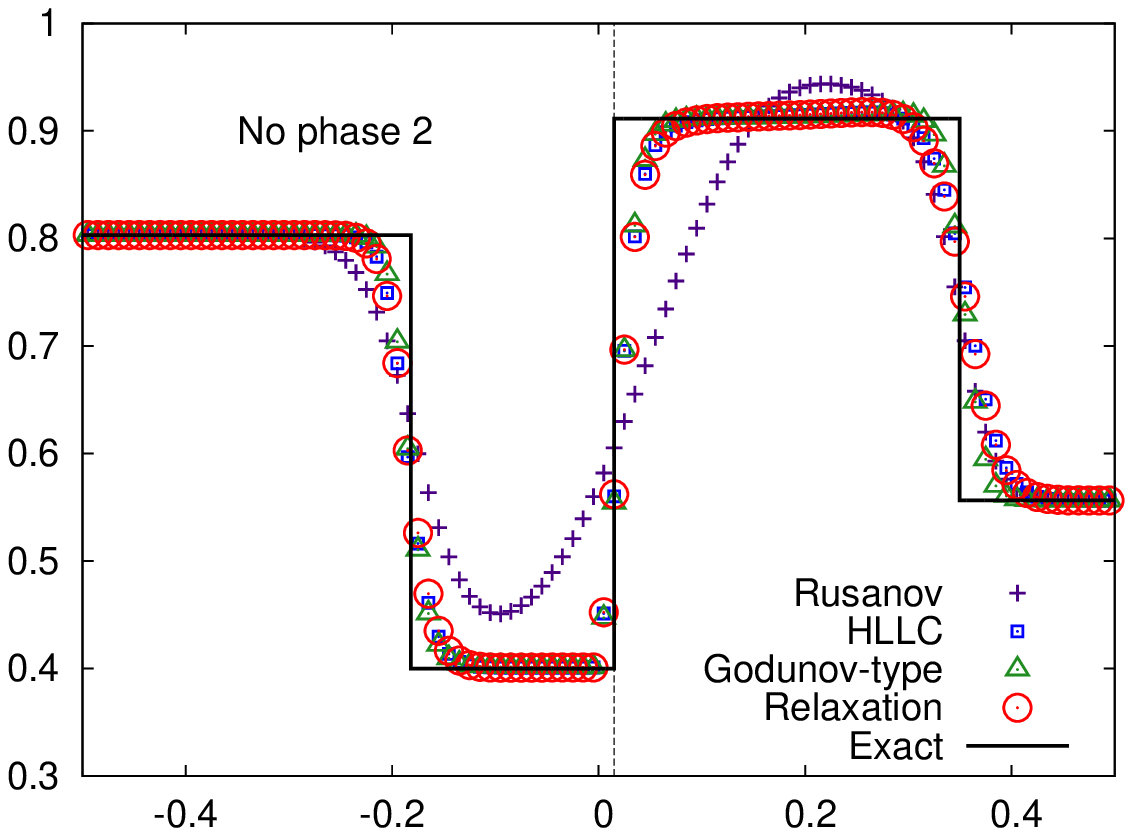} &
\includegraphics[width=7cm,height=4cm]{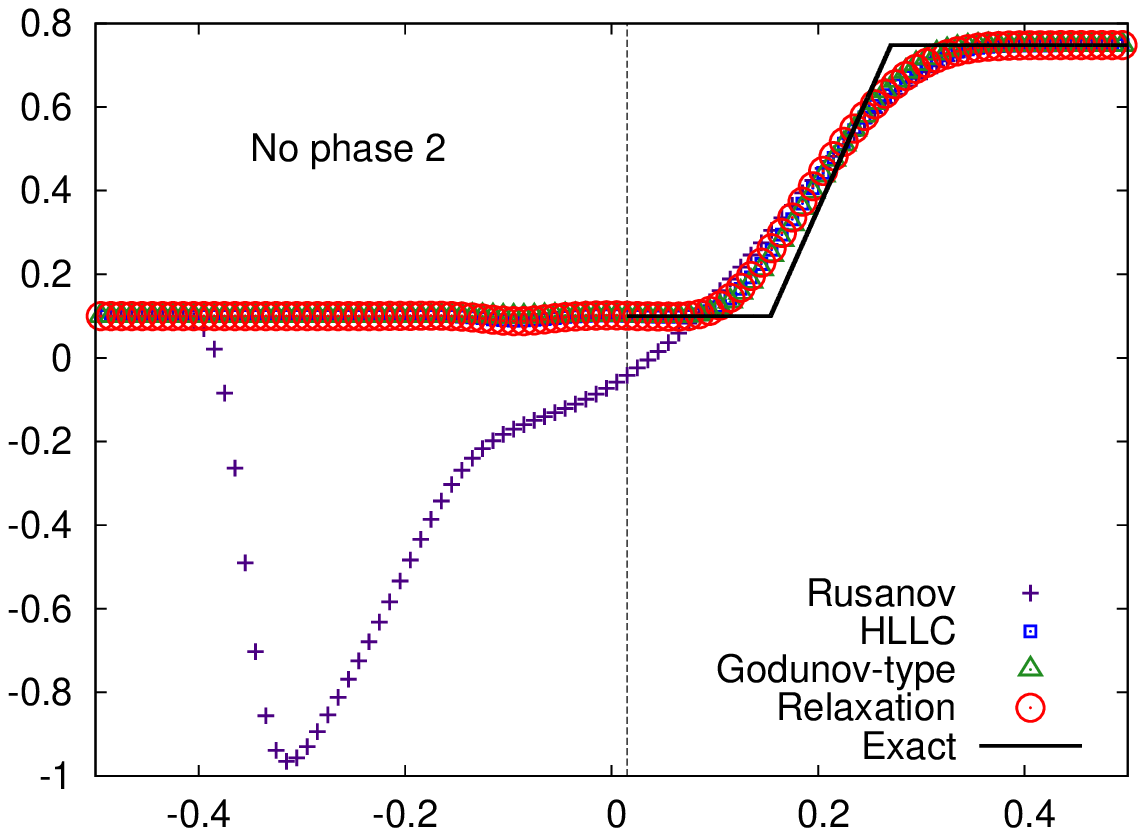} \\[3ex]
$\rho_1$ & $\rho_2$ \\[1ex]
\includegraphics[width=7cm,height=4cm]{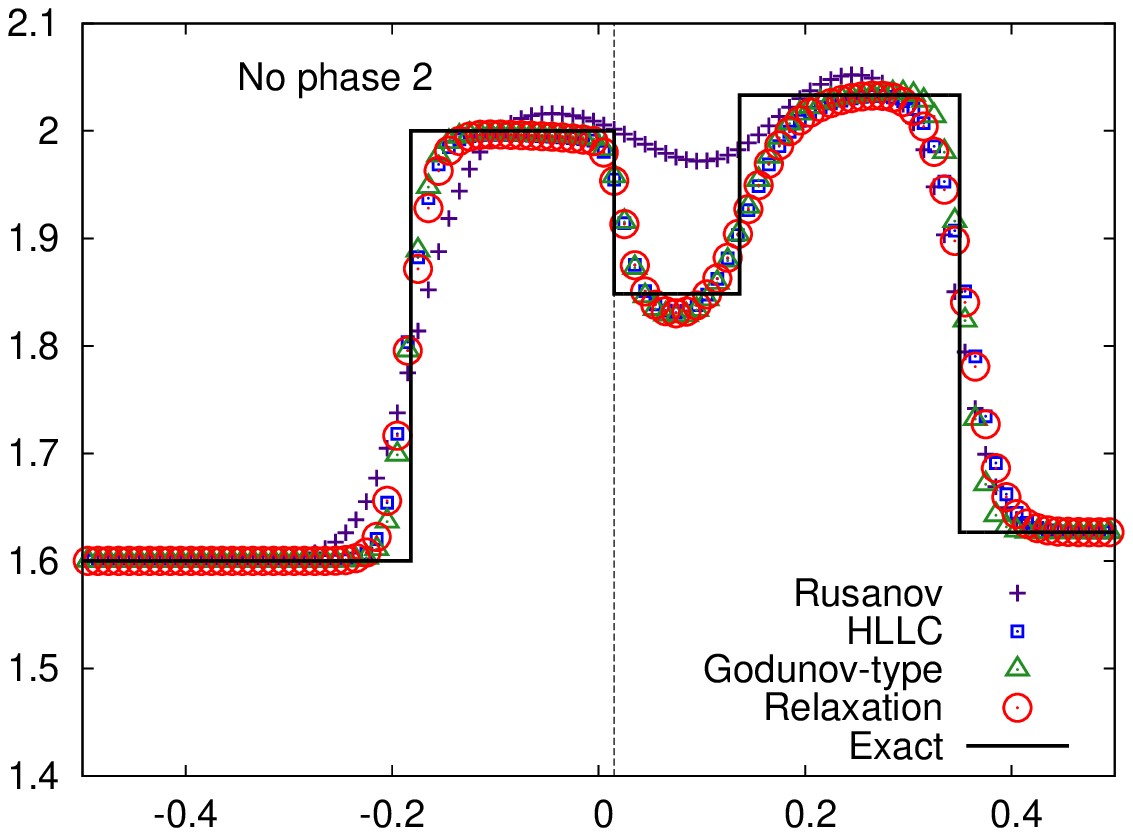} &
\includegraphics[width=7cm,height=4cm]{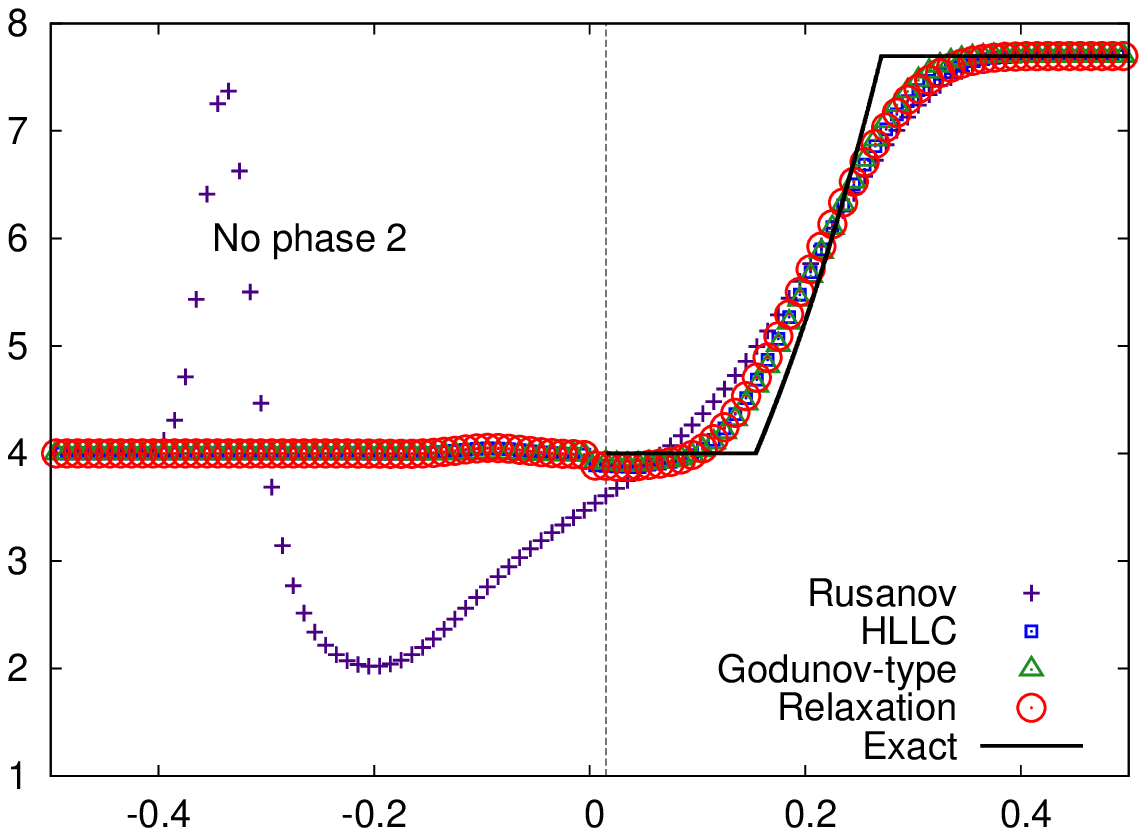}\\[3ex]
$p_1$  & $p_2$\\[1ex]
\includegraphics[width=7cm,height=4cm]{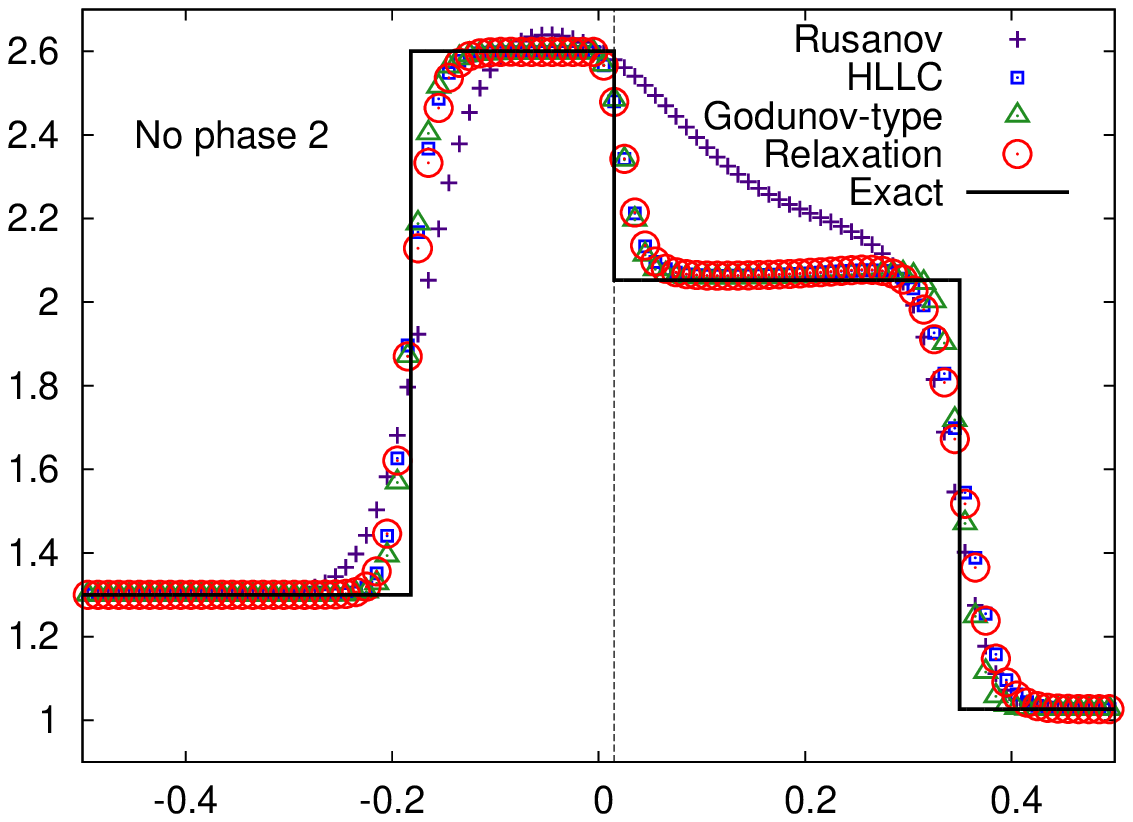} &
\includegraphics[width=7cm,height=4cm]{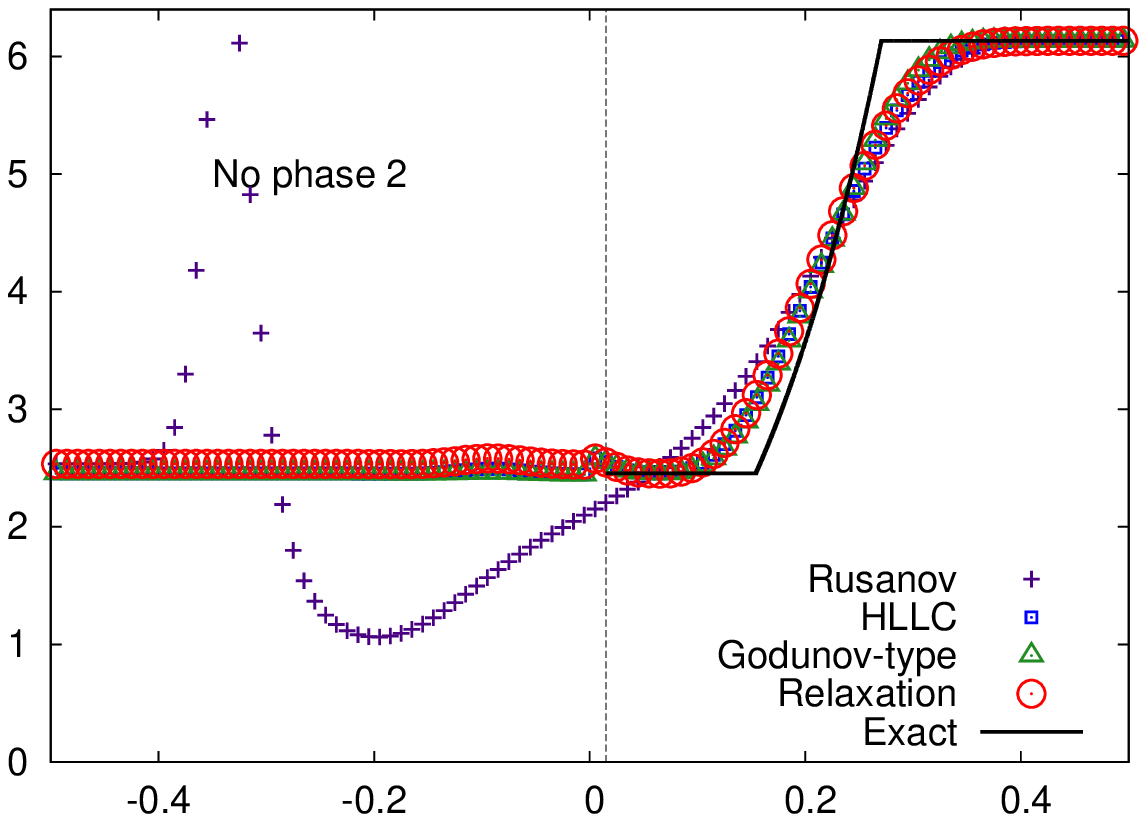}
\end{tabular}
\protect \parbox[t]{13cm}{\caption{Test-case 4: Structure of the solution and space variations of the physical variables at the final time $T_{\rm max}=0.15$. Mesh size: $100$ cells.\label{Figcase4}}}
\end{center}
\end{figure}
%%%%%%%%%%%%%%%%%%%%%%%%%% TC5 %%%%%%%%%%%%%%%%%%%%%%%%%%%%%%%%%%

\subsection{Results for Test-case 5}
\label{test_5}

\begin{table}[ht!]
\centering
\begin{tabular}{|ggggg|}
\hline
		& Region $L$ 	& Region $-$	& Region $+$	& Region $R$		 \\
\hline
$\alpha_1$ 	&$1.0$		&$1.0$		&$0.0$		&$0.0$			\\
$\rho_1$	&$1.6$		&$2.0$		&$-$		&$-$		  \\
$u_1$		&$1.79057$	&$1.0$		&$-$		&$-$			\\
$p_1$		&$5.0$		&$10.0$		&$-$		&$-$			\\
$\rho_2$	&$-$		&$-$		&$2.0$		&$2.67183$				\\
$u_2$		&$-$		&$-$		&$1.0$		&$1.78888$			\\
$p_2$		&$-$		&$-$		&$10.0$		&$15.0$			\\
\hline
\end{tabular}
 \protect \parbox[t]{13cm}{\caption{Test-case 5: Left, right and intermediate states of the exact solution.\label{Table_TC5}}}
\end{table}

The last test-case considers the coupling between two pure phases. A left region, where only phase 1 exists ($\alpha_{1,L}=1$), is separated by a $u_2$-contact discontinuity from a right region, where only phase 2 is present ($\alpha_{1,R}=0$). In the existence region of phase, the solution is composed of a shock (phase 1) or a rarefaction wave (phase 2).

\medskip
In the numerical implementation, we set $\alpha_{1,L}= 1- 10^{-9}$ and $\alpha_{1,R}=10^{-9}$. In addition, in the LHS region, where phase 2 is absent, we choose to set $\rho_{2,L}$, $u_{2,L}$ and $p_{2,L}$ to the values on the right of the $u_2$-contact discontinuity \textit{i.e.} to the values $\rho_{2}^+$, $u_{2}^+$ and $p_{2}^+$. The symmetric choice is made for phase 1 in the RHS region: we set $\rho_{1,R}=\rho_{1}^-$, $u_{1,R}=u_{1}^-$ and $p_{1,R}=p_{1}^-$. Another choice could have been made for the initialization of the absent phase by imposing an instantaneous local thermodynamical equilibrium between the phases at time $t=0$. This would be coherent with the relaxation zero-th order source terms that are usually added to the model when simulating practical industrial configurations.

\medskip
The results are given in Figure \ref{Figcase5}. One can see that, in the LHS region, the quantities of the only present phase 1 are correctly approximated while the quantities of the vanishing phase 2 remain stable despite the division by small values of $\alpha_2$. The same observation can be made for the RHS region. On the contrary, Rusanov's scheme fails to approximate such a vanishing phase solution.

\begin{figure}[ht!]
\begin{center}
\begin{tabular}{cc}
Wave structure & $\alpha_1$ \\[1ex]
\includegraphics[width=7cm,height=4cm]{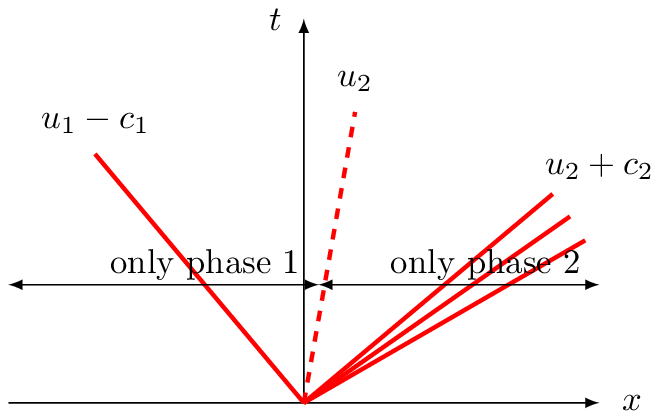}&
\includegraphics[width=7cm,height=4cm]{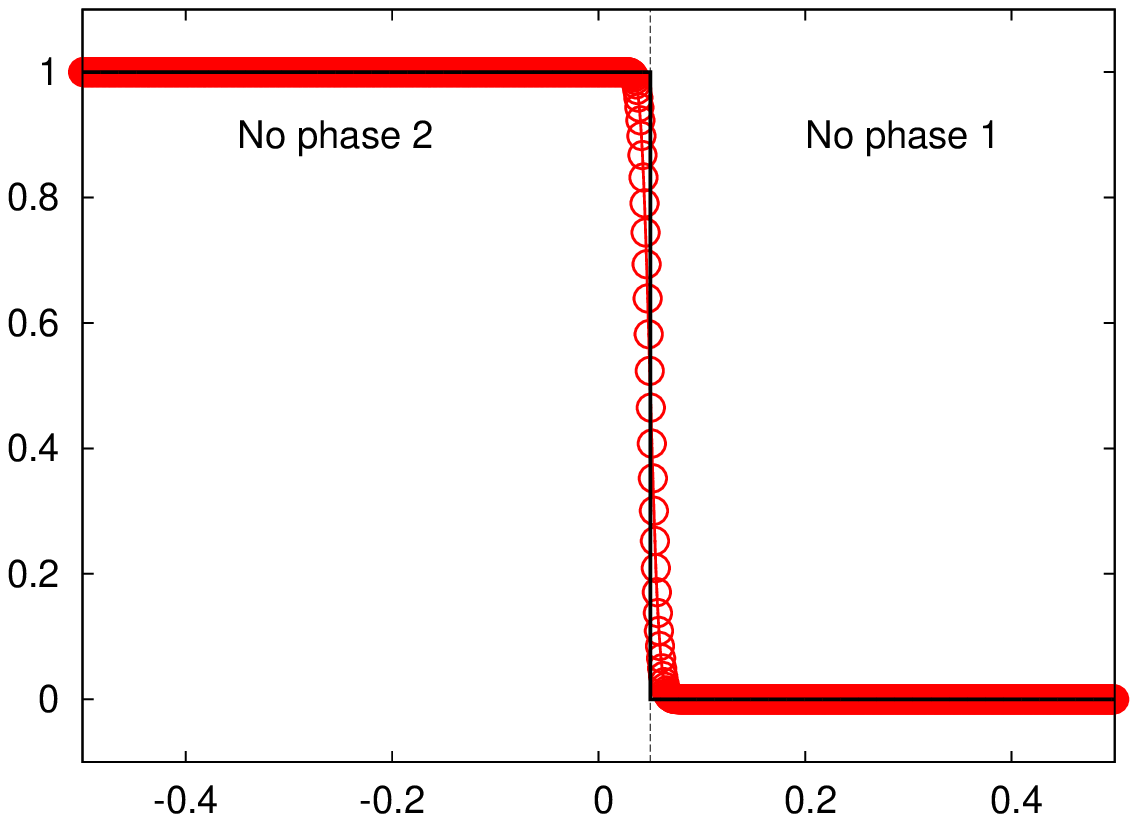} \\[3ex]
$u_1$ & $u_2$\\[1ex]
\includegraphics[width=7cm,height=4cm]{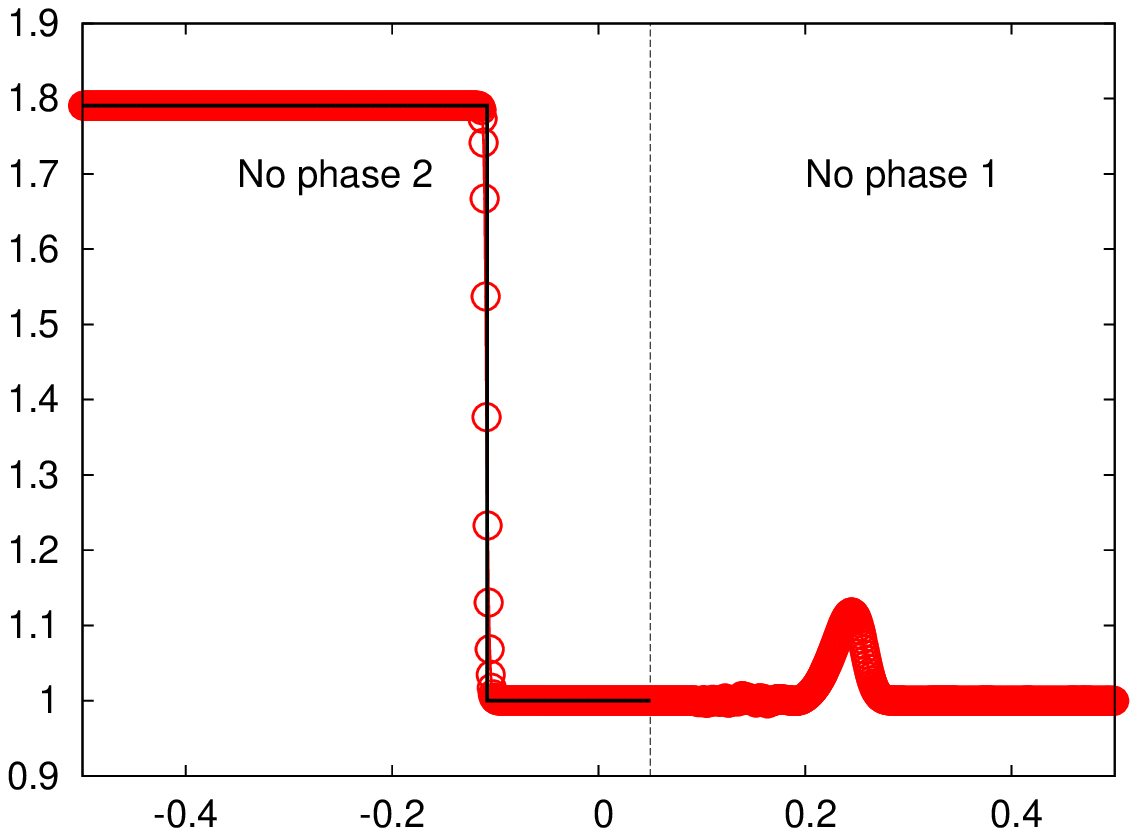} &
\includegraphics[width=7cm,height=4cm]{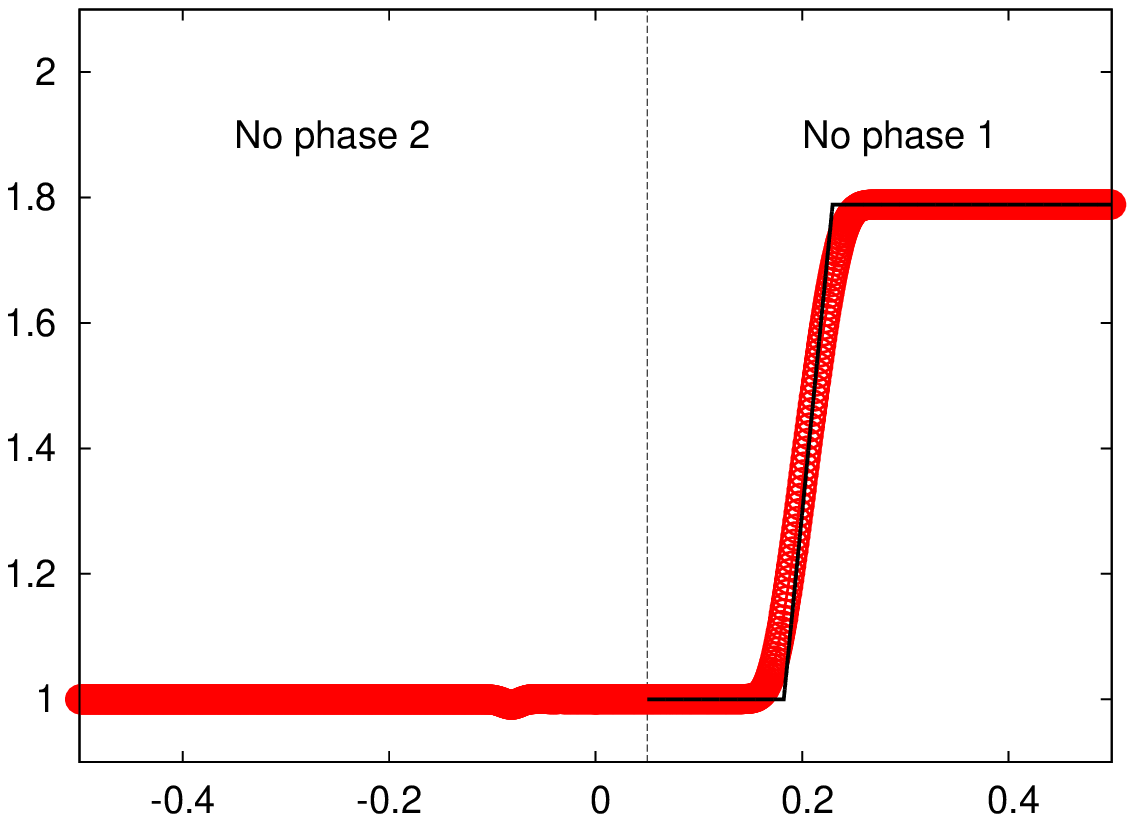} \\[3ex]
$\rho_1$ & $\rho_2$ \\[1ex]
\includegraphics[width=7cm,height=4cm]{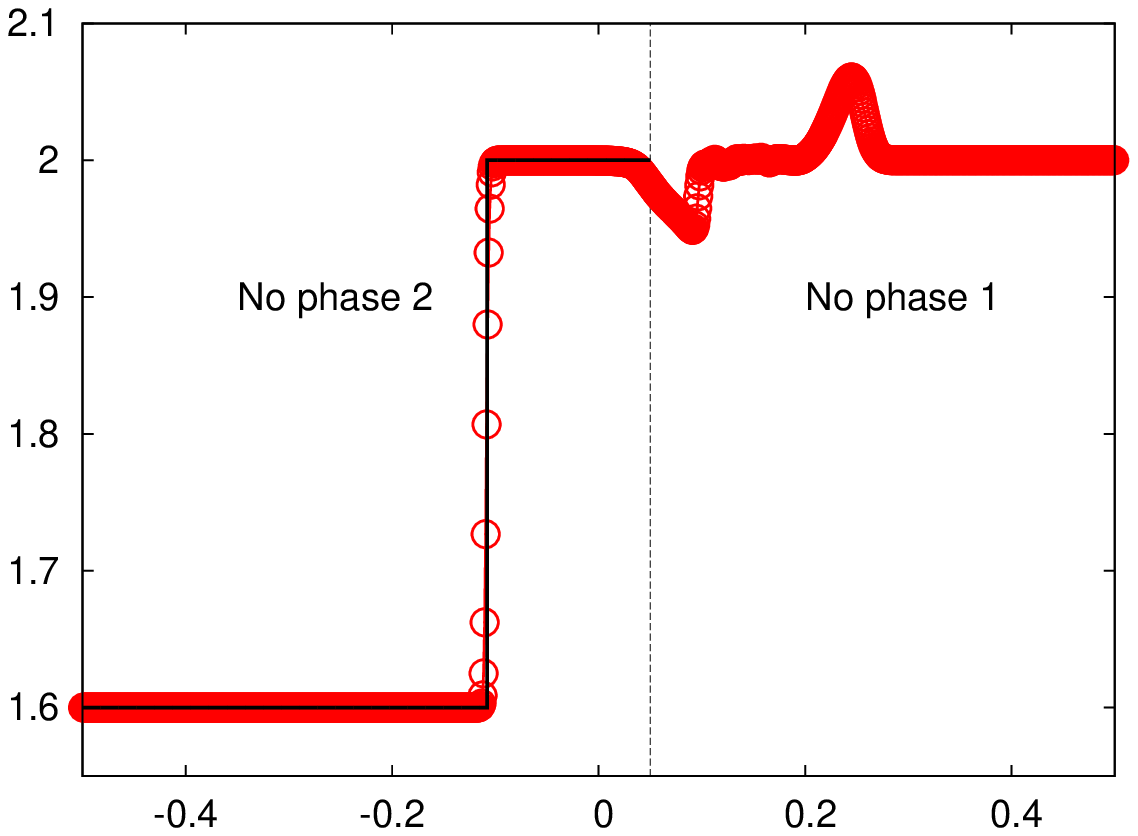} &
\includegraphics[width=7cm,height=4cm]{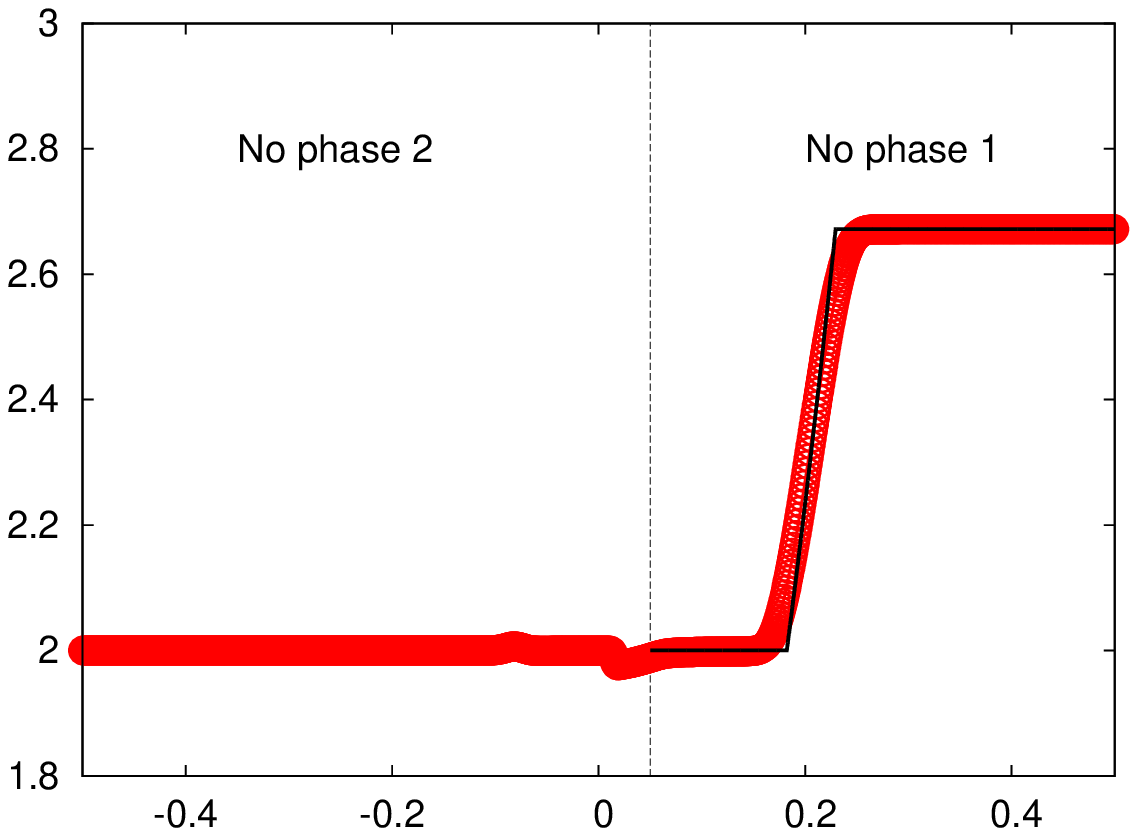}\\[3ex]
$p_1$  & $p_2$\\[1ex]
\includegraphics[width=7cm,height=4cm]{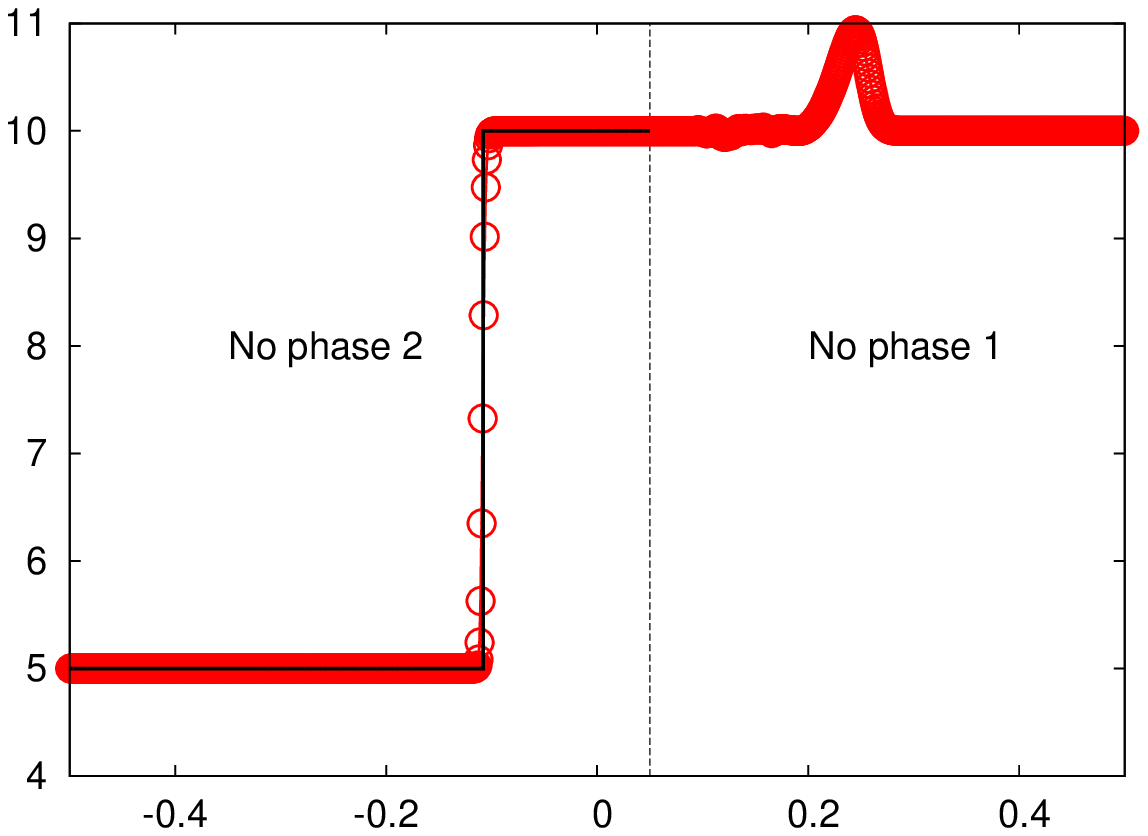} &
\includegraphics[width=7cm,height=4cm]{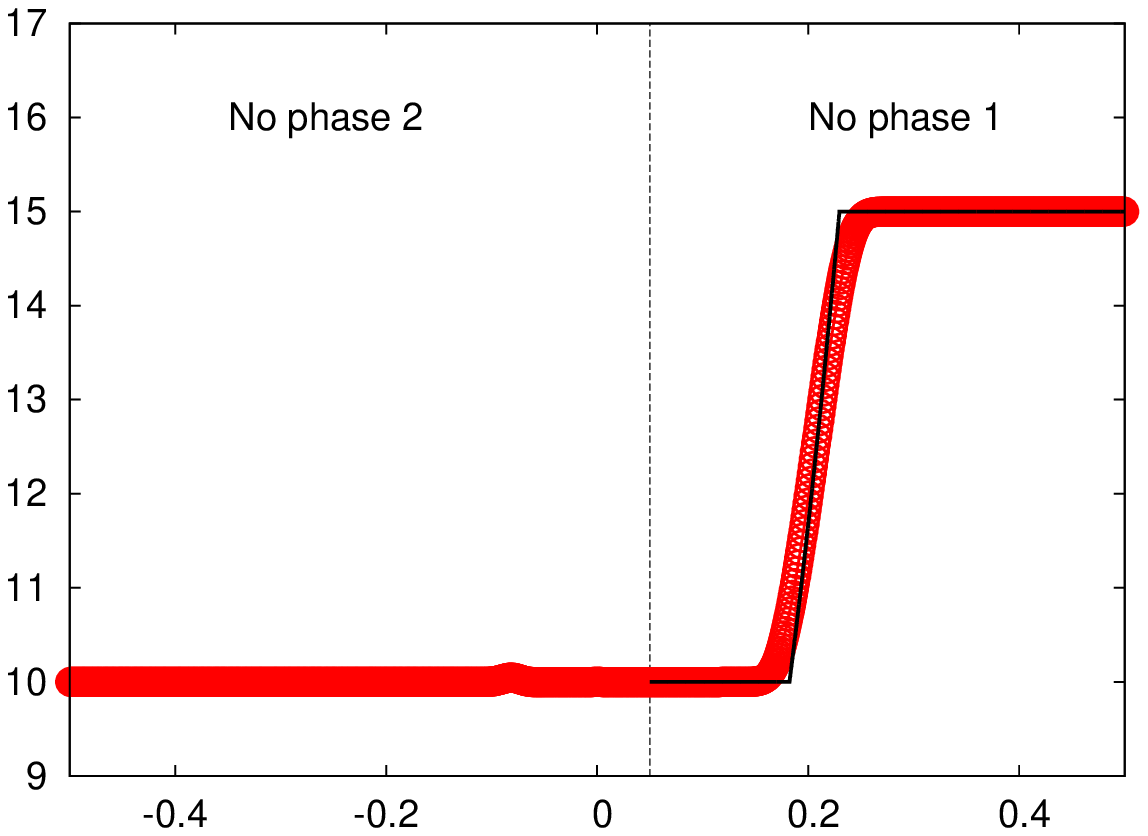}
\end{tabular}
\protect \parbox[t]{13cm}{\caption{Test-case 5: Structure of the solution and space variations of the physical variables at the final time $T_{\rm max}=0.05$. Mesh size: $1000$ cells. Straight line: exact solution, circles: relaxation scheme.\label{Figcase5}}}
\end{center}
\end{figure}

%\clearpage
%\newpage

\section{Conclusion}
\label{sec_conclusion}
The work performed in \cite{CHSS} and in the present paper provides an accurate and robust finite volume scheme for approximating the entropy dissipating weak solutions of the Baer-Nunziato two-phase flow model. The scheme relies on an exact Riemann solver for a relaxation approximation \emph{\`a la Suliciu} of the convective part of the Baer-Nunziato model.
To our knowledge, this is the only existing scheme for which the approximated phase fractions, phase densities and phase internal energies are proven to remain positive without any smallness assumption on the data or on the phase fraction gradient. In addition, it is the only scheme for which discrete counterparts of the entropy inequalities satisfied by the exact solutions of the model are proven for all thermodynamically admissible equations of state, under a fully computable CFL condition.

The scheme is well-adapted for subsonic flows (in terms of the relative velocity between the phases) and flows for which the phases are close to the thermodynamical and mechanical equilibrium, a state which is characterized by the equality of pressures, velocities and temperatures of both phases. This numerical method is therefore a natural candidate for simulating the convective part of the complete two-phase flow model, where zero-th order source terms are added to account for the relaxation phenomena that tend to bring the two phases towards thermodynamical, mechanical and chemical equilibria. When these source terms are added, the relaxation scheme can be implemented within a fractional step procedure in order to account for all the physical effects. In such a procedure, the first step is the treatment of the first order part of the Baer-Nunziato model thanks to the relaxation scheme, while the following steps consist in successive ODE solvers where the various relaxation effects are treated separately. To fix ideas, let us recall the general form of the full Baer-Nunziato model with relaxation source terms: 
\begin{equation}
\label{chap2bn_eq1AE}
\begin{array}{lll}
\dv_t \alpha_k + u_2 \dv_x \alpha_k = \Phi_k, \\
\dv_t (\alpha_k \rho_k) + \dv_x (\alpha_k \rho_k u_k) = \Gamma_k,\\
\dv_t (\alpha_k \rho_k u_k) + \dv_x (\alpha_k \rho_k u_k^2 + \alpha_k p_k) - p_1 \dv_x \alpha_k =D_k+ \mathscr{U} \Gamma_k,\\
\dv_t (\alpha_k \rho_k E_k) + \dv_x (\alpha_k \rho_k E_k u_k+ \alpha_k p_k u_k) - u_2 p_1 \dv_x \alpha_k = Q_k + \mathscr{U} D_k + \mathscr{H} \Gamma_k - p_1\Phi_k, 
\end{array}
\end{equation}
where $\mathscr{U}=\frac 12 (u_1+u_2)$ and $\mathscr{H}=\frac 12 u_1 u_2$. The quantities $\Phi_k$, $\Gamma_k$, $D_k$ and $Q_k$ account respectively for the relaxation of pressures, chemical potentials, velocities and temperatures according to:
\begin{equation*}
 \begin{array}{ll}
  \Phi_k  = \Theta_p (p_k-p_{3-k}), & \qquad
  \Gamma_k = \Theta_\mu (\mu_{3-k}-\mu_k),\\
  D_k = \Theta_u (u_{3-k}-u_k), & \qquad
  Q_k = \Theta_T (T_{3-k}-T_k).
 \end{array}
\end{equation*}
We refer to \cite{LDGH} for the precise definition of the chemical potentials $\mu_k$ and that of the positive quantities $\Theta_p$, $\Theta_\mu$, $\Theta_u$ and $\Theta_T$. In \cite{LDGH}, a fractional step method is described for the treatment of these source terms. It is proven that, provided stiffened gas \eos\, for both phases, every ODE-type step of this method is well posed in the sense that existence and uniqueness of the considered quantities are guaranteed in the relevant intervals. Moreover, at the semi-discrete level in time, each one of these steps is compatible with the total entropy inequality satisfied by \eqref{chap2bn_eq1AE}:
\begin{multline*}
 \dv_t \Big( \sum_{k=1,2} \alpha_k \rho_k s_k \Big ) + \dv_x \Big ( \sum_{k=1,2} \alpha_k \rho_k s_k u_k \Big ) \\  \leq -\frac{\Theta_p}{T_2}(p_1-p_2)^2 -\Theta_\mu(\mu_1-\mu_2)^2-\Theta_u \frac{T_1 + T_2}{2 T_1 T_2} (u_1-u_2)^2 - \frac{\Theta_T}{T_1 T_2} (T_1-T_2)^2.
\end{multline*}

The relaxation scheme was specially designed for the simulation of vanishing phase solutions, where in some areas of the flow, the fluid is quasi monophasic \textit{i.e.} one of the phases has nearly disappeared. In particular, the scheme has been proven to robustly handle sharp interfaces between two quasi-monophasic regions as assessed by the results of Test-case 5  (see Section \ref{test_5}). Simulating vanishing phase solutions is a crucial issue for a detailed investigation of incidental configurations in the nuclear industry such as the Departure from Nucleate Boiling (DNB) \cite{dnb}, the Loss
of Coolant Accident (LOCA) \cite{loca} or the Reactivity Initiated Accident (RIA) \cite{ria}. Some numerical methods had already been proposed for the approximation of vanishing phase solutions (\cite{SWK,TT}). The work performed in \cite{CHSS} and in the present paper provides a detailed theoretical and numerical answer to the robustness issues rising up when attempting to simulate vanishing phase solutions.

Despite a relatively complex theory aiming at constructing the underlying approximate Riemann solver, and at analyzing the main properties of the numerical method (positivity, discrete entropy inequalities), the proposed scheme is a rather simple scheme as regards its practical implementation as explained in the appendices \ref{sec_app}. The scheme applies for all equations of state for which the pressure is a given function of the density and of the specific internal energy. In particular, this allows the use of incomplete or tabulated equations of state.

It appears that the relaxation scheme has similar performances as two of the most popular existing schemes available for the Baer-Nunziato model, namely Schwendeman-Wahle-Kapila's first order accurate Godunov-type scheme \cite{SWK} and Toro-Tokareva's finite volume HLLC scheme \cite{TT}. In addition, the scheme compares very favorably with Lax-Friedrichs type schemes that are commonly used in the nuclear industry for their known robustness and simplicity. As a matter of fact, the relaxation finite volume scheme was proved to be much more accurate than Rusanov's scheme for the same level of refinement. In addition, for a prescribed level of accuracy (in terms of the $\xL^1$-error for instance), the computational cost of the relaxation scheme is much lower than that of Rusanov's scheme. Indeed, for some test-cases, reaching the same level of accuracy on some variables may require more than a hundred times more CPU-time to Rusanov's scheme than to the relaxation scheme! In a recent benchmark on numerical methods for two-phase flows \cite{benchmark}, the relaxation scheme was proven to compare very well with various other schemes in terms of CPU-time performances as well as robustness \cite{Dallet}.

Thanks to the invariance of the Baer-Nunziato model under Galilean transformations, the finite volume formulation of the relaxation scheme allows a straightforward extension to 2D and 3D unstructured meshes.
As a matter of fact, the scheme has already been implemented in a proprietary module for 3D two-phase flows developed by the French national electricity company EDF within the framework of the industrial CFD code \textit{Code\_Saturne} \cite{saturne}. The scheme has been successfully applied within nuclear safety studies, for numerical simulations of the primary circuit of pressurized water reactors. A forthcoming paper is in preparation, where the relaxation scheme is used for the simulation of 3D industrial cases.

\clearpage 

\section{Appendices}
\label{sec_app}

\subsection{Construction of the solution to the Riemann problem \eqref{BNrelax_entrop_conv}-\eqref{relax_CI}.}
\label{constr_sol}

Given $(\vect{W}_L,\vect{W}_R,a_1,a_2)$ (satisfying $\T_{k,L}=\tau_{k,L}$ and $\T_{k,R}=\tau_{k,R}$ for $k\in\unde$) such that the conditions of Theorem \ref{relax_thm} are met, we give the expression of the piecewise constant solution of the Riemann problem \eqref{BNrelax_entrop_conv}-\eqref{relax_CI}. For the sake of simplicity, the solution will be expressed in non conservative variables $\tW=(\alpha_1,\tau_1,\tau_2,u_1,u_2,\pi_1,\pi_2,\E_1,\E_2)$.

In practice, when implementing the numerical scheme, the relaxation Riemann solution of \eqref{BNrelax_entrop_conv}-\eqref{relax_CI} is used to compute the numerical fluxes at each interface between two states $(\U_L,\U_R)$ and the relaxation states $(\vect{W}_L,\vect{W}_R)$ are actually computed from these two states $(\U_L,\U_R)$. For this reason, the solution will be denoted
\[
 \xi \longmapsto \tW(\xi;\U_L,\U_R;a_1,a_2).
\]

\bigskip
We recall the following notations built on the initial states $(\vect{W}_L,\vect{W}_R)$ (and therefore depending on $(\U_L,\U_R)$) and on the relaxation parameters $(a_1,a_2)$, which are useful for the computation of the solution.

\medskip
For $k$ in $\lbrace 1,2 \rbrace$:
\begin{equation}
\label{diese1}
\begin{aligned}
 \udd_k(\U_L,\U_R;a_k) &:= \dfrac{1}{2} \left (u_{k,L}+u_{k,R} \right )-\dfrac{1}{2a_k} \left (p_{k,R} - p_{k,L} \right ),
\\
\pidd_k(\U_L,\U_R;a_k)  &:= \dfrac{1}{2}  \left (p_{k,R} + p_{k,L} \right )- \dfrac{a_k}{2} \left (u_{k,R}- u_{k,L} \right ),
\\
\tdd_{k,L}(\U_L,\U_R;a_k)  &:= \tau_{k,L} + \dfrac{1}{a_k}(\udd_k(\U_L,\U_R;a_k)  - u_{k,L}), 
% = \tau_{k,L}+\dfrac{1}{2a_k}(u_{k,R} - u_{k,L}) - \dfrac{1}{2a_k^2}(p_k(\tau_{k,R}) - p_k(\tau_{k,L})),
\\
\tdd_{k,R}(\U_L,\U_R;a_k)  &:= \tau_{k,R} - \dfrac{1}{a_k}(\udd_k(\U_L,\U_R;a_k)  - u_{k,R}).
\end{aligned}
\end{equation}
We also recall the dimensionless number of equation \eqref{chap3notabis}:
\begin{equation*}
\Lambda^\alpha(\U_L,\U_R)  :=\dfrac{\alpha_{2,R}-\alpha_{2,L}}{\alpha_{2,R}+\alpha_{2,L}}, % \in (-1,1).
\end{equation*}
and define as $\Udd(\U_L,\U_R;a_1,a_2)$ the central expression of assumption $\B$ of Theorem \ref{relax_thm}:
\begin{equation*}
\Udd(\U_L,\U_R;a_1,a_2)  := \dfrac{\udd_1(\U_L,\U_R;a_1) -\udd_2(\U_L,\U_R;a_2) -\frac{1}{a_2}\Lambda^\alpha  (\U_L,\U_R) \Big(\pidd_1(\U_L,\U_R;a_1) -\pidd_2(\U_L,\U_R;a_2) \Big )}{1+\frac{a_1}{a_2}|\Lambda^\alpha(\U_L,\U_R)|}. 
\end{equation*} 
% These quantities only depend on the parameters $(a_1,a_2)$ and on the velocities and densities of the initial states (in particular, they are independent of the initial phase fractions $\alpha_{1,L}$ and $\alpha_{1,R}$). 

Later in this section, we will omit the dependency of these quantities on $(\U_L,\U_R;a_1,a_2)$. Following Theorem \ref{relax_thm}, if $a_1$ and $a_2$ are such that $\tdd_{1,L}$, $\tdd_{1,R}$, $\tdd_{2,L}$, $\tdd_{2,R}$ are positive, and if condition $\A$ which reads $-a_1\tdd_{1,R}< \Udd < a_1 \tdd_{1,L}$ holds true, then there exists a self-similar solution to the Riemann problem \eqref{BNrelax_entrop_conv}-\eqref{relax_CI}. Following \cite{CHSS}, we distinguish three different cases corresponding 
to different orderings of the kinematic waves, $u_1^*< u_2^*$, $u_1^*= u_2^*$ or $u_1^*> u_2^*$. With each one of these wave configurations is associated a different expression of assumption $\B$ depending on the sign of $\Udd$.

%\noindent $\bullet$ \textit{\textbf{The solution has the ordering $\bf u_2^* < u_1^*$}} \ 
\subsubsection*{Solution with the wave ordering $\bf u_2^* < u_1^*$:}
\label{constr_12}
The solution $\xi \mapsto \tW(\xi;\U_L,\U_R;a_1,a_2)$ has the wave ordering $ u_2^* < u_1^*$ if the following assumption holds:
\begin{equation*}
 \Bu \qquad 0 < \Udd < a_1 \tdd_{1,L}.\textcolor{white}{-} \
\end{equation*}

\begin{figure}[ht!]
\begin{center}
\includegraphics[width=15cm,height=4cm]{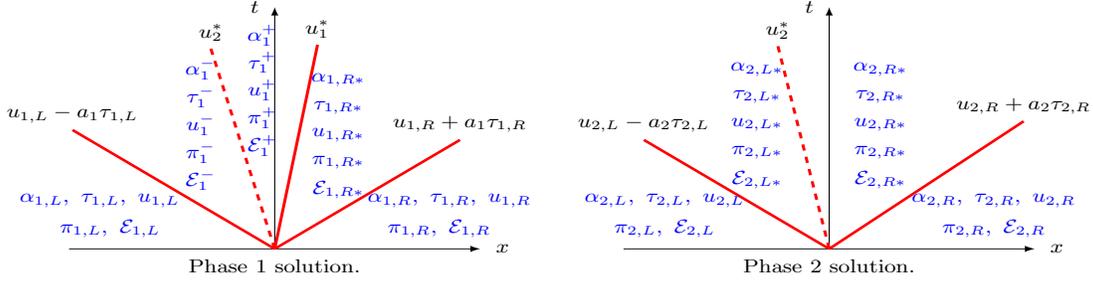}
\protect \parbox[t]{13cm}{\caption{Intermediate states of the exact solution of the Riemann problem \eqref{BNrelax_entrop_conv}-\eqref{relax_CI} with the wave ordering $ u_2^* < u_1^*$ .\label{FigSol1}}}
\end{center}
\end{figure}

The intermediate states, which are represented in Figure \ref{FigSol1}, and the velocities $u_1^*(\U_L,\U_R;a_1,a_2)$ and $u_2^*(\U_L,\U_R;a_1,a_2)$ (simply denoted $u_1^*$ and $u_2^*$ hereafter) are computed through the following steps performed in the very same order. 
\begin{enumerate}
\item Define $\nu := \dfrac{\alpha_{1,L}}{\alpha_{1,R}}$, $\Mdd_L:=\dfrac{\udd_1-\udd_2}{a_1 \tdd_{1,L}}$ and $\Pdd_L:=\dfrac{\pidd_1-\pidd_2}{a_1^2 \tdd_{1,L}}$.
\item Define successively the functions
\begin{align*}
& \M_0(\omega) := \dfrac{1}{2} \left( \dfrac{1+\omega^2}{1-\omega^2} \left( 1+\dfrac{1}{\nu} \right ) - \sqrt{\left(
\dfrac{1+\omega^2}{1-\omega^2}\right )^2 \left( 1+ \dfrac{1}{\nu}
\right )^2 -\dfrac{4}{\nu} }\right ),  \\
& \M_\mu(m) := \dfrac{1}{\nu}\dfrac{m+(1-\mu)\frac{\tdd_{1,R}}{\tdd_{1,L}}}{1-(1-\mu)\frac{\tdd_{1,R}}{\tdd_{1,L}}}, \quad \text{with $\mu \in (0,1)$. For instance $\mu=0.1$,}  \\
&\M(m) :=\min \left ( \M_0 \left (\frac{1- m}{1+m} \right ), \M_\mu(m) \right ),  \\
& \Psi (m) := m+ \dfrac{a_1}{a_2}\dfrac{\alpha_{1,R}}{\alpha_{2,L}+\alpha_{2,R}}\left ((1+\nu) m -2\nu \M \left (m\right) \right ).
\end{align*}

\item  
Use an iterative method (\textit{e.g.} Newton's method or a dichotomy (bisection) method) to compute $\Me_L \in (0,1)$ such that 
\begin{equation}
\label{chap3mel}
\Psi (\Me_L)=
\Mdd_L -\dfrac{a_1}{a_2}\Lambda^\alpha\Pdd_L. 
\end{equation}
According to \cite{CHSS}, $\Me_L$ always exists under $\Bu$ and is unique if $\mu$ is close enough to one. In practice, the iterative method is initialized at $m^0=\max(0,\min(\Mdd_L,1))$.

\item The velocity $u_2^*$ is obtained by $ u_2^* = \udd_1-a_1\tdd_{1,L}\Me_L$.

% \item In order to compute the quantities relative to phase 1, define successively the two functions
% \begin{eqnarray}
% && \M_{1,R*} (m) := \dfrac{\nu \M_0 \left (\frac{1-m}{1+m} \right  )(1+m)}{\frac{\tdd_{1,R}}{\tdd_{1,L}} \left (1+\nu \M_0 \left (\frac{1-m}{1+m} \right  ) \right )+m - \nu \M_0 \left (\frac{1-m}{1+m} \right  )},
% \label{chap3mach3}\\
% && \nonumber \\
% && \nonumber \\
% && \M(m) := 
% \left \lbrace
% 		\begin{array}{ll}
% 			  \M_0 \left (\frac{1-m}{1+m}\right), & \ \ \text{if} \ \left |\M_{1,R*} (m) \right | < 1, \\
% 			   (1-\theta) \dfrac{1}{\nu} \dfrac{m + \frac{\tdd_{1,R}}{\tdd_{1,L}}}{m+2-\frac{\tdd_{1,R}}{\tdd_{1,L}}}, & \ \ \text{otherwise},
% 		 \end{array}
% \right. \label{chap3mdem}
% \end{eqnarray}
% where $\theta \in (0,1)$ is a small parameter. In practice we take $\theta = 0.1$.
\item The velocity $u_1^*$ is obtained by $u_1^* =  u_2^* + \nu a_1  \tdd_{1,L} \M(\Me_L) \dfrac{ 1 - \Me_L}{1 -\M(\Me_L)}$.

\item  The intermediate states for phase 1 are given by
\begin{itemize}
 \item Phase fractions: $\alpha_1^-=\alpha_{1,L}$, $\alpha_1^+=\alpha_{1,R*}=\alpha_{1,R}$.
 \item Specific volumes:
 $$
 \tau_1^- = \tdd_{1,L} \dfrac{ 1 - \Me_L}{1 -\M(\Me_L)}, \quad \tau_1^+ = \tdd_{1,L} \dfrac{ 1 + \Me_L }{1 + \nu \M(\Me_L)}, \quad \tau_{1,R*} = \tdd_{1,R}+ \tdd_{1,L} \dfrac{\Me_L - \nu \M(\Me_L)}{1 + \nu \M(\Me_L)}.
 $$
 \item Velocities:
 $$
 u_1^- = u_2^* + a_1 \tdd_{1,L} \M(\Me_L) \dfrac{ 1 - \Me_L}{1 -\M(\Me_L)}, \quad u_1^+ = u_{1,R*} =  u_1^*.
 $$
 \item Relaxation pressures $\pi_1(\tau_1,\T_1,s_1)$:
 $$
 \pi_1^- = p_{1,L}+a_1^2(\tau_{1,L}-\tau_1^-),\quad \pi_1^+ = p_{1,L}+a_1^2(\tau_{1,L}-\tau_1^+), \quad \pi_{1,R*}=p_{1,R}+a_1^2(\tau_{1,R}-\tau_{1,R*} ).
 $$
%  \item Relaxation internal energies $e_1(\T_1,s_1)$:
%  $
%  e_{1}^-=e_1^+=e_{1,L}, \ e_{1,R*} = e_{1,R}.
%  $
\item Relaxation total energies $\E_1(u_1,\tau_1,\T_1,s_1)$:
$$
\begin{aligned}
&\E_1^-=(u_1^-)^2/2+e_{1,L}+((\pi_1^-)^2-p_{1,L}^2)/(2a_1^2), \\
&\E_1^+=(u_1^+)^2/2+e_{1,L}+((\pi_1^+)^2-p_{1,L}^2)/(2a_1^2), \\
&\E_{1,R*}=(u_{1,R*})^2/2+e_{1,R}+(\pi_{1,R*}^2-p_{1,R}^2)/(2a_1^2). 
\end{aligned}
$$.
 \end{itemize}

\item  The intermediate states for phase 2 are then given by 
\begin{itemize}
 \item Specific volumes:
 $
 \tau_{2,L*} = \tau_{2,L} +\dfrac{1}{a_2}(u_2^*-u_{2,L}), \qquad \tau_{2,R*} = \tau_{2,R} -\dfrac{1}{a_2}(u_2^*-u_{2,R}).
 $
 \item Velocities:
 $
 u_{2,L*}=u_{2,R*}=u_2^*.
 $
 \item Relaxation pressures $\pi_2(\tau_2,\T_2,s_2)$:
 $$
 \pi_{2,L*}=p_{2,R}+a_2^2(\tau_{2,L}-\tau_{2,L*} ), \qquad \pi_{2,R*}=p_{2,R}+a_2^2(\tau_{2,R}-\tau_{2,R*} ).
 $$
%  \item Relaxation internal energies $e_2(\T_2,s_2)$:
%  $
%  e_{2,L*}=e_{2,L}, \ e_{2,R*} = e_{2,R}.
%  $
 \item Relaxation total energies $\E_2(u_2,\tau_2,\T_2,s_2)$:
$$
\begin{aligned}
&\E_{2,L*}=(u_{2}^*)^2/2+e_{2,L}+(\pi_{2,L*}^2-p_{2,L}^2)/(2a_2^2),\\
&\E_{2,R*}=(u_{2}^*)^2/2+e_{2,R}+(\pi_{2,R*}^2-p_{2,R}^2)/(2a_2^2).
\end{aligned}
$$.
\end{itemize}
\end{enumerate}

% \smallskip
% 
% \begin{rem}
% In \cite{CHSS}, a kinetic relation is designed in order to define the Mach number $\M$ that parametrizes the phase 1 solution. It consists in imposing the lower-bound $\mu \tdd_{1,R}$ on the specific volume $\tau_{1,R*}$. If $\M(\Me_L) =\M_0 \left (\frac{1-\Me_L}{1+\Me_L}\right)$ is such that this lower-bound is satisfied, then the chosen solution is the unique   energy-preserving solution. Otherwise, maintaining the lower-bound $\mu \tdd_{1,R}$ for  $\tau_{1,R*}$ requires an energy dissipation which is ensured by taking $\M(\Me_L) =\M_\mu(\Me_L)$ (see \cite{CHSS} for more details).
% \end{rem}

\subsubsection*{Solution with the wave ordering $\bf u_2^* > u_1^*$:}
The solution $\xi \mapsto \tW(\xi;\U_L,\U_R;a_1,a_2)$ has the wave ordering $u_2^* > u_1^*$ if the following assumption holds:
\begin{equation*}
\Bd \qquad -a_1\tdd_{1,R} < \Udd < 0.
\end{equation*}
% \begin{center}
% \begin{tikzpicture}[scale=2]
% \small
% \tikzstyle{axes}=[thin,>=latex]
% \begin{scope}[axes]
% \draw (-1.5,0)   node {
% 	\begin{tikzpicture}[scale=2]
% \small
% \tikzstyle{axes}=[thin,>=latex]
% \begin{scope}[axes]
%         \draw[->] (-1,0)--(1,0) node[right=3pt] {$x$};
%         \draw[->] (0,0)--(0,1.25) node[left=3pt] {$t$};
% 	\draw [very thick,color=red] (0,0) -- (35:1cm) node[black,above]{$u_{1,R}+a_1\tau_{1,R}$};
%         \draw [very thick,color=red,dashed] (0,0) -- (70:1cm) node[black,above]{$u_2^*$};
% 	\draw [very thick,color=red] (0,0) -- (85:1cm) node[black,above]{$u_1^*$};
% 	\draw [very thick,color=red] (0,0) -- (145:1cm) node[black,above]{$u_{1,L}-a_1\tau_{1,L}$};
% %
% 	
% %
% \end{scope}
% \normalsize
% \end{tikzpicture}
% };
% \draw (1.7,0)   node {
% 	\begin{tikzpicture}[scale=2]
% \small
% \tikzstyle{axes}=[thin,>=latex]
% \begin{scope}[axes]
%         \draw[->] (-1,0)--(1,0) node[right=3pt] {$x$};
%         \draw[->] (0,0)--(0,1.25) node[left=3pt] {$t$};
% 	\draw [very thick,color=red] (0,0) -- (38:1cm) node[black,above]{$u_{2,R}+a_2\tau_{2,R}$};
% 	\draw [very thick,color=red,dashed] (0,0) -- (70:1cm) node[black,above]{$u_2^*$};
% 	\draw [very thick,color=red] (0,0) -- (145:1cm) node[black,above]{$u_{2,L}-a_2\tau_{2,L}$};
% %
% 	
% \end{scope}
% \normalsize
% \end{tikzpicture}
% };
% \end{scope}
% \normalsize
% \end{tikzpicture}
% \end{center}
For the determination of the wave velocities and the intermediate states, the simplest thing to do is to exploit the Galilean invariance of the equations. In this case indeed, the solution is obtained by the transformation
\begin{equation}
 \tW(\xi;\U_L,\U_R;a_1,a_2) := \mcal{V}\tW(-\xi;\mcal{V}\U_R,\mcal{V}\U_L;a_1,a_2),
\end{equation}
where the operator $\mcal{V}$ changes the velocities into their opposite values:
\begin{equation}
\mcal{V}: (x_1,x_2,x_3,x_4,x_5,x_6,x_7,x_8,x_9) \mapsto (x_1,x_2,x_3,-x_4,-x_5,x_6,x_7,x_8,x_9).
\end{equation}
Of course, the function $\xi\mapsto \tW(\xi;\mcal{V}\U_R,\mcal{V}\U_L;a_1,a_2)$ is computed through the first case, since for these new initial data $(\mcal{V}\U_R,\mcal{V}\U_L)$, it is condition $\Bu$ that holds.\\

\subsubsection*{Solution with the wave ordering $\bf u_2^* = u_1^*$:}
The solution $\xi \mapsto \tW(\xi;\U_L,\U_R;a_1,a_2)$ has the wave ordering $u_2^* = u_1^*$ if the following assumption holds:
\begin{equation*}
\Bt \qquad  \Udd = 0. \qquad \qquad \quad
\end{equation*}

\begin{figure}[ht!]
\begin{center}
\includegraphics[width=15cm,height=4cm]{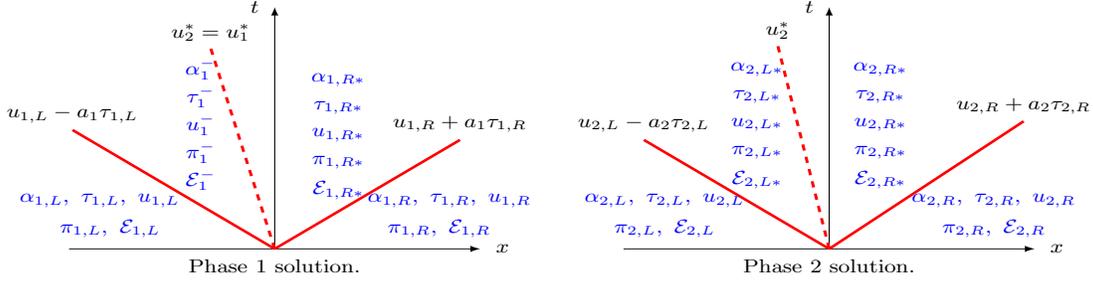}
\protect \parbox[t]{13cm}{\caption{Intermediate states of the exact solution of the Riemann problem \eqref{BNrelax_entrop_conv}-\eqref{relax_CI} with the wave ordering $ u_2^* = u_1^*$ .\label{FigSol2}}}
\end{center}
\end{figure}

The kinematic velocities are given by $ u_2^* = u_1^* = \udd_1$.
The intermediate states for phase 2 are obtained by the same formulae as in the case $u_2^* < u_1^*$, while the intermediate states for phase 1 (see Figure \ref{FigSol2}) read
\[
\begin{aligned}
&\alpha_1^- = \alpha_{1,L}, 					&\alpha&_{1,R*} = \alpha_{1,R},\\
&\tau_1^- = \tdd_{1,L} , 					&\tau&_{1,R*} = \tdd_{1,R},\\
&u_1^- = \udd_1,  						&u&_{1,R*} = \udd_1, \\
&\pi_1^- = \ p_{1,L}+a_1^2(\tau_{1,L}-\tdd_{1,L}), 		&\pi&_{1,R*} = \ p_{1,R}+a_1^2(\tau_{1,R}-\tdd_{1,R*}),\\
&\E_1^- = (u_1^-)^2/2+e_{1,L}+((\pi_1^-)^2-p_{1,L}^2)/(2a_1^2), &\E&_{1,R*}=(u_{1,R*})^2/2+e_{1,R}+(\pi_{1,R*}^2-p_{1,R}^2)/(2a_1^2). 
\end{aligned}
\]

\subsubsection*{The non-conservative product $\bf\textbf{d}(\vect{W})\dv_x \vect{W}$:} 
\label{chap3nonconv}
When $\alpha_{1,L} \neq \alpha_{1,R}$, the non-conservative product $\textbf{d}(\vect{W})\dv_x \vect{W}$ identifies with a Dirac measure propagating at the constant velocity $u_2^*$. This Dirac measure is given by
\begin{equation*}
 \textbf{D}^*(\vect{W}_L,\vect{W}_R)\delta_{x-u_2^*t} ,
\end{equation*}
where 
$\textbf{D}^*(\vect{W}_L,\vect{W}_R):=(\alpha_{1,R}-\alpha_{1,L})\left ( u_2^*, 0,0, - \pi_1^*, + \pi_1^*, 0,0,0,0 \right )^T$. The pressure $\pi_1^*$ is defined for $\alpha_{1,R} \neq \alpha_{1,L}$ by
\begin{equation*}
 \pi_1^* :=  \pidd_2 - a_2\frac{\alpha_{2,R}+\alpha_{2,L}}{\alpha_{1,R}-\alpha_{1,L}}(u_2^*-\udd_2).
\end{equation*}

\subsection{Practical implementation of the relaxation finite volume scheme}
\label{choixa1a2}
In this appendix, we describe in detail the practical implementation of the scheme. We recall the space and time discretization:  we assume a positive space step  $\Delta x$ and the time step $\Delta t$ is dynamically updated through the CFL condition. The space is partitioned into cells $\R=\bigcup_{j\in \Z} [x_{j-\frac{1}{2}},x_{j+\frac{1}{2}}[$ with $x_{j+\frac{1}{2}}=(j+\frac 12) \Delta x$ for all $j$ in $\Z$. The centers of the cells are denoted $x_{j}=j \Delta x$ for all $j$ in $\Z$. We also introduce the discrete intermediate times $t^{n}=n\Delta t, \ n\in \xN$.

\medskip

The solution of the Cauchy problem:
\begin{equation*}
\left \lbrace
          \begin{array}{ll}
		\dv_t \U + \dv_x {\bf \Fcal}(\U) + {\bf \Ccal}(\U)\dv_x \U = 0, & x \in \R, t>0, \\
		 \U(x,0) =\U_0(x), &  x \in \R,
	   \end{array}
\right.
\end{equation*}
is approximated at time $t^n$ by $\U_j^n$ on the cell $[x_{j-\frac{1}{2}},x_{j+\frac{1}{2}}[$. The values of the approximate solution are inductively computed as follows:

\medskip
\noindent \textit{Initialization}:
\[
 \U_j^0 = \frac{1}{\Delta x} \int_{x_{j-\frac 12}}^{x_{j+\frac 12}} \U_0(x)\dx.
\]

\medskip
\noindent \textit{Time evolution}:
\begin{equation}
\label{FVS}
\U_{j}^{n+1} = \U_{j}^{n} - \dfrac{\Delta t}{\Delta x}
\left (\mathcal{F}^{-}(\U_{j}^{n},\U_{j+1}^{n}) -
\mathcal{F}^{+}(\U_{j-1}^{n},\U_{j}^{n})  \right).
\end{equation}

\medskip
At each cell interface $x_{j+\frac 12}$, the numerical fluxes $\mathcal{F}^{\pm}(\U_{j}^{n},\U_{j+1}^{n})$ are computed thanks to the relaxation approximate Riemann solver. They depend on the states $(\U_{j}^{n},\U_{j+1}^{n})$ but also on the local values of the relaxation parameters $a_{k,j+\frac 12}^n,\,k=1,2$. Denoting $\U_L=\U_{j}^n$ and $\U_R=\U_{j+1}^n$ and $a_k,\,k=1,2$ for simplicity, the fluxes  $\mathcal{F}^{\pm}(\U_L,\U_R)$ are computed through the following steps.

\begin{enumerate}
 \item  \textit{Local choice of the pair $(a_1,a_2)$.}
 The pair of parameters $(a_1,a_2)$, must be chosen large enough so as to satisfy several requirements:
\begin{itemize}
 \item In order to ensure the stability of the relaxation approximation, $a_k$ must satisfy Whitham's condition (\ref{whithambis}). For simplicity however, we do not impose Whitham's condition everywhere in the solution of the Riemann problem (\ref{BNrelax_entrop_conv})-(\ref{relax_CI}) (which is possible however), but only for the left and right initial data at each interface:
 \begin{equation}
  \text{for $k$ in $\lbrace 1,2 \rbrace$}, \quad a_k > \max \left (\rho_{k,L}\,c_{k}(\rho_{k,L},e_{k,L}),\rho_{k,R}\,c_{k}(\rho_{k,L},e_{k,L}) \right ),
 \end{equation}
 where $c_k(\rho_k,e_k)$ is the speed of sound in phase $k$. In practice, no instabilities were observed during the numerical simulations due to this simpler Whitham-like condition.
 
\item In order to compute the solution of the relaxation Riemann problem, the specific volumes $\tdd_{k,L}(\U_L,\U_R;a_k)$ and $\tdd_{k,R}(\U_L,\U_R;a_k)$ defined in \eqref{diese1} must be positive. The expressions of $\tdd_{k,L}(\U_L,\U_R;a_k)$ and $\tdd_{k,R}(\U_L,\U_R;a_k)$
% \begin{eqnarray*}
% \tdd_{k,L}(\U_L,\U_R;a_k) &=&  \tau_{k,L} +\dfrac{1}{2a_k}(u_{k,R} - u_{k,L}) - \dfrac{1}{2a_k^2} \left
% (p_{k,R}-p_{k,L} \right ), \\
% \tdd_{k,R}(\U_L,\U_R;a_k) &=& \tau_{k,R} +\dfrac{1}{2a_k}(u_{k,R} - u_{k,L}) + \dfrac{1}{2a_k^2} \left (
% p_{k,R}-p_{k,L} \right ). 
% \end{eqnarray*}
are two second order polynomials in $a_k^{-1}$ whose constant terms are respectively $\tau_{k,L}$ and $\tau_{k,R}$. Hence, by taking $a_k$ large enough, one can guarantee that $\tdd_{k,L} (\U_L,\U_R;a_k)>0$ and $\tdd_{k,R}(\U_L,\U_R;a_k) >0$, since the initial specific volumes $\tau_{k,L}$ and $\tau_{k,R}$ are positive. 

\item Finally, in order for the relaxation Riemann problem  (\ref{BNrelax_entrop_conv})-(\ref{relax_CI}) to have a positive solution, $(a_1,a_2)$ must be chosen so as to meet condition $\B$ of Theorem \ref{relax_thm} as well as the positivity condition of the phase 2 densities $\C$ (see the comments after Theorem \ref{relax_thm}).
\end{itemize}
Thereafter, we propose an iterative algorithm for the computation of the parameters $(a_1,a_2)$ at each interface.
%\texttt{Fixedpoint}$(a_{1},a_2)$ is a subroutine that computes a numerical approximation of the solution $u_2^*$ of the fixed-point problem (\ref{chap3mel}), using some numerical method such as Newton's method or a dichotomy algorithm. 
The notation $\texttt{not} (\textbf{P})$ is the negation of the logical statement $\textbf{P}$.

\medskip
\hspace{1.5cm}
\vrule
\begin{minipage}{20cm}
 \begin{itemize} \itemsep2pt \parskip0pt \parsep0pt
\item Choose $\eta$ a (small) parameter in the interval $(0,1)$.
 \item For $k$ in $\lbrace 1,2 \rbrace$ initialize $a_{k}$:
\begin{description}
\item \qquad $a_{k} := (1+\eta) \max \left( \rho_{k,L}\, c_{k}(\rho_{k,L},e_{k,L}), \rho_{k,R}\, c_{k}(\rho_{k,R},e_{k,R})  \right )$.
\end{description}
\item For $k$ in $\lbrace 1,2 \rbrace$:\\
do $\lbrace a_{k} := (1+\eta)a_{k} \rbrace$  while $ \big ( \tdd_{k,L}(\U_L,\U_R;a_k) \leq 0$ or $\tdd_{k,R}(\U_L,\U_R;a_k) \leq 0 \big )$.
\item do $\lbrace$  $\, a_{2} := (1+\eta)a_{2},$
\begin{description}
 \item  \qquad do $\lbrace a_{1} := (1+\eta)a_{1} \rbrace $ while \big (not$\textbf{\B}$\big),
 \item  \qquad compute the value of $u_2^*$ in the solution $\tW(\U_L,\U_R;a_1,a_2)$,
\end{description}
\textcolor{white}{do} $\rbrace$ while \big (not$\C$\big).
\end{itemize}
\end{minipage}
\\

\medskip
In this algorithm, the computation of $u_2^*$ requires the computation of the solution of the fixed-point problem (\ref{chap3mel}), using some numerical method such as Newton's method or a dichotomy (bisection) algorithm. It is possible to prove that this algorithm always converges in the sense that there is no infinite looping due to the while-conditions. Indeed, it is easy to observe that assumptions $\B$ and $\C$ are always satisfied if the parameters $(a_1,a_2)$ are taken large enough. Moreover, this algorithm provides reasonable values of $a_1$ and $a_2$, since in all the numerical simulations, the time step obtained through the CFL condition (\ref{chap2cfl}) remains reasonably large and does not go to zero. In fact, the obtained values of $a_1$ and $a_2$ are quite satisfying since the relaxation scheme compares very favorably with Rusanov's scheme, in terms of CPU-time performances (see Section \ref{numtest}).

\item \textit{Calculation of the numerical fluxes.} Once the relaxation parameters are known, one may give the expressions of the numerical fluxes $\mathcal{F}^{\pm}(\U_L,\U_R)$. Observe that, as a by-product of the above algorithm for the computation of $(a_1,a_2)$, the propagation velocity $u_2^*$ is already known, and one does not need to redo the fixed-point procedure. Given the solution $\xi \mapsto \tW(\xi;\U_L,\U_R;a_1,a_2)$ of the relaxation Riemann problem (\ref{BNrelax_entrop_conv})-(\ref{relax_CI}) (see appendix \ref{constr_sol} for the expression of the intermediate states), which we denote $\tW(\xi)$ for the sake of simplicity, the numerical fluxes are computed as follows:
\[
 \mathcal{F}^{\pm}(\U_L,\U_R)= 
 \left [
 \begin{matrix}
 0 \\[2ex]
 (\alpha_1 \rho_1 u_1)\big ( \tW(0^{\pm} )\big) \\[2ex]
 (\alpha_2 \rho_2 u_2) \big ( \tW(0^{\pm} )\big)\\[2ex]
 (\alpha_1 \rho_1 u_1^2 + \alpha_1 \pi_1) \big ( \tW(0^{\pm} )\big)\\[2ex]
 (\alpha_2 \rho_2 u_2^2 + \alpha_2 \pi_2) \big ( \tW(0^{\pm} )\big)\\[2ex]
 (\alpha_1 \rho_1 \E_1 u_1 + \alpha_1 \pi_1 u_1) \big ( \tW(0^{+} )\big)\\[2ex]
 (\alpha_2 \rho_2 \E_2 u_2 + \alpha_2 \pi_2 u_2) \big ( \tW(0^{+} )\big)
 \end{matrix}
 \right ]
 +
 \left [
 \begin{matrix}
  (u_2^*)^{\pm}\\[2ex]
  0\\[2ex]
  0\\[2ex]
  -\dfrac{(u_2^*)^\pm}{u_2^*}\pi_1^* \\[2ex]
  \dfrac{(u_2^*)^\pm}{u_2^*}\pi_1^* \\[2ex]
  -(u_2^*)^{\pm}\pi_1^* \\[2ex]
  (u_2^*)^{\pm}\pi_1^*
 \end{matrix}
 \right ] (\alpha_{1,R}-\alpha_{1,L}),
\]
where $u_2^*$ is already known as a result of the first step (choice of the pair $(a_1,a_2)$) and the expression of $\pi_1^*$ is given at the end of appendix \ref{constr_sol}. In the above expression of the numerical fluxes, we have denoted $(u_2^*)^+=\max(u_2^*,0)$, $(u_2^*)^-=\min(u_2^*,0)$ and the functions $x\mapsto \frac{(x)^\pm}{x}$ are extended by $0$ at $x=0$.

\end{enumerate}

Finally, the time step is computed so as to satisfy the CFL condition:
\[
\frac{\Delta t}{\Delta x} \ \underset{k\in\lbrace 1,2\rbrace, j \in \Z} {\max} \max \left \lbrace |
(u_{k}-a_{k}\tau_{k})^n_j|,|
(u_{k}+a_{k}\tau_{k})^n_{j+1}|\right \rbrace < \frac{1}{2},
\]
and the scheme \eqref{FVS} can be now applied to update the values of the unknown $\U_j^{n+1}$ for $j\in\Z$.

%\subsection{Practical choice of the pair $(a_1,a_2)$}
%

%\newpage

\bigskip
\textbf{Acknowledgements.} The authors would like to thank Nicolas Seguin who is a co-author of \cite{CHSS} upon which the present paper is based. The authors would like to thank him for his thorough reading of this paper which has constantly led to improving the text. 
The authors are grateful to Eleuterio Toro and Svetlana Tokareva who kindly agreed to pass the test-cases of the present paper with their HLLC scheme and to share the results. The authors would also like to warmly thank Donald Schwendeman and Michael Hennessey who also agreed to pass the test-cases with the Schwendeman-Wahle-Kapila's Godunov-type scheme and to share the results. In both cases, their remarkable reactivity has contributed to improving this paper. Finally, the authors would like to thank the reviewers for their constructive remarks.
This work has been partially funded by ANRT and EDF through an EDF-CIFRE contract 529/2009.

\small
\bibliographystyle{plain}
\bibliography{bibfile}
\normalsize

\end{document}